\documentclass[12pt]{amsart}

\usepackage{tikz} 
\usepackage{amsmath}
\usepackage{amssymb,amsfonts,mathrsfs}
\usepackage[colorlinks=true,linkcolor=blue,citecolor=blue]{hyperref}
\usepackage[driver=pdftex,margin=3cm,heightrounded=true,centering]{geometry}

\usepackage{euler}

\usepackage{mllocal} 
\newcommand{\cH}{\mathscr H}

\numberwithin{equation}{section}
\begin{document}

\title[Cheeger-M\"uller Theorem on manifolds with cusps]
{Cheeger-M\"uller Theorem on manifolds with cusps}

\date{\today}

\author{Boris Vertman} 
\address{Institute of Mathematics and Computer Science, University of M\"unster, } 
\address{Einsteinstrasse 62, 48149 M\"unster, Germany} 
\email{vertman@uni-muenster.de}
\urladdr{http://wwwmath.uni-muenster.de/42/arbeitsgruppen/ag-differentialgeometrie/}


\begin{abstract} {We prove equality between the renormalized
Ray-Singer analytic torsion and the intersection R-torsion
on a Witt-manifold with cusps, up to an error term
determined explicitly by the Betti numbers of the cross section of the
cusp and the intersection R-torsion of a model cone. 
In the first step of the proof we compute explicitly the
renormalized Ray-Singer analytic torsion of a model cusp in general
dimension and without the Witt-condition. In the second step we
establish a gluing formula for renormalized Ray-Singer analytic
torsion on a general class of non-compact manifolds in any dimension
that includes Witt-manifolds with cusps, but also scattering
manifolds with asymptotically conical ends. 
In the final step, a Cheeger-M\"uller theorem on cusps 
follows by a combination of the previous explicit computation
and the gluing formula.}
\end{abstract}

\maketitle 
\tableofcontents

\section{Introduction} 

One of the fundamental achievements in modern spectral geometry is the
proof by Cheeger and M\"uller of the Ray-Singer conjecture, which
asserts equality between the analytic and Reidemeister torsions of a
compact smooth odd-dimensional manifold equipped with a flat Hermitian
vector bundle, associated to a unitary representation of the
fundamental group. Since the analytic torsion is defined in terms of
the spectrum of the Hodge Laplace operator, and the Reidemeister
torsion is a purely combinatorial invariant, their equality (along
with the Atiyah-Singer index theorem) has many crucial applications in
fields including topology, number theory and mathematical physics,
notably the Chern-Simons perturbation theory. \medskip

The Reidemeister torsion invariant for manifolds which are not simply
connected was introduced by Reidemeister in \cite{Re1, Re2} and
extended to higher dimensions by Franz in \cite{Fr}, as a tool for a
full PL-classification of lens spaces. The Reidemeister torsion
provided the first example of a topological invariant that
distinguished homotopic but not homeomorphic spaces. The definition of
Reidemeister and Franz was extended later to smooth manifolds by
Whitehead \cite{Wh} and de Rham \cite{Rh}, who proved that a spherical
Clifford-Klein manifold is determined up to isometry by its
fundamental group and its Reidemeister torsion. \medskip

Analytic torsion was introduced by Ray and Singer in their influential
paper \cite{RS} as an analytic counterpart to the Reidemeister
torsion, and has been equated to the Reidemeister torsion in the
setting of lens spaces. This observation has led Ray and Singer to
conjecture equality of analytic and Reidemeister torsions on general
closed odd-dimensional manifolds, which was proved in the celebrated
theorem by Cheeger \cite{Che} and M\"uller \cite{Mue-AT}. Independent 
proofs of the Ray-Singer conjecture have been obtained by Bismut and 
Zhang using the Witten deformation \cite{BZ}, by Vishik using 
the gluing principle \cite{Vi} and Hassel using analytic surgery \cite{Ha}, 
to name a few. \medskip

The Cheeger-M\"uller theorem extends to compact manifolds with boundary under product
type assumptions on the metric structures, cf. L\"uck \cite{Lu} and
Vishik \cite{Vi}. In that case, both the analytic
Ray-Singer and the combinatorial Reidemeister torsions are equal up to
an error term determined explicitly by the Euler characteristic of the
boundary. Dependence of analytic torsion on the metric structures near
the boundary has been studied by Br\"uning-Ma \cite{BM} and Dai-Fang
\cite{DF}. \medskip
 
Establishing a Cheeger-M\"uller type theorem outside the setting of
compact smooth manifolds has proven a tedious task with various
incremental steps being taken in this direction. We mention here
partial results obtained in the setting of spaces with isolated
conical and edge singularities, hyperbolic as well as scattering
spaces. \medskip

On manifolds with isolated conical singularities, the cut and paste
property of analytic torsion, established by Lesch \cite{Les}, reduces
the analysis to the discussion of a truncated cone. Explicit though
intricate formulae for analytic torsion of a truncated cone have been
derived using the double summation method of Spreafico \cite{Spr:ZIF,
Spr:ZFA}, by Melo-Hartmann-Spreafico \cite{MHS} and Vertman
\cite{Ver}. Further understanding the various terms in the explicit
formula for analytic torsion has been obtained by M\"uller-Vertman
\cite{MV} and Hartmann-Spreafico \cite{HS}. \medskip

On the combinatorial side, Dar \cite{Dar} has introduced the
intersection R-torsion for stratified spaces, computed recently by
Dai-Huang \cite{DH} in context of truncated cones. The construction 
is based on the intersection cohomology theory by Goresky and MacPherson 
\cite{GM1, GM1}. The intersection R-torsion of Dar is defined a priori only for
flat vector bundles over the stratified space. However, Albin, Rochon and Sher 
\cite[\S 8]{ARS} have extended the combinatorial definition to a class
of flat vector bundles defined over the smooth stratum only. \medskip

In view of the gluing formula for analytic torsion by Lesch \cite{Les}, one seeks to
establish a Cheeger-M\"uller type result by comparing the intersection
R-torsion with the analytic Ray-Singer torsion for model cones, an
ansatz which has not yet been successful due to highly non-trivial
spectral contributions on the analytic side. Nevertheless,
conjecturing a topological interpretation for the spectral analytic
torsion invariant seems reasonable for a class of singular spaces by a
recent observation of metric independence for manifolds with edges by
the author jointly with Mazzeo in \cite{MazVer}. \medskip

In the setting of non-compact hyperbolic spaces, the original
definition of analytic torsion does not make sense due to the
continuous spectrum of the Hodge Laplacian. Still, a renormalized
version of analytic torsion exists and the intricate algebraic
structure of the hyperbolic space, equipped with a flat Hermitian
vector bundle that corresponds to a canonical non-unitary unimodular
representation of the fundamental group, allows for a deep analysis of
the relation between the renormalized analytic and Reidemeister
torsions by Pfaff \cite{Pf} and M\"uller-Pfaff \cite{MPa, MPb}.
Renormalized analytic torsion has also been discussed in the setting
of non-compact asymptotically conical (scattering) manifolds by
Guillarmou and Sher in \cite{Sher}, though its relation to the
intersection Reidemeister torsion is still an open question. 
\medskip

Our discussion is organized as follows. \S \ref{pre}
is devoted to setting the notation, introducing the fundamental concepts
and stating the main results. In \S \ref{bessel-section} we gather all the relevant facts 
on Bessel functions. We apply these facts to define zeta-regularized determinants of 
certain cusp operators in \S \ref{section-det}. We establish a variational formula 
for that determinant in \S \ref{variation-section}. We then turn to the structure of
the de Rham complex of a model cusp and decompose the complex in \S \ref{decomposition-section}
into suitable subcomplexes. In \S \ref{int-sec} we establish an integral representation of a
zeta function associated to certain subcomplexes and proceed in \S \ref{spreafico-section}
with its analytic continuation to zero using Spreafico's double summation method. 
In \S \ref{thm-sec} we compute
the renormalized analytic torsion of the model cusp explicitly in terms
of the Euler characteristic and Betti numbers of the cross section.
The gluing formula for the renormalized analytic torsion on certain 
classes of complete non-compact manifolds is proved in \S \ref{gluing-section}.
These results allow us to deduce a Cheeger-M\"uller type result for Witt-manifolds
with cusps in \S \ref{CM-section}.

\begin{remark}
Parallel to the present announcement, Albin, Rochon and Sher 
\cite{ARS} have uploaded a very interesting discussion of renormalized 
torsion on a general class of manifolds with fibered cusps. 
Their results use a completely different
ansatz and methodology, and initially require "strong acyclicity" 
assumptions, which we do not pose in the present discussion. 
The strong acyclicity assumption was relaxed to the Witt condition in the special 
case of (non-fibered) cusps in their subsequent paper \cite{ARS2},
where the authors obtained our result by a method of degeneration. 
\end{remark}

\begin{remark}
Combination of the gluing formula in \S \ref{gluing-section}
and Guillarmou-Sher \cite{Sher} allows to compute 
analytic torsion of a truncated cone in terms of the 
renormalized analytic torsion of the infinite cone. This 
leads to potentially new computational results using 
Lesch \cite[Chapter II]{Lesch-habil}.
\end{remark}

\section{Preliminaries and statement of the main results}\label{pre}

In this section we outline some fundamental facts on manifolds
with cusps and state our main results. 

\subsection{Riemannian manifolds with cusps} \label{intro-cusps}

The present work deals with non-compact Riemannian manifolds with cusps, equipped with a flat
Hermitian vector bundle that corresponds to a unitary representation.
More precisely, let $(M,g)$ be an oriented complete Riemannian
manifold of odd dimension $\dim M=m$, where $M=K\cup_{N} \U$ is a
union of a compact manifold $K$ with boundary $\partial K = N\sqcup
N'$ comprised of two boundary components, and $\U=N\times [1,\infty)$
is a non-compact end glued to $K$ along $N=N\times \{1\}$. We assume
\begin{align*} 
g\restriction \U = \frac{dx^2+g^N}{x^2}, \ x\in [1,\infty), 
\end{align*} 
where $g^N$ is a Riemannian metric on the
closed manifold $N$ of dimension $\dim N=n$.\medskip
 
Fix a base point $q\in M$ and consider a unitary representation of the
fundamental group $\rho: \pi_1(M,q) \to U(r,\C)$. The corresponding
vector bundle $E$ is equipped with the canonical Hermitian metric $h$,
induced by the standard Hermitian inner product on $\C^r$, and the
canonical flat covariant derivative $\nabla$, induced by the exterior
derivative on the universal cover $\widetilde{M}$. Denote by
$\Omega^k_0(M,E)$ the space of $E$-valued differential forms of degree
$k$, compactly supported in the open interior of $M$. The covariant
derivative extends by Leibniz rule to a differential on
$\Omega^*_0(M,E)$ and by flatness defines the corresponding de Rham
complex $(\Omega^*_0(M,E), d_*)$.\medskip

We should point out that the condition on the vector bundles to be associated 
to a unitary representation, is posed so that
the induced Hermitian metric structure is product over $\U$. Equivalently, in our 
statements we may simply choose a (non-canonical) Hermitian metric that is 
product over $\U$ without specifying the underlying representation $\rho$.
\medskip

The metric structures $(g,h)$ define an $L^2$-scalar product on
$\Omega^*_0(M,E)$ and we denote its completion with respect to the
$L^2$-scalar product by $L^2_*(M,E;g,h)$. Let $d^t_p$ denote the
formal adjoint of $d_p$, acting on $\Omega^p_0(M,E)$, and consider the
Hodge Laplace operator 
\begin{align*} 
\Delta_p := d^t_p d_p + d_{p-1}
d^t_{p-1}: \Omega^p_0(M,E) \to \Omega^p_0(M,E). 
\end{align*} 

In order to fix a self-adjoint realization of $\Delta_p$ in $L^2_p(M,E;g,h)$, 
we recall the notion of the \emph{maximal} domain for any differential 
operator $P:\Omega^*_0(M,E) \to \Omega^*_0(M,E)$ 
\begin{align}
\dom_{\max} (P) := \{\w \in L^2_*(M,E;g,h) \mid P \w  \in L^2_*(M,E;g,h)\},
\end{align}
where $P\w\in L^2_*$ is understood in the distributional sense. \medskip

We can now introduce self-adjoint domains of $\Delta_*$ with either
relative or absolute boundary conditions at $N'=\partial M$. More
precisely, let $\iota:N'\hookrightarrow M$ be the obvious embedding, 
and $*$ be the Hodge star operator on $M$. 
Then we define two natural geometric self-adjoint
extensions $\Delta_{p, \textup{rel}}$ and $\Delta_{p, \textup{abs}}$ 
for the Hodge Laplacian $\Delta_p$ by specifying their domains\footnote{The restrictions
$\iota^*\w,  \iota^*(d\w),  \iota^*(d^t\w)$ are well-defined for any
$\w \in \dom_{\max} (\Delta_p)$ by \cite[Th. 1.9]{Paq:PMP}.}

\begin{equation} \label{domains}\begin{split}
&\dom_{\textup{rel}}(\Delta_p):= \{\w\in
\dom_{\max} (\Delta_p) \mid \iota^*\w = 0, \iota^*(d^t\w)=0\}, \\
&\dom_{\textup{abs}}(\Delta_p):= \{\w\in
\dom_{\max} (\Delta_p) \mid \iota^*(*\w) = 0, \iota^*(*d\w)=0\}, \end{split}
\end{equation}
respectively.

\subsection{Renormalized analytic torsion on manifolds with cusps} 

Consider a family of compact submanifolds $M_R:=K\cup_{N} (N \times [1,R]) \subset M$,
parametrized by $R\geq 1$. The following observation forms the basis for 
the general definition of renormalized analytic torsion and is a consequence of
explicit computations on the cusp $\U$ and the microlocal heat kernel 
description by Vaillant \cite{Vai}, cf. \S \ref{ex1} for the proof.

\begin{thm}\label{reg-trace-thm} Let $\textup{tr}\, \cH_p$ denote the pointwise
trace of the heat kernel $\cH_p$ of the Hodge Laplacian $\Delta_p$ with either
absolute or relative boundary conditions at $N'=\partial M$. 
\begin{enumerate}
\item Then in each degree 
$p=0, \ldots, m,$ there exists a finite family $(\gamma_j)_{j=0}^r 
\subset \R$ of positive numbers, and $(k_j)_{j=0}^r \subset \N_0$ such that
\begin{align}\label{R-expansion} 
\int_{M_R} \textup{tr}\, \cH_p(t,q,q)
\textup{dvol}_g(q) \sim_{R\to \infty} \sum_{j=0}^r \sum_{k=0}^{k_j}
a_{jk}(t) \, R^{\gamma_j} \log^k (R) + a_0(t) + o(1). 
\end{align} 
\item The renormalized trace\footnote{We point out that the notion of renormalized trace for 
non-trace class operators strongly depends on the choice 
of a defining function $x$.} $\textup{Tr}_{r}\cH_p(t)$ is then defined to be the
constant term $a_0(t)$ in the asymptotics. There exists a finite family
$(\A_j)_{j=0}^\ell \subset \R$ of negative numbers, and
$(i_j)_{j=0}^\ell \subset \N_0$, such that for some $\varepsilon >0$
\begin{align}\label{t-expansion} 
\textup{Tr}_{r}\cH_p(t) \sim_{t\to
0+} \sum_{j=0}^\ell \sum_{i=0}^{i_j} b_{ij} \, t^{\A_j} \log^i (t) +
b_0 + O(t^{\varepsilon}). 
\end{align} 
\item Assume the Witt condition $H^{n/2}(N,E)=0$. Denote by $\ker \Delta_p$
the finite dimensional subspace of harmonic forms in $L^2_p(M,E;g,h)$ with fixed boundary 
conditions at $N'$. Then $(\textup{Tr}_{r}\cH_p(t) - \dim \ker \Delta_p)$ is of exponential
decay as $t\to \infty$. 
\end{enumerate}
\end{thm}

As a consequence of Theorem \ref{reg-trace-thm}, the zeta function of
the Hodge Laplacian $\Delta_p$ with either either relative or absolute boundary
conditions at $N'=\partial M$
\begin{align}\label{zeta-def} \zeta (s, \Delta_p) :=
\frac{1}{\Gamma (s)} \int_0^\infty t^{s-1}
\left(\textup{Tr}_{r}\cH_p(t) - \dim \ker \Delta_p \right) dt, \,
\Re(s) \gg 0, 
\end{align} 
is well-defined and extends meromorphically
to $\Re(s)> -\varepsilon$, with a simple pole at $s=0$ with residue $b_0$. Hence we may
define the renormalized analytic torsion as follows.

\begin{defn}\label{tors} The scalar renormalized analytic torsion of
$(M,E,g,h)$ with respect to either relative or absolute boundary
conditions at $N'=\partial M$ is denotes by either $T(M,E, N')\in \R^+$
or $T(M,E)\in \R^+$, respectively, and defined by specifying its logarithmic value as
follows\footnote{We omit the metric structures $g,h$ from notation in
case they are fixed.} 
\begin{align*} 
\log T(M,E, N') &:= \frac{1}{2} \sum_{p=0}^m (-1)^{p} \, p \, 
\left.\frac{d}{ds}\right|_{s=0} \left(\zeta(s,\Delta_{p, \textup{rel}})
- s^{-1} \underset{s=0}{\textup{Res}}\, \zeta(s,\Delta_{p,\textup{rel}})\right), \\
\log T(M,E) &:= \frac{1}{2} \sum_{p=0}^m (-1)^{p} \, p \, 
\left.\frac{d}{ds}\right|_{s=0} \left(\zeta(s,\Delta_{p, \textup{abs}})
-s^{-1} \underset{s=0}{\textup{Res}}\, \zeta(s,\Delta_{p,\textup{abs}})\right).
\end{align*}
\end{defn}

Let $H^p(M,E) := \ker \Delta_{p, \textup{abs}}$ denote the subspace of
harmonic forms in $L^2_p(M,E;g,h)$ with absolute boundary conditions.
We write $\det H^p(M,E):= \Lambda ^{\textup{top}} H^p(M,E)$. The
determinant line of $L^2$-cohomology is defined by 
\begin{align*} 
\det H^*(M,E) := \bigotimes_{p=0}^m \det H^p(M,E)^{(-1)^{p+1}}, 
\end{align*}
where $V^{-1}$ denotes the dual of a finite-dimensional vector space
$V$. The $L^2$-inner product of $L^2_*(M,E;g,h)$ yields a norm of
$H^*(M,E)$ and $\det H^*(M,E)$, which we denote by $\| \cdot \|_{\det
H^*(M,E)}$. Analogous construction makes sense for the Hodge Laplacian
with relative boundary conditions at $N'$. In that case we denote the
harmonic forms with relative boundary conditions by $H^*(M,E,N')$, and
its determinant line by $\det H^*(M,E,N')$.

\begin{defn}\label{tors-norm} The renormalized Ray-Singer analytic
torsion with either relative or absolute boundary conditions at
$N'=\partial M$, is defined as a norm on the corresponding determinant
line of $L^2$-cohomology by 
\begin{align*} \|\cdot \|_{(M,E)}^{RS} &:=
T(M,E) \| \cdot \|_{\det H^*(M,E)}, \\ \|\cdot \|_{(M,E,N')}^{RS} &:=
T(M,E,N') \| \cdot \|_{\det H^*(M,E,N')}. 
\end{align*} \end{defn}

\subsection{Br\"uning-Ma metric anomaly and analytic torsion of a model cusp} 

Our first main result identifies $T(\U,E,N)$ explicitly in terms of
twisted cohomology of the cross section $(N,E\restriction N)$ and
metric anomaly of analytic torsion at the regular boundary $N\times
\{1\}$. The metric anomaly has been studied by Dai-Fang in \cite{DF}
and Br\"uning-Ma \cite{BM}. In \cite{MHS}, de Melo, Hartmann and
Spreafico validated the anomaly formula of Br\"uning-Ma. We recall the
basic facts from \cite{BM} that are needed for the statement of our
first main result. \medskip

Consider an oriented compact Riemannian manifold $(X,g^X)$ of odd
dimensions, with boundary $\partial X$, equipped with a flat Hermitian
vector bundle $(E,\nabla,h)$. Denote by $\nabla^{TX}$ the Levi-Civita
connection of the Riemannian metric $g^X$. Br\"uning and Ma define in
\cite[(1.19)]{BM} a secondary class $B(\nabla^{TX})\in
\Omega^*(\partial X,E|_{\partial X})$, which depends only on the jets
of $g^X$ at $\partial X$, is trivially zero if $g^X$ is product in a neighborhood
of the boundary, and describes the metric anomaly in the
following sense. \medskip

Consider two Riemannian metrics $g^X_1, g^X_2$, on $X$ and denote by
$\nabla^{TX}_i, i=1,2,$ the corresponding Levi-Civita connections. The
Ray-Singer analytic torsion norms on $\det H^*(X,E)$, corresponding to
$g^X_1,g^X_2$, are denoted by $\|\cdot \|^{RS}_{(X,E;g^X_i)}$,
$i=1,2$, respectively, and are defined similar to Definition
\ref{tors-norm} without the renormalization procedure in Theorem
\ref{reg-trace-thm} for the heat trace. Then
\begin{align}\label{BM-thm} \log \left(\frac{\|\cdot
\|^{RS}_{(X,E,g^X_1)}}{\|\cdot \|^{RS}_{(X,E,g^X_2)}}
\right)=\frac{\textup{rank}(E)}{2} \left[\int_{\partial
X}B(\nabla^{TX}_2)-\int_{\partial X}B(\nabla^{TX}_1)\right].
\end{align}

We can now state our first main result.

\begin{thm}\label{scalar-torsion-thm} 
Denote the twisted cohomology of $(N,E\restriction N)$ by
$H^*(N,E)$ and the Euler characteristic by 
$$
\chi(N,E) = \sum_{p=0}^n (-1)^p \dim H^p(N,E).
$$ 
Then for $R>1$ sufficiently
large, the renormalized scalar analytic torsion of the model 
cusp $(\U_R=[R,\infty)\times N, g)$ with
respect to relative boundary conditions at $\partial \U_R \equiv N\times
\{R\}$ is given by 
\begin{equation*}\begin{split}
\log T(\U_R, E, N, g) &= \frac{\textup{rank}(E)}{(-2)} \int_{\partial \U_R}
B(\nabla^{T\U_R}_{g}) + \sum_{p\neq n/2}  \frac{(-1)^{p+1}}{4}\dim H^p(N,E) 
\log \left|\frac{n}{2}-p\right| \\ &+ \sum_{p\neq n/2} \frac{(-1)^{p+1}}{2} \dim H^p(N,E)
\left|\frac{n}{2}-p\right| \log \left(2 \left|\frac{n}{2}-p\right|\right) \\
&+ \sum_{p=0}^n \frac{(-1)^{p+1}}{2} \dim H^p(N,E)
\left|\frac{n}{2}-p\right| \log R.
\end{split} \end{equation*} 
\end{thm}

\subsection{Gluing formula for analytic torsion on non-compact manifolds} 

Consider a non-compact oriented odd-dimensional Riemannian manifold
$(M,g)$ with $M=K\cup_{N} \U$, where we do not specify the behavior of
$g$ over $\U$, but pose in view of Theorem \ref{reg-trace-thm} the
following

\begin{assum}\label{assum1} The pointwise trace $\textup{tr}\, \cH_p$ of the heat
kernel $\cH_p$ of the Hodge Laplacian $\Delta_p$ with either relative or
absolute boundary conditions at $N'$ admits in each degree $p=0,
\ldots, m,$ an asymptotic expansion \eqref{R-expansion} and the
renormalized trace $\textup{Tr}_{r}\cH_p(t)$ is defined as the constant
term in that asymptotic expansion. Assume $\textup{Tr}_{r}\cH_p(t)$ 
admits asymptotic expansions
\begin{equation}\label{t-asymptotics} \begin{split}
\textup{Tr}_{r}\cH_p(t) &\sim_{t\to 0+} \sum_{j=0}^\ell
\sum_{i=0}^{i_j} b_{ij} \, t^{\A_j} \log^i (t) + b_0 +
O(t^{\varepsilon}), \\ \textup{Tr}_{r}\cH_p(t) &\sim_{t\to \infty}
\sum_{j=0}^d \sum_{i=0}^{k_j} c_{ij} \, t^{-\beta_j} \log^i (t) + c_0
+ O(t^{-\delta}), \end{split} \end{equation} 
for some finite families
$(\A_j)_{j=0}^\ell, (\beta_j)_{j=0}^d \subset (0,\infty)$,
integers $(i_j)_{j=0}^\ell, (k_j)_{j=0}^d \subset \N_0$, coefficients
$(b_{ij})_{ij}, (c_{ij})_{ij} \subset \R$ and $\varepsilon, \delta > 0$. 
We assume that the subspace $\ker \Delta_p$ of
$L^2$-integrable harmonic forms is finite dimensional.
\end{assum}

Under Assumption \ref{assum1} we may define in each degree $p=0,\ldots, m,$
the zeta-function of the Hodge Laplacian $\Delta_p$ in terms of regularized integrals
\begin{align*} \zeta(s,\Delta_p) &:= \frac{1}{\Gamma(s)}\regint_0^1
t^{s-1} \textup{Tr}_{r}\cH_p(t)dt
\\ &+ \frac{1}{\Gamma(s)}\regint_1^\infty t^{s-1}
\textup{Tr}_{r}\cH_p(t)dt, \
\Re(s) \in (-\varepsilon, \delta), 
\end{align*} where the regularized
integral $\regint_0^1$ is defined as the constant term in the
asymptotic expansion of $\int_u^1$ as $u\to 0$. Existence of a partial
asymptotic expansion is a consequence of \eqref{t-expansion}. The
regularized integral $\regint_1^\infty$ is defined similarly as the
constant term in the asymptotic expansion of $\int_1^u$ as $u\to \infty$. \medskip

The zeta function $\zeta(s,\Delta_p)$ extends meromorphically to an open
neighborhood of $s=0$, and following Definitions \ref{tors} and
\ref{tors-norm} we may define the renormalized Ray-Singer analytic
torsion of $(M,E,g,h)$. A gluing formula for the renormalized
Ray-Singer norm $\|\cdot \|^{RS}_{(M,E)}$ is established here under
the additional two assumptions.

\begin{assum}\label{assum2} 
\begin{enumerate}
\item Consider in each degree $p$ a smooth one-parameter family
$\Delta_{p,\theta}, \theta \in \mathbb{S}^1$, of self-adjoint operators in $L^2_p(M,E,g,h)$
with $\Delta_{p,\theta} = \Delta_p + V_\theta$, where the perturbation $V_\theta$ 
arises in one of the following two ways. \medskip

Either $\Delta_{p,\theta}$ is defined by a smooth family of 
Riemannian metrics $g_\theta$ which coincide outside a compact 
neighborhood $\mathscr{K}\subset M$. Alternatively, $V_\theta$ commutes with and vanishes
on any smooth section that is trivially zero in an open neighborhood of a compact subset
$\mathscr{K} \subset \U_{R-\varepsilon} \backslash \U^\circ_{R+\varepsilon} 
\cong (R-\varepsilon, R+\varepsilon) \times N$ for some $\varepsilon > 0$. We consider 
the obvious reflection mapping $S: (R-\varepsilon, R] \times N \to [R, R+\varepsilon)\times N$
and identify $\Omega^*(\U_{R-\varepsilon} \backslash \U^\circ_{R+\varepsilon}, E)$
with $\Omega^*([R,R+\varepsilon) \times N, E\oplus S^*E)$. We assume that under such identification
$V_\theta$ acts as a first order differential operator on $\Omega^*([R,R+\varepsilon) \times N, E\oplus S^*E)$
with compact support $\supp V_\theta \subset (R,R+\varepsilon) \times N$. 
\medskip

Assume that the corresponding one-parameter family 
of heat kernels $\cH_{p,\theta}$ satisfies 
for each $\theta \in \mathbb{S}^1$ the second 
part of Assumption \ref{assum1} concerning the large times 
asymptotic expansion of the renormalized trace. Assume that 
the large time asymptotic expansion is differentiable in $\theta$.
\medskip

Denote by $P_{p,\theta}$ the integral kernel of the orthogonal projection onto the 
kernel of $\Delta_{p,\theta}$. For any $\phi \in C^\infty_0(M)$ the kernel 
$\phi \cH_{p,\theta}$ is trace class and we assume that its trace admits an asymptotic expansion
for large times that is stable under $(t\partial_t)$ differentiation
\begin{equation}\label{large-times}
\textup{Tr}\, \phi \cH_{p,\theta} = \textup{Tr} \, \phi P_{p,\theta} + O(t^{-\sigma}), \ t\to \infty.
\end{equation}

\item Consider cutoff functions $\phi, \psi \in
C^\infty(M,\R)$ with $\supp \phi \subset M$ compact, $\supp \phi \cap
\supp \psi = \varnothing$. Denote by by $D=d+d^t$ the Gau\ss \, Bonnet operator. 
Then for any $Q\in \N$ we assume that
\begin{align*} &|\phi(q) \textup{tr}\, \cH_{p,\theta}
(t,q,\cdot) \psi(\cdot)| \leq f_1 \cdot t^Q, \\
&|\phi(q) \textup{tr}\, (D \cH_{p,\theta} (t,q,\cdot)) \psi(\cdot)| \leq
f_2 \cdot t^Q, 
\end{align*} with $f_1, f_2 \in L^2(M,E,g,h)$,
uniformly in $t\in (0,t_0]$ and $q\in M$. 
\end{enumerate}
\end{assum}

Assumption \ref{assum2} is designed specifically to cover 
relatively compact perturbations of the Hodge Laplacian that appear in the gluing 
formula for analytic torsion by Lesch \cite[Section 3.1]{Les}, as well as compactly 
supported perturbations 
of the Riemannian metric $g$. The assumption is satisfied for two fundamental 
classes of spaces, complete manifolds with a spectral gap around zero, and 
spaces with a microlocal calculus of the resolvent at low energies, cf. Guillarmou
and Sher \cite{Sher}, which corresponds to a microlocal description of the heat kernel at large times.
\medskip

We point out that assuming the partial asymptotic expansion \eqref{large-times} 
even without specifying the explicit form of the constant term, the form of the constant term is obtained
from a theorem by Chavel and Karp \cite{ChKa} with an elaboration by Simon 
\cite{Sim}, where the result by Chavel and Karp was shown 
to be a straightforward consequence of the spectral theorem 
and elliptic regularity. \medskip

The second part of Assumption \ref{assum2} is a replacement of the classical
off-diagonal Gaussian estimates on the heat kernel, with emphasis on integrability
of the estimates in the second spacial component. Gaussian upper bounds (in fact for all times)
first appeared in the setting of non-compact complete manifolds 
with bounded sectional curvature in the work of 
Cheng, Li and Yau \cite{CLY}. Davies \cite{Dav1} developped an abstract method 
for the derivation of Gaussian estimates from the log-Sobolev inequality, and established 
pointwise Gaussian bounds for the spacial and time derivative of the heat kernel in 
\cite{Dav2}. Sharp estimates have been obtained by Li and Yau \cite{LY}
under certain curvature assumptions, to name a few results in this direction. \medskip

However, without the assumption of bounded sectional curvature, 
Gaussian estimates may not hold in general, with certain examples 
discussed by Barlow and Bass \cite{BB}, cf. also Grigor'yan and Telcs \cite{GrTe}. 
Moreover, in various microlocal descriptions of the heat kernel asymptotics, see 
\cite{Vai}, \cite{Sher} and \cite{MazVer}, Gaussian estimates are not directly 
available. Assumption \ref{assum2} allows to encompass these examples and is still 
sufficient for the analytic arguments here. \medskip 

A central observation is now invariance of the renormalized Ray-Singer analytic 
torsion under compactly supported perturbation of the Riemannian metric.
\begin{thm}\label{invariance}
Let $(g_\theta)_\theta, \theta\in \R,$ denote a smooth family of Riemannian metrics on $M$,
with $\supp \frac{d}{d\theta} g_\theta$ contained in a compact neighborhood of $M$.
Then under Assumptions \ref{assum1} and \ref{assum2} the renormalized Ray-Singer analytic 
torsion $\| \cdot \|^{RS}_{(M,E,g_{\theta})}$, defined with respect to $g_\theta$, is a smooth
family of norms on $\det H^*(M,E)$ such that
$$\frac{d}{d\theta} \| \cdot \|^{RS}_{(M,E,g_{\theta})}=0.$$
\end{thm}

We derive a gluing formula for the renormalized Ray-Singer analytic 
torsion under the third and final

\begin{assum}\label{assum3} 
Either the spectrum $\spec \Delta_*$ of the
Hodge Laplacian in all degrees admits a spectral gap around zero, i.e.
there exists $\varepsilon > 0$ such that $(0,\varepsilon)
\cap \spec \Delta_* = \varnothing$; or the
Hermitian vector bundle $(E,\nabla)$ is acyclic over $N$, i.e.
$H^*(N,E)=0$. \end{assum}

Non-compact manifolds $(M,g)$ satisfying these three assumptions
include two particular classes of spaces. On one hand, the previously 
introduced manifolds with cusps and the Witt condition $H^{n/2}(N,E)=0$
satisfy Assumptions \ref{assum1}, \ref{assum2} and \ref{assum3}.
These spaces have been studied by Vaillant \cite{Vai} and are closely 
related to the hyperbolic manifolds with cusps, studied e.g. by M\"uller-Pfaff
in \cite{MPa}, \cite{MPb}. On the other hand, a second example comes from scattering manifolds, 
studied by Guillarmou and Sher \cite{Sher}, with $g\restriction \U=dx^2+ 
x^2 g^N$ and $H^*(N,E)=0$. \medskip

We now formulate the gluing formula for the renormalized Ray-Singer analytic
torsion in terms of canonical isomorphisms between determinant lines, defined
in terms of long exact sequences in cohomology. Consider $M=K\cup_N \U$, assume 
$\partial M=\varnothing$ for notational simplicity\footnote{Our main results hold also
for $\partial M \neq \varnothing$ with relative or absolute boundary conditions fixed
at $\partial M$.} and introduce for the obvious inclusion $\iota$ of $N\equiv N\times \{1\}$ 
into either $K$ or $\U$ the following complexes
\begin{align*}
&\Omega^*_r(\U,E):= \{\w \in \Omega^*(\U,E) \mid \iota^*\w=0\}, \\
&\Omega^*_r(K,E):= \{\w \in \Omega^*(K,E) \mid \iota^*\w=0\}, \\
&\Omega^*_r(M,E):=\{(\w_1, \w_2) \in \Omega^*(\U,E) \oplus \Omega^*(K,E) \mid
\iota^*\w_1=\iota^*\w_2\}.
\end{align*}
We consider the following short exact sequences of complexes
\begin{align*}
&0 \to \Omega^*_r(\U,E) \xrightarrow{\A} \Omega^*_r(M,E) 
\xrightarrow{\beta} \Omega^*(K,E) \to 0, \\
&0 \to \Omega_r^*(\U,E) \oplus \Omega_r^*(K,E) 
\xrightarrow{\gamma} \Omega^*_r(M,E) \xrightarrow{r} \Omega^*(N,E)\to 0,
\end{align*}
where $\A(\w)=(\w, 0)$ is the extension by zero, $\beta(\w_1, \w_2)=\w_2$ is the
restriction to $K$, $\gamma$ is the obvious inclusion, and $r$ the restriction to $N\times \{1\}$.
The harmonic forms of $\Omega^*_r(M,E)$ and $\Omega^*(M,E)$ coincide by \cite[Proposition 1.1]{Vi}. 
Hence, the $L^2$-cohomology of the complex $\Omega^*_r(M,E)$ coincides with $H^*(M,E)$, 
and the short exact sequences yield the following long exact sequences in cohomology
\begin{align*}
\cH(\U,K)&: \ldots H^p(\U,E,N) \xrightarrow{\A^*} H^p(M,E) \xrightarrow{\beta^*}
H^p(K,E) \xrightarrow{\delta^*} H^{p+1}(\U,E,N) \ldots, \\
\cH(\U,K,N)&: \ldots H^p(\U,E,N) \oplus H^p(K,E,N) \xrightarrow{\gamma^*} 
H^p(M,E) \xrightarrow{r^*} H^p(N,E) \\ & \qquad \qquad \qquad \qquad \qquad \quad  
\xrightarrow{\delta^*} H^{p+1}(\U,E,N) \oplus H^{p+1}(K,E,N)\ldots,
\end{align*} where $\delta^*$ denotes the respective connecting homomorphisms. \medskip

The long exact sequences in cohomology induce isomorphisms between 
determinant lines in a canonical way, cf. Nicolaescu \cite{Nic}
\begin{align*}
&\Phi: \det H^*(\U,E,N) \otimes \det H^*(K,E) \to \det H^*(M,E), \\
&\Phi': \det H^*(\U,E,N) \otimes \det H^*(K,E,N) \otimes \det H^*(N,E)
\to \det H^*(M,E).
\end{align*}

We may now state our second main result.
\begin{thm}\label{gluing-theorem}
Consider a non-compact oriented odd-dimensional Riemannian manifold
$(M,g)$ with $M=K\cup_{N} \U$ and a flat Hermitian vector bundle $(E,\nabla,h)$, 
satisfying Assumptions \ref{assum1}, \ref{assum2}, \ref{assum3}. Let the 
metric structures $(g,h)$ be product in an open neighborhood of the cut $N$. Then
renormalized Ray Singer analytic torsion obeys the following gluing laws
\begin{equation*}\begin{split}
\| \Phi (\cdot \otimes \cdot ) \|^{RS}_{(M,E)} &= 2^{-\frac{\chi(N,E)}{2}}
\| \cdot \|^{RS}_{(\U,E,N)} \otimes \| \cdot \|^{RS}_{(K,E)}, \\
\|\Phi' (\cdot \otimes \cdot \otimes \cdot )\|^{RS}_{(M,E)} &= 
\| \cdot \|^{RS}_{(\U,E,N)} \otimes \| \cdot \|^{RS}_{(K,E,N)} \otimes
\| \cdot \|_{\det H^*(N,E)}.
\end{split}\end{equation*}
\end{thm}

\subsection{Cheeger-M\"uller theorem for Witt-manifolds with cusps}\label{CM} 

Consider the non-compact manifold $M=K\cup_N \U, \partial K = N,$ and its one-point compactification
$M^*=M \cup \{\infty\}$, which may be viewed as a stratified space with the principal stratum 
$M$, a single singular stratum $\{\infty\}$ of zero dimension and a conical neighborhood $\U^*= \U \cup \{\infty\}$
with cross section $N$. \medskip

Goresky and MacPherson \cite{GM1, GM2} have introduced an intersection cohomology 
theory $IH^*_{\mathfrak{p}}(M^*,E)$\footnote{In case $\partial M \neq \varnothing$ we fix either relative 
or absolute boundary conditions at the boundary in the combinatorial as well as in the analytic setting.
In the combinatorial setting, a cochain satisfying relative boundary conditions is zero on boundary chains,
by definition; and absolute boundary conditions pose no restriction.} 
of stratified spaces by specifying a geometric condition of allowable simplicial chains, 
the so-called perversity $\mathfrak{p}$. Assuming the Witt condition $H^{n/2}(N,E)=0$, the intersection 
cohomology $IH^*(M^*,E)$ in middle (upper and lower) perversity of Goresky-MacPherson coincides with 
the $L^2$-cohomology of $M$ with the cusp metric $g$, compare for instance the Hodge cohomology theory 
by Hausel, Hunsicker and Mazzeo \cite{HHM}, which can be easily extended to the case of flat unitary vector bundles
$E$ and yields
\begin{align*}
IH^*(M^*,E) \cong H^*(M,E).
\end{align*}
Let $h$ denote a preferred basis on $IH^*(M,E)$ and consider 
the (scalar) intersection R-torsion $\tau(M^*,E,h)$ of $M^*$ 
defined with respect to the preferred basis $h$. The intersection R-torsion 
has been introduced by Dar \cite{Dar} for flat vector bundles over $M^*$,
which excludes vector bundles defined over $M$ with non-trivial restriction $E\restriction N$.
The extension of the combinatorial construction to flat vector bundles defined in the
smooth interior $M\subset M^*$ is due to Albin, Rochon and Sher \cite[\S 8]{ARS}.
\medskip

In \cite[Proposition 8.14 and 8.15]{ARS} the authors prove that the intersection 
R-torsion $\tau(M^*,E,h)$, defined for flat vector bundles over $M$ 
corresponding to unimodular representations of the fundamental group, 
is indeed a topological invariant, depending only on the choice of $h$.
The basis $h$ yields a norm $\| \cdot \|_{\det IH^*(M^*,E),h}$
on the determinant line bundle of the intersection cohomology and
we define the intersection R-torsion norm by
$$
\| \cdot \|^R_{(M^*,E)} = \tau(M^*,E,h)\| \cdot \|_{\det IH^*(M^*,E),h}.
$$
As a norm, the intersection R-torsion is independent of the choice of $h$.
\medskip

In order to state our final result, 
fix a preferred basis $h_N$ on $H^*(N,E)$
which is orthonormal with respect to $g^N$. Since 
\begin{equation*}
IH^p(\U^*,E) \cong \left\{ \begin{split}
&H^p(N,E), \  p < n/2, \\
&0, \ p> n/2, \end{split} \right. 
\end{equation*}
$h_N$ yields a preferred basis on the intersection cohomology $IH^*(\U^*,E)$.
Let $\tau(\U^*, E, h_N)$ be the scalar intersection R-torsion on $\U^*$, 
defined with respect to the preferred basis $h_N$. \medskip

Our final main result main compares
the intersection R-torsion norm of $M^*$ with the renormalized Ray-Singer analytic torsion 
of $(M,g)$ as norms on the determinant lines
$\det IH^*(M,E) \cong \det H^*(M,E)$.

\begin{thm}\label{main}
Let $(M,g)$ be an odd dimensional non-compact Witt-manifold without boundary and with a cusp end
$\U=N\times [1,\infty)$ and $g\restriction \U = x^{-2}(dx^2 + g^N)$. Let $M$ be 
equipped with a flat Hermitian vector bundle $E$, induced by a unitary 
representation of the fundamental group. The intersection R-torsion $\| \cdot \|^R_{(M^*,E)}$ and the 
renormalized Ray-Singer analytic torsion $\| \cdot \|^{RS}_{(M,E,g)}$, both defined with respect 
to relative or absolute boundary conditions at $\partial M$, are norms on the determinant 
line $\det IH^*(M,E) \cong \det H^*(M,E)$ and 
\begin{equation*}\begin{split}
\log \frac{\| \cdot \|^{RS}_{(M,E,g)}}{\| \cdot \|^{R}_{(M^*,E)}}
= &-\log \tau(\U^*,E,h_N) \\
&+ \sum_{p\neq n/2}  \frac{(-1)^{p+1}}{2}\dim H^p(N,E) 
\left|\frac{n}{2}-p\right| \log \left( 2\left|\frac{n}{2}-p\right| \right).
\end{split}\end{equation*}
\end{thm}

The statement extends in the obvious way to the case of finitely many cusps 
by the gluing formula for the renormalized and intersection R-torsions. \medskip

The basic idea behind the proof of Theorem \ref{main} is the use of the 
gluing formula in Theorem \ref{gluing-theorem}
to reduce the analysis to the model cusp $\U$. On the analytic side we may then
apply Theorem \ref{scalar-torsion-thm}. On the combinatorial side we point out that 
the intersection R-torsion of a model cone has been studied by Dai-Huang in \cite{DH}.
\bigskip

\emph{Acknowledgements.} I thank Werner M\"uller for suggesting the topic and encouragement, Jonathan Pfaff, 
Matthias Lesch, Xianzhe Dai, Ulrich Bunke and Frederik Rochon for helpful discussions. 
I am grateful to Luiz Hartmann for careful reading of the manuscript and the explicit computations on the 
model cusp. I greatly appreciate the careful reading and valuable improvements suggested by the anonymous
referee. I thank the Hausdorff Center for Mathematics in Bonn and the Institute for Mathematics and Computer
Sciences in M\"unster for hospitality and financial support.

\section{Asymptotics of modified Bessel functions}\label{bessel-section}

In this section we gather all the relevant statements on the asymptotics
of the modified Bessel functions of first and second kind, denoted by 
$I_t(s)$ and $K_t(s)$, respectively. We denote their respective derivatives in 
$s$ by $I'_t(s)$ and $K'_t(s)$. We distinguish between the following 
three cases: fixed argument $s$ and large order $t$, large argument and 
fixed order, as well as large argument and large order. We employ the standard references
Abramowitz and Stegun \cite{AS}, Olver \cite{Olv:AAS} as well as Watson \cite{Wat}. 
We also refer to Sidi and Hoggan \cite{Sidi} in \S \ref{bessel-2-section}.

\subsection{Asymptotics for large arguments and fixed order}\label{bessel-1-section}

We infer from \cite[(9.7.1), (9.7.2)]{AS}, see also \cite[p.202, \S 7.23 (1)-(2)]{Wat}, 
that for fixed order $t\in \C$ and large argument $s$ the Bessel functions admit 
the following asymptotic expansions 
\begin{equation}\label{IK1}
\begin{split}
&I_t(s) = \frac{e^s}{\sqrt{2\pi s}} \left( 1 + 
\sum_{k=1}^\infty a_k s^{-k}\right), \\
&K_t(s) = \sqrt{\frac{\pi}{2s}} e^{-s}  \left( 1 + 
\sum_{k=1}^\infty b_k s^{-k}\right), 
\end{split} \ \textup{as} \ |s|\to + \infty,
\end{equation}
with $|\arg(s)| < \frac{\pi}{2}$ in case of $I_t(s)$,
and $|\arg(s)| < \frac{3\pi}{2}$ in case of $K_t(s)$. 
The expansions hold uniformly if $|\arg(s)| \leq \frac{\pi}{2} -\varepsilon$ in case of $I_t(s)$,
and $|\arg(s)| \leq \frac{3\pi}{2} -\varepsilon$ in case of $K_t(s)$, for any $\varepsilon > 0$.
As is explained in the last paragraph of \cite[p. 199]{Wat}, the asymptotic expansions 
\cite[p. 199, \S 7.21 (1)-(4)]{Wat} and hence also the expansions \eqref{IK1} here, hold
locally uniformly for any order $t\in \C$, modulo an eventual change of coefficients.
The coefficients $a_k$ and $b_k$ are polynomials in $4t^2$
of order $k\in \N$. \medskip

The derivatives of the Bessel functions also admit an asymptotic expansion of 
the same structure albeit with different coefficients, cf. \cite[(9.7.3), (9.7.4)]{AS}
\begin{equation}\label{IK2}
\begin{split}
&I'_t(s) = \frac{e^s}{\sqrt{2\pi s}} \left( 1 + 
\sum_{k=1}^\infty a'_k s^{-k}\right), \\
&K'_t(s) = - \sqrt{\frac{\pi}{2s}} e^{-s}  \left( 1 + 
\sum_{k=1}^\infty b'_k s^{-k}\right),
\end{split} \ \textup{as} \ |s|\to + \infty,
\end{equation}
with $|\arg(s)| < \frac{\pi}{2}$ in case of $I'_t(s)$,
and $|\arg(s)| < \frac{3\pi}{2}$ in case of $K'_t(s)$. 
The expansions hold uniformly if $|\arg(s)| \leq \frac{\pi}{2} -\varepsilon$ in case of $I_t(s)$,
and $|\arg(s)| \leq \frac{3\pi}{2} -\varepsilon$ in case of $K_t(s)$, for any $\varepsilon > 0$.
As before, the coefficients $a'_k$ and $b'_k$ are polynomials in $4t^2$
of order $k\in \N$. The expansions \eqref{IK1} and \eqref{IK2}
hold locally uniformly in $t\in \C$. \medskip

We will also need an expansion of the derivative of $K_t(s)$
with respect to its order. We quote \cite[p. 325 Exercise 1.2 and 1.3]{Olv:AAS}
and obtain for fixed $t$, locally uniformly in $s$ with $|\arg(s)| < \frac{3}{2}\pi$
\begin{equation}\label{IK3}
\frac{d}{dt} K_t(s) \sim \sqrt{\frac{\pi}{2s}} \frac{t e^{-s} }{s} \left( 1 + 
\sum_{k=1}^\infty c_k s^{-k}\right), \ |s| \to \infty.
\end{equation}

\subsection{Asymptotics for fixed arguments and large order}\label{bessel-2-section}

For the asymptotics of modified Bessel functions for large order 
we refer to Sidi and Hoggan \cite{Sidi}. Even though such a full asymptotic expansion
seems not being presented elsewhere, it can also be derived from \eqref{uniform} and 
\eqref{uniform2} below by taking the argument to zero. Asymptotics of $I_t(s)$
also follows from the behaviour of the (unmodified) Bessel function \cite[(9.3.1)]{AS},
compare also \cite[p. 374 (7.01)]{Olv:AAS} for the leading order term in the asymptotics.
The Stirling formula asymptotics for the Gamma function, see e.g. 
\cite[(6.1.37)]{AS}, asserts that for $|\arg(t)|<\pi$ as $|t|\to \infty$ we have 
\begin{align}
\Gamma(t) \sim \left(\frac{t}{e}\right)^t \sqrt{\frac{2\pi}{t}}
\left(1+\sum_{j=1}^\infty \frac{g_j}{t^{j}}\right).
\end{align}
In view of that expansion, we infer from \cite{Sidi} locally uniformly in $|\arg(s)|< \pi$
\begin{equation}\label{t1}
\begin{split}
&I_t(s) \sim \frac{1}{t\Gamma(t)}
\left(\frac{s}{2}\right)^t \left(1+\sum_{j=1}^\infty \frac{A_j}{t^{j}}\right)
\sim \frac{1}{\sqrt{2\pi t}}
\left(\frac{e s}{2t}\right)^t \left(1+\sum_{j=1}^\infty c_jt^{-j}\right), 
\\ &K_t(s) \sim  \frac{\Gamma(t)}{2}
\left(\frac{s}{2}\right)^{-t} \left(1+\sum_{j=1}^\infty \frac{A_j}{(-t)^{j}}\right) 
\sim \sqrt{\frac{\pi}{2t}}
\left(\frac{e s}{2t}\right)^{-t} \left(1+\sum_{j=1}^\infty d_jt^{-j}\right), 
\end{split}
\end{equation}
as $|t|\to \infty$ with $|\arg(t)| < \pi$ in case of $I_t(s)$,
and $|\arg(t)| < \frac{\pi}{2}$ in case of $K_t(s)$. The expansions hold 
uniformly if $|\arg(s)| \leq \pi-\varepsilon$ in a bounded domain, $|\arg(t)| \leq \pi -\varepsilon$ in case of $I_t(s)$,
and $|\arg(t)| \leq \frac{\pi}{2} -\varepsilon$ in case of $K_t(s)$, for any $\varepsilon > 0$.
The coefficients $A_j$, $c_j$ and $d_j$ are polynomials in $(s/2)^2$ of degree $j\in \N$.
\medskip

Similar expansions hold for the derivatives. Recall the standard 
recurrence relations for the derivatives of Bessel functions, cf. 
\cite[(9.6.26)]{AS}
\begin{align}
I'_t(s) = I_{t+1}(s) + \frac{t}{s}I_t(s), \quad 
K'_t(s) = -K_{t+1}(s) + \frac{t}{s}K_t(s).
\end{align}
From here and \eqref{t1} we infer directly the following expansions for the derivatives
\begin{equation}\label{t2}
\begin{split}
&I'_t(s) \sim \frac{1}{\sqrt{2\pi t}} \frac{t}{s}
\left(\frac{e s}{2t}\right)^t \left(1+\sum_{j=1}^\infty c'_jt^{-j}\right), 
\\ &K'_t(s) \sim - \sqrt{\frac{\pi}{2t}} \frac{t}{s}
\left(\frac{e s}{2t}\right)^{-t} \left(1+\sum_{j=1}^\infty d'_jt^{-j}\right), 
\end{split} \ \textup{as} \ |t| \to \infty,
\end{equation}
uniformly for $|\arg(s)|< \pi-\varepsilon$, $|\arg(t)| \leq \pi -\varepsilon$ in case of $I'_t(s)$,
and $|\arg(t)| \leq \frac{\pi}{2} -\varepsilon$ in case of $K'_t(s)$, for any $\varepsilon > 0$. 
The coefficients $c'_j$ are polynomials in $(s/2)$ of degree $1$, the coefficients $d'_j$
are polynomials in $(s/2)^2$ of degree $j\in \N$.

\subsection{Asymptotics for large arguments and large order}\label{bessel-3-section}

We now study asymptotics of Bessel functions, when the argument and the order 
grow with a fixed ratio. Following Olver \cite[p. 377 (7.16), (7.17)]{Olv:AAS}, and his
extension of validity in \cite[Ch. 10 \S 8]{Olv}, see also \cite[(9.7.7), (9.7.8)]{AS}, 
we have for $t>0$ and $|\arg(s)| <\frac{\pi}{2}$ 
\begin{equation}\label{uniform}
\begin{split}
I_t(ts) &\sim \frac{e^{t\nu}}{(2\pi t)^{1/2} (1+s^2)^{1/4}} 
\left(1+ \sum_{k=1}^\infty \frac{U_k(p)}{t^k}\right), \\
K_t(ts) &\sim \sqrt{\frac{\pi}{2t}}\frac{e^{-t\nu}}{(1+s^2)^{1/4}} 
\left(1+ \sum_{k=1}^\infty \frac{U_k(p)}{(-t)^k}\right), 
\end{split} \ \textup{as} \ t \to \infty,
\end{equation}
where
\begin{equation}\label{pnu}
\begin{split}
&\nu = \nu(s) := \sqrt{1+s^2} + \log (s/(1+\sqrt{1+s^2})), \\
&p = p(s) := 1/\sqrt{1+s^2},
\end{split}
\end{equation}
and the coefficients $U_k(p)$ are polynomials in $p$ of degree $3k$. 
Using Cauchy product formula we find in particular\footnote{pointed out by Luiz Hartmann.} 
\begin{equation}\label{uniform1}
\begin{split}
I_t(ts) K_t(ts) &\sim \frac{1}{(2t)} (1+s^2)^{-1/2} 
\left(1+ \sum_{k=1}^\infty \frac{U_k(p)}{t^k}\right) \cdot 
\left(1+ \sum_{k=1}^\infty \frac{U_k(p)}{(-t)^k}\right)
\\ & \sim \frac{1}{(2t)} (1+s^2)^{-1/2} \left(1+ \sum_{k=1}^\infty \frac{U'_{2k}(p)}{t^{2k}}\right),
\ \textup{as} \ t \to \infty,
\end{split} 
\end{equation}
where the coefficients $U'_{2k}(p)$ are polynomials in $p$ of degree $6k$.
Similar expansions hold for the derivatives and in fact for $|\arg (s)| <\frac{\pi}{2}$
(cf. \cite[p. 378, Ex. 7.2]{Olv:AAS} and \cite[(9.7.9), (9.7.10)]{AS})
\begin{equation}\label{uniform2}
\begin{split}
I'_t(ts) &\sim \frac{e^{t\nu}}{(2\pi t)^{1/2} s (1+s^2)^{-1/4}} 
\left(1+ \sum_{k=1}^\infty \frac{V_k(p)}{t^k}\right), \\
K'_t(ts) &\sim - \sqrt{\frac{\pi}{2t}}\frac{e^{-t\nu}}{s (1+s^2)^{-1/4}} 
\left(1+ \sum_{k=1}^\infty \frac{V_k(p)}{(-t)^k}\right), 
\end{split} \ \textup{as} \ t \to \infty,
\end{equation}
where the coefficients $V_k(p)$ are again polynomials in $p$ of degree $3k$.
These asymptotic expansions in fact hold uniformly for 
$|\textup{arg}(s)| \leq \frac{\pi}{2} -\varepsilon$ for any $\varepsilon > 0$.

\subsection{Extensions of validity for uniform expansions}\label{bessel-4-section}

Extensions of validity for \eqref{uniform} and \eqref{uniform2}
have been studied by Olver \cite{Olv:AAS} Chapter 10 in the Section \S 8. 
These validity extensions are obtained as an application of the main 
theorem in \cite[p. 366, Theorem 3.1]{Olv:AAS}, where one proceeds as follows.
The Bessel equation is transformed as in \cite[p. 375 \S 7.3]{Olv:AAS} to 
\begin{align}
\frac{d^2 W}{d\nu^2} = \left(t^2 + \psi (\nu(s))\right) W, 
\quad \textup{where} \ \ \psi(\nu(s)) = \frac{s^2(4-s^2)}{4(1+s^2)^3},
\end{align}
where $\nu = \nu(s)$ is defined in \eqref{pnu} and 
the fundamental system is given by 
\begin{align}
\sqrt{s} \left(\frac{1+s^2}{s^2}\right)^{1/4} I_t(ts), 
\quad \sqrt{s} \left(\frac{1+s^2}{s^2}\right)^{1/4} K_t(ts).
\end{align}
Choose any domain $\mathbb{D}\subset \C$ for $s$ and any domain 
$\mathbb{T}\subset \C$ for $t$. The domain $\mathbb{D}$
for $s$ corresponds to a domain $\Delta \subset \C$ for $\nu(s)$, which can be constructed 
out of $\mathbb{D}$ using the transformation described in \cite[p. 375 \S 7.3 (i) - (v)]{Olv:AAS}.
The main theorem \cite[p. 366, Theorem 3.1]{Olv:AAS} now asserts that 
the asymptotic expansions of solutions of the form \eqref{uniform} and \eqref{uniform2}
hold uniformly for any choice of domains satisfying the following conditions\footnote{these
conditions are somewhat stronger than the optimal conditions posed by Olver}:

\begin{enumerate}
\item the closure of $\mathbb{D}$ does not contain the points $s=\pm i$,
which are the pole singularities of $\psi(\nu(s))$,
\item there exist reference points $\A_1, \A_2 \in \Delta$, such that any point in $\Delta$
can be connected inside $\Delta$ to $\A_j$ by a continuous path $\mathscr{L}_j$ of 
finitely many straight lines directed from $\A_j$ to $\nu$, $j=1,2$. 
\item $\Re(t \nu)$ is non-decreasing for $\nu$ varying along $\mathscr{L}_1$,
and non-increasing for $\nu$ varying along $\mathscr{L}_2$, with the prescribed orientation. 
Such paths are called \emph{progressive}.
\end{enumerate}

These conditions have been worked out by Olver in \cite[p. 381 \S 8.3]{Olv:AAS}.
The asymptotic expansion of $I_t(ts)$ and hence also of its derivative $I'_t(ts)$ 
hold uniformly for $|\arg (t)| <\frac{\pi}{2}-\varepsilon$, where $\varepsilon > 0$ is 
any small positive number, and any $s\in \C$ away from the cuts in the complex 
plane as depicted in \cite[p. 381, Fig. 8.6]{Olv:AAS}. The expansions of $K_t(ts)$ 
and hence also of its derivative $K'_t(ts)$ hold uniformly
for $|\textup{arg}(s)| \leq \frac{\pi}{2} -\varepsilon$ and $|\textup{arg}(t)| \leq \frac{\pi}{2} -\varepsilon'$ 
for any $\varepsilon, \varepsilon' > 0$, as explained in the last paragraph of 
\cite[p. 381, \S 8.3]{Olv:AAS}. In particular, \eqref{uniform} and \eqref{uniform2}
hold for the following two choices of domains $(\mathbb{T}_1, \mathbb{D}_1)$
and $(\mathbb{T}_2, \mathbb{D}_2)$
\begin{equation}\label{TD12}
\begin{split}
\mathbb{T}_1 = \{t \in \C \mid \arg(t) \in &[0, \pi /2 - \varepsilon]\} \\ 
\mathbb{D}_1 = \{s \in \C  \mid \arg(s) \in &[-\pi /2 + \varepsilon', 0]\}. \\[3mm]
&\mathbb{T}_2 = \{t \in \C \mid \arg(t) \in [-\pi /2 + \varepsilon, 0]\} \\ 
&\mathbb{D}_2 = \{s \in \C  \mid \arg(s) \in [0, \pi /2 - \varepsilon']\}.
\end{split}
\end{equation}
for any small positive numbers $\varepsilon, \varepsilon' > 0$. \medskip

Extending validity of expansions even further is possible
along the lines of \cite[p. 380, \S 8.2]{Olv:AAS}, but is not worked out explicitly in 
\cite[p. 381, \S 8.3]{Olv:AAS} for brevity reasons. However, for applications we need 
to extend the regions of validity to include segments of the imaginary axis. We define
for any small positive numbers $\varepsilon, \delta>0$
\begin{equation}
\begin{split}
\mathbb{T}_3 = \{t \in \C \mid &\arg(t) \in [\pi /2 - \varepsilon, \pi/2]\} \\ 
\Delta_3 = \{\nu \in \C  \mid &\arg(\nu) \in [-\pi/2, -\pi /2 + \varepsilon], 
\ \textup{Im}(\nu) \leq - \frac{i \pi}{2} (1+\delta) \}. \\[3mm]
&\mathbb{T}_4 = \{t \in \C \mid \arg(t) \in [-\pi/2, -\pi /2 + \varepsilon]\} \\ 
&\Delta_4 = \{\nu \in \C  \mid \arg(\nu) \in [\pi /2 - \varepsilon, \pi/2],  
\ \textup{Im}(\nu) \geq \frac{i \pi}{2} (1+\delta) \}.
\end{split}
\end{equation}
We fix the reference points as follows. In case of $(\mathbb{T}_3, \Delta_3)$ we set 
$\A_1 = - \frac{i \pi}{2} (1+\delta)$ and $\A_2 = \lim\limits_{u\to +\infty} u e^{(-\pi /2 + \varepsilon)i}$.
In case of $(\mathbb{T}_4, \Delta_4)$ we set $\A_1 = \frac{i \pi}{2} (1+\delta)$ and 
$\A_2 = \lim\limits_{u\to +\infty} u e^{(\pi /2 - \varepsilon)i}$. With these choices, 
the conditions $(i)-(iii)$ from above are satisfied and hence the theorem \cite[p. 366, Theorem 3.1]{Olv:AAS}
applies. Transforming the domains $\Delta_3, \Delta_4$ to $\mathbb{D}_3, \mathbb{D}_4$, 
respectively, we obtain after taking subdomains 
\begin{equation}\label{TD34}
\begin{split}
\mathbb{T}_3 = \{t \in \C \mid &\arg(t) \in [\pi /2 - \varepsilon, \pi/2]\} \\ 
\mathbb{D}_3 = \{s \in \C  \mid &\arg(s) \in [-\pi/2, -\pi /2 + \varepsilon'], 
\ \textup{Im}(s) \leq - i (1+\delta') \}. \\[3mm]
&\mathbb{T}_4 = \{t \in \C \mid \arg(t) \in [-\pi/2, -\pi /2 + \varepsilon]\} \\ 
&\mathbb{D}_4 = \{s \in \C  \mid \arg(s) \in [\pi /2 - \varepsilon, \pi/2],  
\ \textup{Im}(s) \geq i (1+\delta') \}.
\end{split}
\end{equation}
for some appropriate $\varepsilon \in (0, \varepsilon)$ and 
$\delta' > \delta$. We conclude that the expansions \eqref{uniform} and \eqref{uniform2}
hold uniformly for the domains $(\mathbb{T}_3, \mathbb{D}_3)$
and $(\mathbb{T}_4, \mathbb{D}_4)$ as well.

\section{The zeta determinant of scalar cuspidal operators}\label{section-det}
 
In this section we study scalar Sturm-Liouville operators which will naturally 
appear in the analysis of the Hodge Laplacian of a manifolds with cusps.
We establish existence and a variation formula for the zeta-determinant of a 
scalar cusp-type Sturm-Liouville operator. More precisely, 
fix $\mu>0$ and consider a family of scalar \emph{cusp} operators
\begin{align}\label{d-t}
D_t := -(x\partial_x)^2 - (x\partial_x) + x^2 \mu^2 + t^2 -\frac{1}{4}: 
C_0^\infty(R,\infty) \to  C_0^\infty(R,\infty), \quad (t\geq 0).
\end{align}
The second order differential equation $D_t f=0$ admits a fundamental system 
of solutions $x^{-1/2}I_t(\mu x), x^{-1/2}K_t(\mu x),$ in terms of modified
Bessel functions of first and second order. By the first asymptotic expansion in \eqref{IK1}, 
$x^{-1/2}I_t(\mu x)$ does not lie in $L^2((R,\infty), dx)$. Consequently, $D_t$ is in the
limit point case at infinity and we a self-adjoint extension of $D_t$ by putting e.g. 
Dirichlet boundary conditions at $x=R$
$$
\dom(D_t) = \{f\in \dom_{\max}(D_t) \mid f(R)=0\},
$$
where we point out as before, that elements in the maximal domain $\dom_{\max}(D_t)$
are absolutely continuous and in fact continuously differentiable at $x=R$. 
Replacing Dirichlet with generalized Neumann boundary conditions $f'(R)+\A f(R)=0,$ leads 
to an analogous discussion which we do not repeat here. If the parameter $t\geq 0$
is allowed to be complex with $|\arg(t)|< \frac{\pi}{2}$, then $D_t$ is not self-adjoint 
anymore, but a closed operator in $L^2((R,\infty), dx)$ with an Agmon angle. \medskip

The purpose of the present section is the definition of the zeta-regularized 
determinant of $D_t$.

\begin{prop}
The operator $D_t$ with $|\arg(t)|< \frac{\pi}{2}$ is invertible. 
The inverse $D_t^{-1}$ is a trace class operator with trace given 
by the integral of its Schwartz kernel along the diagonal
\begin{align}
\Tr D_t^{-1} \equiv \int_R^\infty G_t(x) = \int_R^\infty (\mu x)^{-1} 
\left(I_tK_t(\mu x) - \frac{I_t(\mu R)}{K_t(\mu R)} K^2_t(\mu x)\right) dx.
\end{align}
\end{prop}

\begin{proof}
Assume first that $t\geq 0$. Denote by $\psi$ a solution to $D_t f=0$ that is square integrable at infinity, 
i.e. $\psi \in L^2((R,\infty), dx)$. Denote by $\phi$ a solution
to $D_t f=0$ satisfying Dirichlet boundary conditions at $x=R$. Both solutions are uniquely 
determined up to a multiplicative constant and we put
\begin{align*}
\psi(x) = x^{-1/2}K_t(\mu x), \quad \phi(x) = x^{-1/2} \left(
I_t(\mu x) - \frac{I_t(\mu R)}{K_t(\mu R)} K_t(\mu x)\right),
\end{align*} 
where we point out that $K_t(\mu x) > 0$ is nowhere vanishing for $x>0$
and $t > -1$, as asserted e.g. in \cite[p. 374]{AS}. Hence $\psi(x)$ is strictly positive. 
Since $I_t(\mu x)$ is growing monotonously, while $K_t(\mu x)$ is falling monotonously, 
the quotient $I_t(\mu x)/K_t(\mu x)$ is growing monotonously and hence 
by positivity of $K_t(\mu x)$, the second solution $\phi(x)$ is also strictly 
positive for $x>R$. The corresponding Wronski determinant is computed as follows 
\begin{align*}
W(\phi, \psi) &= (\phi' \psi - \psi' \phi)(x) = 
x^{-1} \mu (I'_t K_t - I_t K'_t)(\mu x) = \frac{1}{x^2}.
\end{align*}
The Green function $G_t$ of $D_t$ is obtained by the usual ansatz
\begin{align*}
G_t(x,y) = \left\{ \begin{array}{c}
A \phi(x) \psi(y), \ x\leq y, \\ A\psi(x) \phi(y), \ x\geq y,
\end{array}\right. 
\end{align*}
where $A$ is computed from the condition $D_tG_t(\cdot, y) = \delta(\cdot - y)$
and is given in terms of the Wronski determinant by 
$$
A= \left(x^2 W(\phi, \psi)\right)^{-1}= 1.
$$
In particular we find for the Green function at the diagonal
\begin{align}\label{resolvent-kernel}
G_t(x) \equiv G_t(x,x) = x^{-1} 
\left(I_tK_t(\mu x) - \frac{I_t(\mu R)}{K_t(\mu R)} K^2_t(\mu x)\right).
\end{align}
The Green function $G_t(x,y)$ is continuous on $[R, \infty) \times [R, \infty)$ and by positivity
of solutions $\phi$ and $\psi$, it is non-negative and positive away from $x, y = R$. 
Moreover, $G_t$ is integrable on $[R,\infty)$ along the diagonal by the asymptotic 
expansion \eqref{IK1}. Consequently, by the Mercers theorem, as worked out e.g. by 
Reed and Simon \cite[\S XI.4, Lemma on p. 65]{Reed} we conclude that 
the resolvent $D_t^{-1}$ exists\footnote{$K_t(R)\neq 0$ and
hence $\ker D_t$ is trivial.} and is trace class. This already proves the proposition 
in case $t\geq 0$.\medskip

In particular, we conclude that $D_t^{-1}$ is a compact self-adjoint operator with discrete spectrum accumulating at zero. 
Hence the cusp operator $D_t$ is a self-adjoint operator in $L^2((R,\infty), dx)$ with discrete
positive spectrum accumulating at infinity. The shifted operator $(D_t + z^2)$ with $|\arg(z)|< \frac{\pi}{2}$ still 
admits a trivial kernel and the Schwartz kernel of its resolvent $(D_t + z^2)^{-1}$ is given by
$G_{\sqrt{t^2+z^2}}$, which can be checked to be well-defined and continuous but clearly not positive 
any longer. That Schwartz kernel is still integrable on $[R,\infty)$ along the diagonal by the asymptotic 
expansion \eqref{IK1}. Since the spectrum of $(D_t + z^2)$ and hence also of the inverse $(D_t + z^2)^{-1}$ 
is discrete, the integral of $G_{\sqrt{t^2+z^2}}$ along the diagonal $[R,\infty)$ equals the trace of the 
resolvent $(D_t + z^2)^{-1}$. \medskip

In particular, we find that for any $t\in \C$ with $|\arg(t)|< \frac{\pi}{2}$
we may write  
\begin{align*}
\textup{Tr} \, D_t^{-1} &= \int_R^\infty G_t(x) dx.
\end{align*}
\end{proof}

We now study the asymptotic expansion of the resolvent trace.

\begin{prop}\label{trace-expansion-prop}
The trace of the resolvent $D_t^{-1}$ with $|\arg(t)|< \frac{\pi}{2}$ admits an expansion 
\begin{equation}\label{trace-expansion}
\begin{split}
\textup{Tr}\, D_t^{-1} \sim \sum_{k=0}^\infty a_k t^{-1-k} + 
\sum_{k=0}^\infty b_k t^{-1-2k} \log (t), \ |t|\to \infty.
\end{split}
\end{equation}
\end{prop}

\begin{proof}
Let us rewrite the trace as follows
\begin{align*}
\textup{Tr} \, D_t^{-1} &= \int_R^\infty G_t(x) dx 
\\ &= \int_R^{2R} G_t(x) dx + t \int_{2R/t}^{1/\mu} G_t(xt) dx + t \int_{1/\mu}^\infty G_t(xt) dx
\\ &=:T_1(t) + T_2(t) + T_3(t),
\end{align*}
where we have substituted $x$ with $xt$ in $T_2$ and $T_3$. 
The asymptotic expansion of $T_1$ as $t\to \infty$ is the simplest of 
all three and follows directly from 
\eqref{t1}. In particular we obtain after cancellations
\begin{equation}
T_1(t) \sim \sum_{k=0}^\infty a'_k t^{-1-k}, \ t\to \infty.
\end{equation}
Let us now study the asymptotics of $T_2$, which in view of the representation
\eqref{resolvent-kernel} is given explicitly by the following expression
\begin{equation}
T_2(t) = t \int_{2R/t}^{1/\mu} (\mu xt )^{-1} 
\left(I_tK_t(\mu x t) - \frac{I_t(\mu R)}{K_t(\mu R)} K^2_t(\mu xt)\right) dx. 
\end{equation}
We now proceed with estimating the second summand in that last 
expression, where we denote all uniform constants (possibly dependent on $\mu$) by $C>0$.
In view of the asymptotics \eqref{t1} we find
\begin{equation}\label{IKR}
\left| \frac{I_t(\mu R)}{K_t(\mu R)} \right| \leq C \left(\frac{e \mu R}{2t}\right)^{2t}.
\end{equation}
In view of the uniform asymptotics \eqref{uniform} we also find for $x<1/\mu$
\begin{equation}
\left| K^2_t(\mu x t)\right| \leq C t^{-1} \left(\frac{e \mu x}{1+\sqrt{1+(\mu x)^2}}\right)^{-2t}
\leq C t^{-1} \left(\frac{e \mu x}{3}\right)^{-2t}.
\end{equation}
From here we conclude after cancellations for any $N\in \N$
\begin{equation}\begin{split}
\int_{2R/t}^{1/\mu} (\mu x )^{-1} \left| \frac{I_t(\mu R)}{K_t(\mu R)} K^2_t(\mu x t) \right|
&\leq C t^{-1} \left(\frac{3R}{2t}\right)^{2t} \int_{2R/t}^{1/\mu} x^{-2t-1} dx
\\ &\leq C t^{-2} \left(\frac{3R}{2t}\right)^{2t} \left(\frac{2R}{t}\right)^{-2t}
\leq C t^{-2} \left(\frac{3}{4}\right)^{2t} \\ & = O(|t|^{-N}), \ \textup{as} \ |t|\to \infty,
\end{split}\end{equation}
From here and in view of the uniform expansion \eqref{uniform1} we conclude 
\begin{equation}\begin{split}
T_2(t) &= \int_{2R/t}^{1/\mu} (\mu x)^{-1} 
I_tK_t(\mu x t) dx + O(|t|^{-N}), \\
&\sim \sum_{k=0}^\infty a''_k t^{-1-k} + 
\sum_{k=0}^\infty b_k t^{-1-2k} \log (t), 
\end{split} \qquad \textup{as} \ |t|\to \infty,
\end{equation}
For $T_3$ we compute similarly, using \eqref{IKR} and \eqref{uniform}
\begin{equation}\begin{split}
T_3(t) = \int_{1/\mu}^\infty (\mu x)^{-1} 
I_tK_t(\mu x t) dx + O(|t|^{-N}) \sim \sum_{k=0}^\infty a'''_k t^{-1-k},
\end{split} \qquad \textup{as} \ |t|\to \infty,
\end{equation}
Summarizing the expansions for $T_1, T_2$ and $T_3$, we conclude
with an asymptotic expansion for the resolvent trace, as stated.
\end{proof}

Note that the resolvent trace asymptotic expansion as established in 
Proposition \ref{trace-expansion-prop} does not admit terms of the form
$t^{-2}\log^k(t), k\in \N$. Hence we may now define the zeta-regularized determinant of $D_t$
using the notion of regularized integrals as Lesch \cite{Les:DRS}.

\begin{defn}\label{det-resolvent}
Denote by the regularized limit $\LIM_{\varepsilon \to 0}$ 
the contant term in the asymptotic expansion as $\varepsilon \to 0$, 
and by $\LIM_{\delta\to \infty}$ the contant term in the 
asymptotic expansion as $\delta\to \infty$. Then the zeta-regularized
determinant $\det\nolimits_{\zeta} D_t$ is defined by\footnote{The 
zeta-regularized determinant can be equivalently defined using the zeta-function
$\zeta(s,D_t)$ introduced in \eqref{zeta-def}, by extending $\zeta(s,D_t)$
meromorphically to $\C$ with a regular point at $s=0$, and 
setting $\log \det\nolimits_{\zeta} D_t := -\zeta'(0,D_t)$.} 
\begin{align*}
\log \det\nolimits_{\zeta} D_t &:= -2 \LIM_{\delta\to \infty} 
\LIM_{\varepsilon \to 0} \int_{\varepsilon}^\delta 
z \, \textup{Tr}\, (D_t + z^2)^{-1} dz \\ &=:-2 \regint_0^\infty 
z \, \textup{Tr}\, (D_t + z^2)^{-1} dz. 
\end{align*}
\end{defn}

\section{Variation of the zeta determinant of a cuspidal operator}\label{variation-section}

In this section we establish a variational formula for $\log \det\nolimits_{\zeta} D_t$
for variable parameter $t$ in the spirit of 
Lesch \cite[Prop. 3.4]{Les:DRS}. Variation with respect to other terms in the differential
expression for $D_t$ is an interesting question in itself, which is addressed in a forthcoming
project jointly with Hartmann and Lesch \cite{LH}.

\begin{prop}\label{t-derivative}
The zeta-regularized determinant of $D_t$ is differentiable in $t\in \R$ with 
$$
\left. \frac{d}{dt}\right|_{t=t_0} \log \det\nolimits_{\zeta} D_t = 
-2t_0 \textup{Tr}\, (D_{t_0})^{-1}.
$$
\end{prop}

\begin{proof}
By a Neumann series argument we find (we write $I$ for the identity operator)
\begin{equation}\label{neumann}
\begin{split}
(D_t+z^2)^{-1} &= (D_{t_0} + z^2 + (t^2-t_0^2))^{-1} \\ &= 
\left(I+(t^2-t_0^2)(D_{t_0}+z^2)^{-1}\right)^{-1} (D_{t_0}+z^2)^{-1} \\
&= \sum_{n=0}^{\infty} (-1)^{n} \left((t^2-t_0^2)(D_{t_0}+z^2)^{-1}\right)^n (D_{t_0}+z^2)^{-1} 	
\end{split}
\end{equation}
By an argument similar to Proposition \ref{trace-expansion-prop}, 
the trace norm of $(D_{t_0}+z^2)^{-1}$ is of the order $O(z^{-1}\log(z))$
as $z \to \infty$. Moreover, the operator norm of $(D_{t_0}+z^2)^{-1}$ is of 
the order $O(z^{-2})$ as $z \to \infty$. Taking trace norms of the expression 
in \eqref{neumann}, we obtain for some uniform constant $C>0$ and $z>1$
\begin{equation}\label{Trac.Dif.t}
\begin{split}
& \|(D_t+z^2)^{-1} - (D_{t_0}+z^2)^{-1}\|_{\Tr} 
\\ & \leq \| (D_{t_0}+z^2)^{-1} \| \cdot \sum_{n=0}^{\infty} (-1)^{n} 
\left|t^2-t_0^2\right|^n \| (D_{t_0}+z^2)^{-1}\|^n_{\Tr}  
\\ & \leq \frac{C}{z^2} \sum_{n=1}^{\infty} 
\left(\left|t^2-t_0^2\right| z^{-1} \log (z) \right)^n 
\leq C \left|t^2-t_0^2\right| z^{-3}\log(z). 
\end{split}
\end{equation}
Next, we consider the difference of logarithmic 
zeta-determinants
\begin{equation}\label{difference}
\begin{split}
\log \det\nolimits_{\zeta} D_t - \log \det\nolimits_{\zeta} D_{t_0} 
= -2 \regint_0^\infty z \left( \textup{Tr}\, (D_t + z^2)^{-1} -  \textup{Tr}\, (D_{t_0} + z^2)^{-1}\right) dz.
\end{split}
\end{equation}
From the estimate \eqref{Trac.Dif.t} we easily conclude 
\begin{equation}\label{estimate-trace2}
\begin{split}
\frac{d}{dt} \left. \textup{Tr}\, (D_t + z^2)^{-1}
\right|_{t=0} &\equiv \frac{d}{dt} \left. \left( \textup{Tr}\, (D_t + z^2)^{-1} 
-  \textup{Tr}\, (D_{t_0} + z^2)^{-1} \right)
\right|_{t=0} \\ &= O(z^{-3}\log (z)), \quad z \to \infty.
\end{split}
\end{equation}
Hence the regularized integral in \eqref{difference} may be replaced by the standard
integral, and we can differentiate under the integral to obtain
\begin{align*}
\left. \frac{d}{dt}\right|_{t=t_0} 
&\left(\log \det\nolimits_{\zeta} D_t - \log \det\nolimits_{\zeta} D_{t_0} \right)
\\ &= -2\int_0^\infty z \left. \frac{d}{dt}\right|_{t=t_0}  
\left( \textup{Tr}\, (D_t + z^2)^{-1} -  \textup{Tr}\, (D_{t_0} + z^2)^{-1}\right) dz
\\ &=-2 \int_0^\infty z \, \textup{Tr}\, \left((D_t + z^2)^{-1} 
\left. \frac{d}{dt}  D_t \ (D_t + z^2)^{-1} \right|_{t=t_0} \right) dz
\\ &= -4t_0 \int_0^\infty z \, \textup{Tr}\, (D_{t_0} + z^2)^{-2} 
= 2t_0 \int_0^\infty \frac{d}{dz} \textup{Tr}\, (D_{t_0} + z^2)^{-1} dz 
= -2t_0 \textup{Tr}\, (D_{t_0})^{-1}.
\end{align*} 
\end{proof}

We can now prove the main result of this subsection.

\begin{thm}\label{variation}
Consider solutions $\phi, \psi$ of $D_t$, where $\psi(x)=x^{-1/2}K_t(\mu x)$,
and $\phi$ satisfies Dirichlet boundary conditions at $x=R$, normalized such that 
$\phi'(R)=1$. Then 
\begin{align*}
\frac{d}{dt} \log \det\nolimits_\zeta D_t &=  \frac{d}{dt} \log \left(x^2 W(\phi, \psi) \right)
\\ &= - \frac{d}{dt} \log \left( I'_t(\mu R) - \frac{I_t(\mu R)}{K_t(\mu R)} K'_t(\mu R)\right).
\end{align*}
\end{thm}

\begin{proof}
The solutions $\phi, \psi$ satisfy the following relations
\begin{align*}
&\left((x\partial_x)^2 + (x\partial_x)\right) \phi = 
\left(x^2\mu^2 + t^2 -\frac{1}{4}\right) \phi, \\
&\left((x\partial_x)^2 + (x\partial_x)\right) \partial_t \psi = 
2t \psi + \left(x^2\mu^2 + t^2 -\frac{1}{4}\right) \partial_t \psi.
\end{align*}
Hence we compute for the Wronskian of $\phi$ and $\partial_t \psi$
\begin{equation}\label{xW}
\begin{split}
&(x\partial_x + 1) [xW(\phi, \partial_t \psi)] 
\\ &= \partial_t \psi \left((x\partial_x)^2 + (x\partial_x)\right) \phi - 
\phi \left((x\partial_x)^2 + (x\partial_x)\right) \partial_t \psi \\
&= \partial_t \psi \left(x^2\mu^2 + t^2 -\frac{1}{4}\right) \phi
- \phi \left(2t \psi + \left(x^2\mu^2 + t^2 -\frac{1}{4}\right) \partial_t \psi\right)
\\ &= -2 t \phi \psi.
\end{split}
\end{equation}
By \eqref{IK3}, $\partial_t \psi = O(e^{-\mu x}/x^2)$ and hence 
$W(\phi, \partial_t \psi)=O(x^{-3})$, as $x\to \infty$. Hence we may compute using
integration by parts and \eqref{xW}
\begin{equation}\label{trace-wronskian}
\begin{split}
-2t (x^2 W(\phi,\psi)) \textup{Tr}\, D_t^{-1} &= 
\int_R^\infty (x\partial_x + 1) [xW(\phi, \partial_t \psi)] dx 
\\ &= \left. x^2 W(\phi, \partial_t \psi) \right|_R^\infty 
= R^2 W(\phi, \partial_t \psi) (R).
\end{split}
\end{equation}
Under the normalization of $\phi$, $W(\partial_t \phi, \psi)(R)=0$ and hence
\begin{align}\label{tW}
\partial_t W(\phi, \psi) = W(\phi, \partial_t \psi) (R).
\end{align}
Using Proposition \ref{t-derivative} and the relations 
\eqref{trace-wronskian}, \eqref{tW} we find
\begin{align*}
\frac{d}{dt} \log \det\nolimits_\zeta D_t  = 
-2t \, \textup{Tr}\, D_t^{-1} = \frac{\partial_t ((x^2 W(\phi, \psi)) (R))}
{(x^2 W(\phi, \psi)) (R)}.
\end{align*}
\end{proof}

Similar computations apply to the case a self-adjoint extension of $D_t$
with generalized Neumann boundary conditions at $x=R$. The corresponding 
result reads as follows.

\begin{thm}\label{variation-Neumann}
Consider a self-adjoint extension of $D_t$ with generalized Neumann boundary 
conditions 
$$
\dom(D_t) = \{f\in \dom_{\max}(D_t) \mid f'(R) + \A f(R)=0\}.
$$
Assume that $\ker D_t$ is trivial. 
Consider solutions $\phi, \psi$ of $D_t$, where $\psi(x)=x^{-1/2}K_t(\mu x)$,
and $\phi$ satisfies the generalized Neumann boundary conditions at $x=R$,
normalized such that $\phi(R)=1$. Then 
\begin{align*}
\frac{d}{dt} \log \det\nolimits_\zeta D_t &=  \frac{d}{dt} \log \left(x^2 W(\phi, \psi) \right)
\\ &= - \frac{d}{dt} \log \left( I_t(\mu R) - \frac{\left(\mu I'_t + \left(\A -\frac{1}{2R}\right) I_t\right)(\mu R)}
{\left(\mu K'_t + \left(\A -\frac{1}{2R}\right) K_t\right)(\mu R)} K_t(\mu R)\right).
\end{align*}
\end{thm} 

\begin{remark}\label{kernel1}
Consider the self-adjoint extension of $D_t$ in $L^2((R,\infty), dx)$ with 
generalized Neumann boundary conditions $f'(R) + \A f(R)=0$. A solution 
$\psi \in L^2((R,\infty), dx)$ to $D_t\psi=0$ is given up to a multiplicative 
constant by $\psi(x)=x^{-1/2}K_t(\mu x)$. Hence $\ker D_t$ is non-trivial 
only if $\psi$ satisfies $\psi'(R) + \A \psi(R)=0$, i.e.
\begin{align*}
\psi \in \ker D_t \ \Leftrightarrow \ \frac{K'_t(\mu R)}{K_t(\mu R)}
= - \left(\A-\frac{1}{2R}\right) \mu^{-1}.
\end{align*}
Using \eqref{IK1} and \eqref{IK2} we find that $(K'_t / K_t)(\mu R) \sim O(\mu^{-1}) -1,$
as $\mu \to \infty$, and hence $\ker D_t = \{0\}$ for $\mu >0$ sufficiently large.
In particular, the operators $\Delta_1, \Delta'_1$ in \eqref{ops} may be assumed 
to be invertible by an appropriate rescaling of the metric $g^N$.
\end{remark}

\section{The de Rham complex of a model infinite cusp}\label{decomposition-section}

Let $(N,g^N)$ be a closed even-dimensional oriented Riemannian manifold, $\dim N=n$,
and consider the model cusp $\U_R=N\times [R,\infty), R>0$, with the cusp metric 
$$
g=\frac{dx^2 + g^N}{x^2}, \quad x\in [R,\infty).
$$ 

Fix a base point $q=(y_0, x_0) \in \U_R$ and consider a unitary representation
$\rho: \pi_1(\U_R,q) \to U(r,\C)$ of the fundamental group $\pi_1(\U_R,q)$. 
The corresponding flat Hermitian vector bundle $(E,\nabla, h)$ over $\U_R$ is equipped with 
the canonical Hermitian metric $h$, and the canonical flat covariant derivative 
$\nabla$, with the former induced from the standard Hermtian inner product 
on $\C^r$ and the latter induced from the exterior derivative on the universal cover 
of $\U_R$.  \medskip

By the product structure of $\U_R$, $\pi_1(\U_R,q) \cong \pi_1(N,y_0)$.
Hence the unitary representation $\rho$ also defines a flat Hermitian vector 
bundle $(E_N, \nabla_N, h_N)$ over $N$, related to the vector bundle over $\U_R$
as follows. Let $\pi: \U_R=N\times [R,\infty) \to N$ be the projection onto the first 
factor. Then for compactly supported sections $s\in \Gamma_0(E)\cong C^\infty_0
((R,\infty), \Gamma(E_N))$
\begin{align*}
&\pi^*E_N = E_N \times [R,\infty) \cong E, \\
&\pi^*h_N = h, \ \nabla s = \frac{\partial s}{dx} \otimes dx + \nabla_N s.
\end{align*}

Denote by $\Omega^p_0(\U_R, E)$ the space of $E$-valued differential forms 
of degree $p$, compactly supported in the open interior of $\U_R$. The flat covariant
derivative $\nabla$ extends by Leibniz rule to a differential operator on $\Omega^*_0(\U_R, E)$ and 
gives rise to the twisted de Rham complex $(\Omega^*_0(\U_R, E), d_*)$. 
Similarly, $\nabla_N$ extends by Leibniz rule to a differential operator on twisted 
differential forms $\Omega^*(N,E_N)$ over $N$ and gives rise to the twisted de Rham 
complex $(\Omega^*(N,E_N), d_{N,*})$. \medskip

We discuss the structure of $(\Omega^*_0(\U_R, E), d_*)$ under the transformation
\begin{align*}
&\Phi: \Omega^p_0(\U_R, E) \rightarrow C^\infty_0((R,\infty), \Omega^{p-1}(N,E_N) \oplus
\Omega^p(N,E_N)), \\ &\Phi \, (\w_p + \w_{p-1} \wedge dx) = x^{-\frac{n+1}{2}+p} (\w_{p-1}, \w_p).
\end{align*}

$\Phi$ extends to a unitary transformation on the $L^2$-completions
$$
\Phi: L^2_p(\U_R, g,h) \to L^2((R,\infty), dx; L^2_{p-1}(N,E_N,g^N,h_N) \oplus 
L^2_{p}(N,E_N,g^N,h_N)).
$$

Under this unitary transformation, $d_*$ acts as follows
\begin{align*}
d_{p}\left(\begin{array}{c} \w_{p-1} \\ \w_p \end{array}\right) = 
\left(\left(\begin{array}{cc} 0 & (-1)^p x\partial_x \\ 0 & 0 \end{array}\right)
+ \left(\begin{array}{cc} xd_{N,p-1} & (-1)^p\left(\frac{n+1}{2}-p\right) \\ 0 & xd_{N,p}
\end{array}\right)\right) \left(\begin{array}{c} \w_{p-1} \\ \w_p \end{array}\right).
\end{align*}

Very much in the spirit of \cite{Ver} and \cite{MV}, we decompose the 
de Rham complex into harmonic and non-harmonic subcomplexes. The non-harmonic
subcomplexes are obtained as follows. Let $\psi \in \Omega^p(N,E_N)$ be a 
coclosed $\eta$-eigenform, $\eta >0$, of the twisted Laplacian $\Delta_N$ of the de Rham 
complex  $(\Omega^*(N,E_N), d_{N,*})$. We write
\begin{align*}
\xi_1 = \left(\begin{array}{c} 0 \\  \psi \end{array}\right), \quad
\xi_2 = \left(\begin{array}{c} \psi \\  0 \end{array}\right), \quad
\xi_3 = \left(\begin{array}{c} 0 \\  \frac{1}{\sqrt{\eta}}d_N\psi\end{array}\right), \quad
\xi_4 = \left(\begin{array}{c}  \frac{1}{\sqrt{\eta}}d_N\psi \\ 0 \end{array}\right).
\end{align*}

Repeated application of $d_*$ shows that the subspace $C^\infty_0((R,\infty), 
\langle \xi_1, \xi_2, \xi_3, \xi_4\rangle )$ is preserved under the action of $d_*$
and in fact defines a \emph{non-harmonic} subcomplex
\begin{align*}
0 \rightarrow C^\infty_0((R,\infty), \langle \xi_1 \rangle)
 \xrightarrow{d^0} C^\infty_0((R,\infty), \langle \xi_2, \xi_3 \rangle)
 \xrightarrow{d^1} C^\infty_0((R,\infty), \langle \xi_4 \rangle)
\rightarrow 0,
\end{align*}
where $d^0,d^1$ are the restrictions of $d_*$, given with respect to the 
basis $\langle \xi_1, \xi_2, \xi_3, \xi_4\rangle$ by the following matrix representations
\begin{align*}
&d^0 = \left(\begin{array}{c} (-1)^p x\partial_x + (-1)^p\left(\frac{n+1}{2}-p\right) \\ 
x \sqrt{\eta} \end{array}\right), \\ &d^1 = \left(x \sqrt{\eta} , (-1)^{p+1}x\partial_x 
+ (-1)^{p+1}\left(\frac{n-1}{2}-p\right)\right).
\end{align*}
The Laplacians of the non-harmonic subcomplex are given by the following scalar actions\footnote{Note that the formal 
adjoint of $(x\partial_x)$ in $L^2((R,\infty),dx)$ is given by $(-x\partial_x-1)$.}
\begin{align*}
&\Delta_0=(d^0)^t d^0 = -(x\partial_x)^2 - (x\partial_x) + x^2 \eta + \left(\frac{n}{2}-p\right)^2 -\frac{1}{4}, \\
&\Delta_1=d^1 (d^1)^t = -(x\partial_x)^2 - (x\partial_x) + x^2 \eta + \left(\frac{n}{2}-p-1\right)^2 -\frac{1}{4},
\end{align*}

As a consequence of Poincare duality on $N$, non-harmonic subcomplexes come in pairs. 
The \emph{twin} subcomplex is obtained by replacing $\psi$ by $\psi':=\frac{1}{\sqrt{\eta}} d_N^t *\psi$,
where $*$ is the Hodge star operator of $N$. $\psi' \in \Omega^{n-p-1}(N,E)$ is again an $\eta$-eigenform
of the twisted Laplacian $\Delta_N$ and we may repeat the construction of the associated subcomplex
as above, denoting the corresponding operators with an additional apostrophe. The resulting 
Laplacians of the twin non-harmonic subcomplex are given by\footnote{In contrast to the setting of 
isolated conical singularities in \cite{Ver}, the twin subcomplexes do not lead to simplifying 
cancellations.}
\begin{align*}
\Delta'_0 = \Delta_1, \quad \Delta'_1 = \Delta_0.
\end{align*} 

The harmonic subcomplexes are constructed as follows.  Let $u\in H^p(N,E)$
be a $\Delta_N$-harmonic twisted differential form of degree $p$. The subspace
$C^\infty_0((R,\infty), \langle u\oplus 0, 0 \oplus u \rangle)$ is again invariant
under the action of $d_*$ and defines a subcomplex
\begin{align*}
0 \rightarrow C^\infty_0((R,\infty), \langle 0\oplus u \rangle)
 \xrightarrow{d_H} C^\infty_0((R,\infty), \langle u \oplus 0 \rangle)
\rightarrow 0,
\end{align*}
where $d_H$ is the restriction of $d_*$, given with respect to the basis
$\langle u\oplus 0, 0 \oplus u \rangle$ by 
\begin{align*}
d_H=(-1)^p \left(x\partial_x + \left(\frac{n+1}{2}-p\right)\right).
\end{align*}
The corresponding Laplacians of the harmonic subcomplex are given by
the following scalar action
\begin{align}\label{harmonic}
\Delta^0_H := d^t_H d_H = 
-(x\partial_x)^2 - (x\partial_x) + \left(\frac{n}{2}-p\right)^2 -\frac{1}{4}
= d_H d_H^t =: \Delta^1_H.
\end{align} 

By the Hodge de Rham decomposition of $\Omega^*(N,E)$, the de Rham complex 
$(\Omega^*_0(\U_R,E), d_*)$ decomposes completely into a direct sum of harmonic
and non-harmonic subcomplexes above.\medskip

We close the subsection with a discussion of relative boundary conditions for 
the Hodge Laplacian $\Delta_*$ of $(\Omega^*_0(\U_R,E), d_*)$ at the regular 
end $x=R$. The Hodge Laplacian is essentially self-adjoint
at $x=\infty$ and the boundary conditions at the infinite cusp end of $\U_R$ amount
only to the $L^2$-integrability condition. \medskip

Consider the inclusion $\iota: N\times \{R\} \hookrightarrow \U_R$ with the pullback 
$\iota^*: \Omega^p(\U_R,E) \to \Omega^p(N,E)$ given by $\iota^*(\w_p + \w_{p-1} \wedge dx) = \w_p(x=R)$.
The self adjoint domain of $\Delta_p$ with relative and absolute 
boundary conditions is then given by (cf. \eqref{domains})

\begin{align*}
\dom_{\textup{rel}}(\Delta_p) =& \, \{ \w\in \dom_{\max} (\Delta_p) 
\mid \iota^*\w = 0, \iota^*(d^t\w)=0\} = \{(\w_{p-1}, \w_p) \in \dom_{\max} (\Delta_p)  \mid
\\ &  \, \w_p(R)=0, (\partial_x \w_{p-1})(R) - \frac{1}{R}\left(\frac{n+1}{2}-p\right) \w_{p-1}(R)=0\},
\\ \dom_{\textup{abs}}(\Delta_p) =& \, \{ \w\in \dom_{\max} (\Delta_p) 
\mid \iota^*(*\w) = 0, \iota^*(*d\w)=0\} = \{(\w_{p-1}, \w_p) \in \dom_{\max} (\Delta_p)  \mid
\\ &  \, \w_{p-1}(R)=0, (\partial_x \w_{p})(R) + \frac{1}{R}\left(\frac{n+1}{2}-p\right) \w_{p}(R)=0\}.
\end{align*}

The relative and absolute boundary conditions are compatible with the decomposition 
of the de Rham complex into harmonic and non-harmonic subcomplexes, 
and induce self-adjoint extensions of the Laplacians $\Delta_{i}, \Delta'_{i}, \Delta^i_H, i=0,1$,
which we make explicit in case of relative boundary conditions:

\begin{equation}\label{ops}
\begin{split}
&\dom_{\textup{rel}}(\Delta_0) = 
\{ f\in \dom_{\max} (\Delta_0) \mid f(R)=0\}, \\
&\dom_{\textup{rel}}(\Delta'_0) =
\{ f\in \dom_{\max} (\Delta'_0) \mid f(R)=0\}, \\
&\dom_{\textup{rel}}(\Delta^0_H) = 
\{ f\in \dom_{\max} (\Delta^0_H) \mid f(R)=0\}, \\
&\dom_{\textup{rel}}(\Delta_1) = 
\{ f\in \dom_{\max} (\Delta_1) \mid (\partial_x f - x^{-1}((n-3)/2-p) f)(R)=0\}, \\
&\dom_{\textup{rel}}(\Delta'_1) = 
\{ f\in \dom_{\max} (\Delta'_1) \mid  (\partial_x f + x^{-1}((n+1)/2-p) f)(R)=0\}, \\
&\dom_{\textup{rel}}(\Delta^1_H) = 
\{ f\in \dom_{\max} (\Delta^1_H) \mid  (\partial_x f - x^{-1}((n-1)/2-p) f)(R)=0\},
\end{split}
\end{equation} 
where we point out that elements in the maximal domains of the Laplacians 
$\Delta_{i}, \Delta'_{i}, \Delta^i_H, i=0,1$, are continuously differentiable at $x=R$
by standard arguments. \medskip

Comparing relative and absolute boundary conditions for the individual
scalar operators of the harmonic and non-harmonic subcomplexes, we find
by Poincare duality on the even-dimensional cross section $N$ as expected

$$
T(\U_R, E,N,g) = T(\U_R,E,g).
$$

\begin{remark}\label{LP}
A minor extension of the arguments by Lax and Phillips \cite{LP} on the 
cutoff Laplacian asserts that for the space 
$$\dom^\perp_p := \{f \in \dom_{\textup{rel}}(\Delta_p) \mid 
\forall_{\w\in H^p(N,E)}: (f,\w)_{L^2(N,E)}=0\},$$
the resolvent of the cutoff Laplacian $\Delta_p \restriction \dom^\perp_p$, 
which is precisely the union of all Laplacians of the non-harmonic subcomplexes, 
is compact and hence admits a discrete spectrum. The spectrum 
is strictly positive, since each $D_t$ with parameter $\mu >0$ and $t\geq 0$
are invertible. 
\end{remark}

\section{An integral representation for infinite sums of zeta functions}\label{int-sec}

In this section we establish an integral representation for an infinite sum of 
zeta functions associated to cuspidal operators in the form as they enter the 
definition of the analytic torsion of a model cusp. We are  concerned with zeta-functions 
associated to the following two families of scalar operators for fixed $c>0, R'>R$
\begin{equation}
\begin{split}
&D_c(\mu) := -(x\partial_x)^2 - (x\partial_x) + x^2 \mu^2 + c^2 -\frac{1}{4}: 
C_0^\infty(R,\infty) \to  C_0^\infty(R,\infty), \\
&D'_c(\mu):=-(x\partial_x)^2 - (x\partial_x) + x^2 \mu^2 + c^2 -\frac{1}{4}: 
C_0^\infty(R,R') \to  C_0^\infty(R,R'),
\end{split}
\end{equation}
with the parameter $\mu^2 \in \spec \Delta_{p,\textup{ccl},N}\backslash \{0\}$, 
where $\Delta_{p,\textup{ccl},N}$ denotes the Hodge Laplacian on coclosed
forms degree $p$ over $N$. We consider their self-adjoint
extensions with Dirichlet and generalized Neumann boundary conditions $(\A\in \R)$

\begin{equation}
\begin{split}
&\dom (D_c(\mu)) := \{f \in \dom_{\max}(D_c(\mu)) \mid f(R)=0\}, \\
&\dom (D_c(\mu, \A)) := \{f \in \dom_{\max}(D_c(\mu)) \mid f'(R) +\A f(R) / R =0\}, \\
&\dom (D'_c(\mu)) := \{f \in \dom_{\max}(D_c(\mu)) \mid f(*)=0, *=R,R'\}, \\
&\dom (D'_c(\mu, \A)) := \{f \in \dom_{\max}(D_c(\mu)) \mid 
f'(*) + \A f(*) / * =0, *=R,R'\}
\end{split}
\end{equation}

We have seen in Section \ref{section-det} that the operators 
$D_c(\mu)^{-1}, D_c(\mu, \A)^{-1}$ are trace class with discrete spectrum. The corresponding
observation for $D'_c(\mu)^{-1}, D'_c(\mu, \A)^{-1}$ is classical\footnote{We assume 
$D'_c(\mu)^{-1}$ and $D'_c(\mu, \A)^{-1}$ to be invertible, which is the case in the applications 
below. In fact, since the 
$L^2$-cohomology of the model cusp is entirely determined by the cohomology 
of $N$, the rescaling assumption of Remark \ref{kernel1} is obsolete in the spectral 
geometric applications below.}. Hence for any 
\begin{align}\label{D-operators}
D_c\in \{D_c(\mu), D_c(\mu, \A), D'_c(\mu), D'_c(\mu, \A)\},
\end{align}
we may enumerate its eigenvalues $\spec D_c = \{\lambda_k\mid k\in \N_0\}
\subset (0,\infty)$ in the ascending order. Denote by $m(\lambda_k)$ the 
multiplicity of the eigenvalue $\lambda_k, k\in \N_0$. Then for 
$\Re (s) >1$ (note that $D_c^{-1}$ is trace class) we may write  
\begin{align*}
\zeta(s,D_c) := \sum_{k=0}^\infty m(\lambda_k)\lambda_k^{-s}
&= \mu^{-2s} \sum_{k=0}^\infty m(\lambda_k)(\lambda_k/\mu^2)^{-s}
\\ &= - \frac{\mu^{-2s}}{\Gamma(s)} \int_0^\infty t^{s-1} \frac{1}{2\pi i}
\int_{\Lambda} e^{-\lambda t} \frac{d}{d\lambda} t (\lambda, D_c) d\lambda \, dt, 
\end{align*}
where 
\begin{align*}
t (\lambda, D_c) = -  \sum_{k=0}^\infty m(\lambda_k) 
\log \left(1- \frac{\lambda \mu^2}{\lambda_k}\right), \quad
\frac{d}{d\lambda} t (\lambda, D_c) = - \textup{Tr}\, 
(\lambda - \mu^{-2}D_c)^{-1},
\end{align*}
and $\Lambda = \{\lambda \in \C \mid \arg (\lambda - \gamma) = \pi /4\}$ is
a counter-clockwise oriented integration contour for some $\gamma \in (0,\lambda_0)$. 
\medskip

This integral representation of $\zeta(s,D_c)$ is a consequence of absolute convergence 
of sums for $\Re (s)>1$. Integrating by parts first in $\lambda\in \Lambda$
and then in $t\in (0,\infty)$ yields, cf. Spreafico \cite[Lemma 1]{Spr:ZFA} 
\begin{align}\label{zeta-integral}
\zeta(s, D_c) = \frac{s^2\mu^{-2s}}{\Gamma(s+1)}
\int_0^\infty t^{s-1} \frac{1}{2\pi i} \int_{\Lambda}
\frac{e^{-\lambda t}}{-\lambda} t (\lambda, D_c) d\lambda \, dt.
\end{align}
In view of \cite[Proposition 4.6]{Les:DRS}, which is a general result on zeta-determinants
of scalar operators with discrete spectrum, we find for\footnote{We define square roots using
the main branch of the logarithm in $\C\backslash \R^-$.}
 $z = \sqrt{ - \lambda}$
and $c(\mu z) = \sqrt{c^2+(\mu z)^2}$
\begin{align}\label{t-mu}
t (\lambda, D_c) = - \log \prod\limits_{k=0}^\infty 
\left(1+ \frac{(\mu z)^2}{\lambda_k}\right)^{m(\lambda_k)}
= - \log \frac{\det_{\zeta}D_{c(\mu z)}}{\det_{\zeta}D_c}.
\end{align}

As a direct consequence of Theorem \ref{variation} and
Theorem \ref{variation-Neumann} we obtain 

\begin{prop}\label{t-cusp}
\begin{align*}
t (-z^2, D_c(\mu)) &= \log \left(I'_{c(\mu z)} - 
\frac{I_{c(\mu z)}}{K_{c(\mu z)}} K'_{c(\mu z)}\right) (\mu R) 
- \log \left(I'_{c} -  \frac{I_c}{K_c} 
K'_{c}\right) (\mu R),  \\
t (-z^2, D_c(\mu, \A)) &= \log \left( I_{c(\mu z)} - 
\frac{\mu I'_{c(\mu z)} + \frac{2\A-1}{R} I_{c(\mu z)}}
{\mu K'_{c(\mu z)} + \frac{2\A-1}{R} K_{c(\mu z)}}  
K_{c(\mu z)} \right) (\mu R) \\
&-  \log \left( I_{c} - 
\frac{\mu I'_{c} + \frac{2\A-1}{R} I_{c}}
{\mu K'_{c} + \frac{2\A-1}{R} K_{c}}  
K_{c} \right) (\mu R). 
\end{align*}
\end{prop}

Similar respresentations hold for the operators $D'_c(\mu), D'_c(\mu, \A)$.

\begin{prop}\label{t-cusp-finite}
\begin{align*}
t (-z^2, D'_c(\mu)) &= t (-z^2, D_c(\mu)) - \log \left(I_{c(\mu z)} - 
\frac{I_{c(\mu z)}(\mu R)}{K_{c(\mu z)}(\mu R)} K_{c(\mu z)}\right) (\mu R') 
\\ &+ \log \left(I_{c} -  \frac{I_c(\mu R)}{K_c(\mu R)}  K_{c}\right) (\mu R'),  \\
t (-z^2, D'_c(\mu, \A)) &= t (-z^2, D_c(\mu, \A))  - \log \left( \mu K'_{c(\mu z)} + 
\frac{2\A-1}{R'} K_{c(\mu z)} \right)(\mu R') 
\\ &- \log \left( \frac{\mu I'_{c(\mu z)} + \frac{2\A-1}{R'} I_{c(\mu z)}}
{\mu K'_{c(\mu z)} + \frac{2\A-1}{R'} K_{c(\mu z)}}  (\mu R')
- \frac{\mu I'_{c(\mu z)} + \frac{2\A-1}{R} I_{c(\mu z)}}
{\mu K'_{c(\mu z)} + \frac{2\A-1}{R} K_{c(\mu z)}}  (\mu R)\right) \\
& + \log \left( \frac{\mu I'_{c} + \frac{2\A-1}{R'} I_{c}}
{\mu K'_{c} + \frac{2\A-1}{R'} K_{c}}  (\mu R')
- \frac{\mu I'_{c} + \frac{2\A-1}{R} I_{c}}
{\mu K'_{c} + \frac{2\A-1}{R} K_{c}}  (\mu R)\right) \\
&+ \log \left( \mu K_{c} + 
\frac{2\A-1}{R'} K_{c} \right)(\mu R'). 
\end{align*}
\end{prop}

\begin{proof}
In order to simplify notation we introduce
\begin{align*}
&F_1(x,S):= x^{-1/2} \left(I_t(\mu x) - \frac{I_t(\mu S)}{K_t(\mu S)}K_t(\mu x)\right), \\
&F_2(x,S):=x^{-1/2} \left( I_t(\mu x) - 
\frac{\left(\mu I'_t + \frac{2\A-1}{S} I_t\right)(\mu S)}
{\left(\mu K'_t + \frac{2\A-1}{S} K_t\right)(\mu S)} K_t(\mu x) \right).
\end{align*}
Let $\phi_t, \psi_t$ be solutions to $D_t f=0, D_t\in \{D'_t(\mu), D'_t(\mu, \A)\}$,
where $\phi_t$ satisfies the corresponding boundary conditions at $x=R$,
and $\psi_t$ satisfies the corresponding boundary conditions at $x=R'$.
Assume that $\phi_t, \psi_t$ are normalized such that $\phi'_t(R)=1, \psi'_t(R')=1$,
if $D_t=D'_t(\mu)$, and $\phi_t(R)=1, \psi_t(R')=1$, if $D_t=D'_t(\mu, \A)$.
Then
\begin{align*}
&\phi_t (x) = \frac{F_1(x,R)}{F'_1(R,R)}, \quad \psi_t (x)=\frac{F_1(x,R')}{F'_1(R',R')}, 
\quad \textup{if} \ D_t=D'_t(\mu), \\
&\phi_t (x) = \frac{F_2(x,R)}{F_2(R,R)}, \quad \psi_t (x)=\frac{F_2(x,R')}{F_2(R',R')}, 
\quad \textup{if} \ D_t=D'_t(\mu, \A).
\end{align*}
By the variational formula of Levit-Smilanski \cite{LS}, cf. also 
\cite[Prop. 3.4]{Les:DRS}
\begin{align*}
\frac{d}{dt} \log \det\nolimits_{\zeta} D_t = \frac{d}{dt} \log W(\phi_t, \psi_t).
\end{align*}
The statement now follows from \eqref{t-mu} and the explicit expressions for 
the normalized solutions $\phi_t$ and $\psi_t$.
\end{proof}

We simplify notation below by setting 
\begin{equation}
\begin{split}
&I_t(r,\A) := 1 + \frac{2\A-1}{r} \frac{I_t(r)}{I'_t(r)}, \quad
K_t(r,\A) := 1 + \frac{2\A-1}{r} \frac{K_t(r)}{K'_t(r)}, \\
&c(u)= \sqrt{c^2+u^2}, \qquad  c_0(u)= \sqrt{c_0^2+u^2}.
\end{split}
\end{equation}

As a corollary of Proposition \ref{t-cusp} and \ref{t-cusp-finite}
we obtain an integral representation for a combination of zeta functions,
which will become relevant in the analysis of the contribution of non-harmonic
sub complexes from \S \ref{decomposition-section} to the analytic torsion of the model cusp.

\begin{cor}
For any $c,c_0 >0$ we obtain
\begin{align*}
&\zeta(s, D'_c(\mu, \A)) - \zeta(s, D_c(\mu, \A))
- \zeta(s, D'_{c_0}(\mu)) - \zeta(s, D_{c_0}(\mu)), 
\\ &= \frac{s^2 \mu^{-2s} }{\Gamma(s+1)}
\int_0^\infty t^{s-1} \frac{1}{2\pi i} \int_{\Lambda}
\frac{e^{-\lambda t}}{-\lambda} t (\lambda, \mu) d\lambda \, dt,
\end{align*}
The term $t(\lambda, \mu)$ is given explicitly by $t(-z^2,\mu) = F_{\mu z} - F_0$, 
where we write 
\begin{equation}\label{F}
\begin{split}
F_u := \, &- \log I_{c(u)}(\mu R', \A)+ \log \frac{I_{c_0(u)}}{I_{c(u)}}(\mu R') 
+ \log \frac{I_{c(u)}}{I'_{c(u)}}(\mu R')  
\\  &- \log \left(1- \frac{I'_{c(u)}}{K'_{c(u)}}(\mu R)
\frac{K'_{c(u)}}{I'_{c(u)}}(\mu R') \frac{I_{c(u)}(\mu R, \A)}{K_{c(u)}(\mu R, \A)}
\frac{K_{c(u)}(\mu R', \A)}{I_{c(u)}(\mu R', \A)} \right) 
\\ &+ \log \left( 1- \frac{I_{c_0(u)}}{K_{c_0(u)}}(\mu R)
\frac{K_{c_0(u)}}{I_{c_0(u)}}(\mu R')\right).
\end{split}
\end{equation}
\end{cor}

\begin{remark}\label{gamma-contour}
The operators $\{D_c(\mu, \A), D_{c_0}(\mu)\}$ and 
$\{D'_c(\mu, \A), D'_{c_0}(\mu)\}$, parametrized by $\mu^2 \in \spec  \Delta_{p,\textup{ccl},N} \backslash \{0\}$,
clearly model the Laplacians of the non-harmonic subcomplexes in \S \ref{decomposition-section}
for $\U_R$ and\footnote{$\U^{\circ}_{R'}$ denotes the open interior of $\U_{R'}$.} 
$\U_R \backslash \U^{\circ}_{R'}\cong [R,R'] \times N$. Hence, by Remark \ref{LP} we may assume that 
the union of spectra of all these operators is discrete and fix $\gamma>0$ to be smaller than 
the smallest non-zero spectral element. This fixes the contour
$\Lambda = \{\lambda \in \C \mid \arg(\lambda - \gamma) = \pi / 4\}$.
\end{remark}

We will prove below in Proposition \ref{tmu-expansion-prop} and Proposition 
\ref{F-0} by a detailed analysis of Bessel functions that 
$t(-z^2,\mu)$ admits a uniform asymptotic expansion for $\mu$ going to infinity, with terms of the 
form $a_j(z) (\mu z)^{-j+1}$ and $b_j \mu^{-j}, j \in \N$ and with coefficients $a_j(z)$ 
bounded as $|z|\to \infty$. Assuming existence of $\gamma>0$, bounding all 
eigenvalues from below, we conclude that the infinite sum 
\begin{align*}
\zeta(s)\equiv \zeta(s)[c,c_0, \A,R', \Delta_{p,\textup{ccl},N}]
:= \ &\sum_{\mu} \zeta(s, D'_c(\mu, \A)) - \zeta(s, D_c(\mu, \A))
\\ - \ &\sum_{\mu} \zeta(s, D'_{c_0}(\mu)) - \zeta(s, D_{c_0}(\mu)), 
\\ &\textup{for} \  \Re (s) > \frac{\dim N +1}{2},
\end{align*}
where we sum over $\mu^2 \in \spec  \Delta_{p,\textup{ccl},N} \backslash \{0\}$
is well-defined for $\Re (s)$ sufficiently large. We will see 
below that the difference between logarithms of scalar analytic torsions of the model cusp 
$\U_R$ and the cylinder $\U_R \backslash \U^{\circ}_{R'}\cong [R,R']\times N$ 
is given in terms of derivatives of such combinations, see \eqref{torsion-difference} below. 
Similar to \eqref{zeta-integral} we have an integral representation\footnote{ 
Uniform asymptotics of $t (\lambda, \mu)$ as $\mu \to \infty$ justifies 
integral representation of $\zeta(s)$ for $\Re (s) \gg 0$.}
\begin{align}\label{zeta-non-harmonic}
\zeta(s) = \frac{s^2}{\Gamma(s+1)}
\int_0^\infty t^{s-1} \frac{1}{2\pi i} \int_{\Lambda}
\frac{e^{-\lambda t}}{-\lambda} \sum_{\mu} t (\lambda, \mu)  \mu^{-2s} d\lambda \, dt.
\end{align}

\section{Spreafico's double summation method applied to cusps}\label{spreafico-section}
 
Spreafico's double summation method, cf. \cite{Spr:ZFA} and \cite{Spr:ZIF}, 
provides a powerful tool for studying zeta-functions of infinite sums of scalar operators.
The following theorem is proved in \cite[p. 364 (2)]{Spr:ZFA} and in a more 
general setting in Hartmann-Spreafico \cite[Theorem 3.2]{HS}, cf. also Spreafico 
\cite[Theorem 2.12]{Spr:ZIF}.

\begin{thm}\label{spreafico-theorem}
Consider for $\Re(s) >  \frac{\dim N +1}{2}$ the following holomorphic function
\begin{align}\label{zeta-non-harmonic}
\zeta(s) = \frac{s^2}{\Gamma(s+1)}
\int_0^\infty t^{s-1} \frac{1}{2\pi i} \int_{\Lambda}
\frac{e^{-\lambda t}}{-\lambda} \sum_{\mu} t (\lambda, \mu)  \mu^{-2s} d\lambda \, dt.
\end{align}

\begin{enumerate}
\item Assume $t(\lambda,\mu)$ admits an asymptotic expansion 
\begin{align}
t(\lambda,\mu) \sim \sum_{j=0}^\infty c_{j-1}(\lambda) \mu^{-j+1}, 
\quad \mu \to \infty,
\end{align}
uniformly in\footnote{We define square roots using
the main branch of the logarithm in $\C\backslash \R^-$.} 
$\lambda =(-z^2) \in \Lambda$ with coefficients and the error term 
growing at most polynomially as $|z|\to \infty$. We subtract a finite 
part of the asymptotics to define 
\begin{align}
p(\lambda, \mu) := t(\lambda,\mu) -  \sum_{j=1}^{n/2} c_{2j}(\lambda) \mu^{-2j},
\quad P(\lambda, s) := \sum_{\mu} p(\lambda, \mu) \mu^{-2s}.
\end{align}

\item Assume $p(\lambda, \mu)$ admits an asymptotic expansion as $|\lambda| \to \infty$,
$$
p(\lambda, \mu) \sim a_\mu \log(-\lambda) + b_\mu + 
O(\lambda^{-1/2}), \quad |\lambda|\to \infty,
$$
Assume that for $\Re(s)>0$ sufficiently large the sums
\begin{equation}
\begin{split}
A(s) := \sum_{\mu} a_\mu \mu^{-2s}, \quad B(s) := \sum_{\mu} b_\mu \mu^{-2s} 
\end{split}
\end{equation}
are absolutely convergent holomorphic series and define meromorphic functions on $\C$
with $A(s)$ and $sB(s)$ regular at $s=0$.
\end{enumerate} \ \medskip

Then $P(\lambda, s)$ and $\zeta(s)$ are both regular at $s=0$ and we have the 
following analytic continuation of the latter to a neighborhood of $s=0$
\begin{equation}
\begin{split}
\zeta(s) &= \frac{s}{\Gamma(s+1)} \left(\gamma A(s) - \frac{1}{s} A(s) - B(s) + P(0,s)\right)
\\ &+ \frac{s^2}{\Gamma(s+1)}  \sum_{j=1}^{n/2} \zeta(s+j, \Delta_{p,\textup{ccl},N})
\int_0^\infty \frac{t^{s-1}}{2\pi i}\int_{\Lambda} \frac{e^{-t\lambda}}{(-\lambda)}
c_{2j}(\lambda) d\lambda \, dt \\ &+ \frac{s^2}{\Gamma(s+1)}h(s),
\end{split}
\end{equation}
where $h$ is analytic at $s=0$ and $\gamma$ denotes the Euler-Mascheroni constant. 
\end{thm}

Note that the assumptions of Theorem \ref{spreafico-theorem} amount to the condition 
that the (double) sequence of eigenvalues for the operators in \eqref{D-operators}
employed in the definition of $\zeta(s)$ are \emph{spectrally decomposable} 
over the sequence $\{\mu\}$ in the sense of Spreafico \cite{Spr:ZIF}.

\begin{remark}
The referee justly questions the fact that seemingly the sequence $\{\mu \} = 
\spec \Delta_{p,\textup{ccl},N}\backslash \{0\}$ can be replaced by a sequence of 
eigenvalues of any discrete positive self-adjoint operator $L$. This of course
cannot be the case. The first and obvious condition is that $L$ admits a well-defined 
zeta-function, which is an absolutely convergent series for $\Re(s)>0$ sufficiently 
large, and admits a meromorphic continuation to $\C$ with $s=0$ being a regular
point and simple pole singularities at integer locations. Under this condition, 
$P(\lambda, s)$ is regular at $s=0$. This of course is the generic
setting for self-adjoint differential operators on closed manifolds. \medskip

The other restriction comes from the fact that the eigenvalues of the operators in \eqref{D-operators},
employed in the definition of $\zeta(s)$ and indexed by the eigenvalues of 
$L$, cannot accumulate at zero. This would make a choice of a contour
$\Lambda \subset \C$ impossible and $\zeta(s)$ would not be well-defined any longer.
\medskip

Finally, the assumptions of Theorem \ref{spreafico-theorem} specify 
that the (double) sequence of eigenvalues for the operators in \eqref{D-operators}
employed in the definition of $\zeta(s)$ are \emph{spectrally decomposable} 
over the sequence $\{\mu \}$ in the sense of Spreafico \cite{Spr:ZIF}. This conditions
encodes the necessary setup for the assumptions of Theorem \ref{spreafico-theorem}
to be satisfied. Thus, the choice of $\{\mu\}$ is far from arbitrary.
\end{remark}

\subsection{Asymptotic expansion of $t(-z^2,\mu)$ as $\mu \to \infty$}\label{tmu-expansion}

Write for any $(-z^2)\in \Lambda$ 
\begin{equation}
\begin{split}
&t\equiv t(\mu) = c(\mu z),\quad  s\equiv s(\mu)= \frac{\mu R'}{c(\mu z)}, \\
&\nu(s) \equiv \nu(s) (\mu) = \log \frac{s(\mu)}{1+\sqrt{1+s(\mu)^2}} + \sqrt{1+s(\mu)^2}.
\end{split}
\end{equation}
where we used the notation $c(\mu z) := \sqrt{c^2 + (\mu z)^2}$ from above.

\begin{lemma}\label{gnu-increasing}
For $R'\gg 0$ and $\mu \gg 0$ both sufficiently large, the real part of $t \nu(s)= (t\nu(s))(\mu, R')$
is a strictly increasing function of $R'$.
\end{lemma}

\begin{proof}
We compute (we work with the variable $R$ instead of $R'$ here)
\begin{align}
\frac{d}{dR} (t\nu(s)) (\mu, R) = \frac{\mu R}{\sqrt{1+ 
s(\mu, R)^2}} \left(1- \frac{1}{1+\sqrt{1+ s(\mu, R)^2}}\right)
+ \frac{c(\mu z)}{R}.
\end{align} 
For a given $\delta>0$ with $|\arg(\lambda)| \geq \delta$,
there exists $\delta'>0$ such that 
\begin{align}
|\arg \sqrt{1+s(\mu, R)^2}^{-1}| \leq \frac{\pi}{2}-\delta'.
\end{align}
Moreover, the following qualitative Figure \ref{s} describes the the graph 
of $\sqrt{1+s(\mu, R)^2}^{-1}$ as $(-z^2)$ varies along $\Lambda$.
Here, $\widehat{a}(\mu, R) = \frac{\sqrt{\gamma - (c/\mu)^2}}{ R}$.

\begin{figure}[h]
\begin{center}
\begin{tikzpicture}[scale=1.3]
\draw[->] (-0.5,0) -- (2.2,0);
\draw[->] (0,-1.5) -- (0,1.5);

\draw (0,1) node {$\bullet$};
\draw (0,-1) node {$\circ$};
\draw (1.5,0) node {$\bullet$};
\draw (-0.7,1) node {$\widehat{a}(\mu, R)$};
\draw (-0.8,-1) node {$-\widehat{a}(\mu, R)$};
\draw (1.7,0.2) node {$1$};

\draw (0,1) .. controls (1,1) and (1.5,1) .. (1.5,0);
\draw (0,-1) .. controls (1,-1) and (1.5,-1) .. (1.5,0);

\end{tikzpicture}
\end{center}

\caption{The graph of $\sqrt{1+s(\mu, R)^2}^{-1}$ as $(-z^2)$ varies along $\Lambda$.}
\label{s}
\end{figure} 

Consequently, we may assume
\begin{align*}
|\arg \left( 1- \sqrt{1+s(\mu, R)^2}^{-1}\right) | \leq \frac{\delta'}{2},
\end{align*}
for $R\gg 0$ sufficiently large. From here we easily conclude by studying
the phase of $\frac{d}{dR} (t\nu(s)) (\mu, R)$ that the real part of 
$\frac{d}{dR} (t\nu(s)) (\mu, R)$ is strictly positive for $R\gg 0$
sufficiently large and $|\arg(\lambda)| \geq \delta$. This proves the 
statement for $|\arg(\lambda)| \geq \delta$.\medskip

It remains to study the case $|\arg(\lambda)| < \delta$. Note, as $\mu \to \infty$
\begin{equation}\label{tnu2}
\begin{split}
&(t\nu(s)) (\mu, R)  
= t(\mu)\log \frac{s(\mu)}{1+\sqrt{1+s(\mu)^2}} + t(\mu) \sqrt{1+s(\mu)^2}
\\ &= - (\mu z) \sqrt{1+\left(\frac{c}{\mu z}\right)^2} \log \frac{z}{R} \left(\sqrt{1+\left(\frac{c}{\mu z}\right)^2} +
\sqrt{\frac{z^2+R^2}{z^2}+ \left(\frac{c}{\mu z}\right)^2}\right)
\\ &+ (\mu z) \sqrt{\frac{z^2+R^2}{z^2}} \sqrt{1+ \frac{z^2}{z^2+R^2}\left(\frac{c}{\mu z}\right)^2}  
 \sim \sum_{j=0}^\infty a_j(z,R,c) (\mu z)^{-2j+1},
\end{split}
\end{equation}
for certain coefficients $a_j(z,R,c)$ that are bounded as $|z|\to \infty$
The leading order term in the expansion \eqref{tnu2} will be important
for the argument below and is given explicitly after cancellations by
\begin{equation}\label{tnu3}
\begin{split}
a_0(z,R,c) = \sqrt{\frac{z^2+R^2}{z^2}}
+ \log R - \log \left(z+ \sqrt{z^2+R^2}\right).
\end{split}
\end{equation}
One checks easily that $a_0(z,R,c)$ is a strictly
increasing function of $R$ for $R\gg 0$ sufficiently large and
bounded $z$. This proves the statement for $|\arg(\lambda)| < \delta$
and $\mu, R \gg 0$ both sufficiently large.
\end{proof}

When $(-z^2)$ varies along the contour $\Lambda$, and $\gamma \mu^2 > c^2$, we find that $t$
varies along the contour $\mathbb{T}$ and $s$ along the contour $\mathbb{D}$,
both of which are illustrated in Figure \ref{TD} below.

\begin{figure}[h]
\begin{center}
\begin{tikzpicture}[scale=1.3]
\draw[->] (-0.5,0) -- (2.2,0);
\draw[->] (0,-2.2) -- (0,2.2);
\draw[dashed] (0,0) -- (1,2.2);
\draw[dashed] (0,0) -- (1,-2.2);
\draw[thick] (0,1) .. controls (0.2,1.1) and (0.5,1) .. (1,2.23);
\draw[thick] (0,-1) .. controls (0.2,-1.1) and (0.5,-1) .. (1,-2.23);
\draw (0.5,0) .. controls (0.45,0.2) and (0.4,0.4) .. (0.21,0.49);
\draw (0.7,0.5) node {$\frac{3}{8}\pi$};
\draw (1.5,1.5) node {$\mathbb{T}$};
\draw (0,1) node {$\circ$};
\draw (0,-1) node {$\bullet$};
\draw (-0.5,1) node {$\widehat{t}(\mu)$};
\draw (-0.6,-1) node {$-\widehat{t}(\mu)$};

\draw[->] (3.5,0) -- (6.2,0);
\draw[->] (4,-2.2) -- (4,2.2);
\draw[dashed] (4,0) -- (5,2.2);
\draw[dashed] (4,0) -- (5,-2.2);
\draw[thick] (4,1.5) .. controls (4.5,1.5) and (4.5,1) .. (4,0);
\draw[thick] (4,-1.5) .. controls (4.5,-1.5) and (4.5,-1) .. (4,0);
\draw (4.5,0) .. controls (4.45,0.2) and (4.4,0.4) .. (4.21,0.49);
\draw (4.7,0.5) node {$\frac{3}{8}\pi$};
\draw (5.5,1.5) node {$\mathbb{D}$};
\draw (4,1.5) node {$\bullet$};
\draw (4,-1.5) node {$\circ$};
\draw (3.5,1.5) node {$\widehat{s}(\mu)$};
\draw (3.4,-1.5) node {$-\widehat{s}(\mu)$};

\end{tikzpicture}
\end{center}

\caption{The domains $\mathbb{T}$ and $\mathbb{D}$.}
\label{TD}
\end{figure} \ \medskip

In that figure, $\widehat{t}(\mu) = i \sqrt{\gamma \mu^2 -c^2}$ and $\widehat{s}(\mu) =
i (\frac{R'}{\sqrt{\gamma}} + O(\mu^{-1}))$ as $\mu \to \infty$.
Note that for $\textup{Im} (t(\mu)) > 0$ we have $\textup{Im} (s(\mu))\leq 0$,
and for $\textup{Im} (t(\mu)) < 0$ we have $\textup{Im} (s(\mu))\geq 0$.
We find that for $R'\gg 0$ sufficiently large, the domains $\mathbb{T}$ and $\mathbb{D}$ are covered by 
the domains of the type $(\mathbb{T}_j, \mathbb{D}_j), j=1,2,3,4$, 
defined in \eqref{TD12} and \eqref{TD34}. Consequently, 
the uniform expansions \eqref{uniform} and \eqref{uniform2}
for these particular choices of $t \in \mathbb{T}$ and $s \in \mathbb{D}$, 
when $\gamma \mu^2 > c^2$ and $R'\gg 0$ is sufficiently large. 
\medskip

We henceforth assume $R'\gg R\gg 0$. 
We simplify the presentation below using a notion of a polyhomogeneity order for polynomials.
\begin{defn}
We say that a polynomial $M$ of degree $\deg M=Q$ is of \emph{polyhomogeneity order} $\rho$,
if there exist coefficients $(g_j)_{j=\rho}^Q$, such that $M(x)=\sum_{j=\rho}^Q g_j x^j$.
\end{defn}

\begin{prop}\label{tmu-expansion-prop}
There exist polynomials $(M_j)_{j\in \N}$ with each polynomial $M_j$ of polyhomogeneity
order $j$, and coefficients $a_j(z,R',c), a_j(z,R',c_0)$ that are bounded as $|z|\to \infty$, 
such that for each fixed $R, R'$ with $R' \gg R \gg 0$
\begin{align*}
F_{\mu z} \equiv F_{\mu z} (\mu, z, R, R')&\sim 
\log \frac{R'}{z}\sqrt{\frac{z^2}{z^2+R'^2}} + 
\sum_{j=1}^\infty M_j\left(\sqrt{\frac{z^2}{z^2+R'^2}}\right) 
(\mu z)^{-j} \\ &+ \sum_{j=0}^\infty (a_j(z,R',c_0) - a_j(z,R',c)) 
(\mu z)^{-2j+1}, \quad  \mu \to \infty,
\end{align*}
uniformly in $z$ with $(-z^2)\in \Lambda$.
\end{prop}

\begin{proof}
We discuss the individual terms in the last expression of \eqref{F} with $u=\mu z$.
Write $t_0=c_0(\mu z)$ and $s_0 = \mu R'/t_0$. We find
in view of \eqref{uniform} as $\mu \to \infty$
\begin{align*}
 \log \frac{I_{c_0(\mu z)}}{I_{c(\mu z)}}(\mu R') &\sim 
\frac{1}{2}\log \frac{t \sqrt{1+s^2}}{t_0 \sqrt{1+s^2_0}}
+ t_0\nu(s_0) - t\nu(s) \\ &+ \log \left(1+ \sum_{k=1}^\infty \frac{U_k(p(s_0))}{t_0^k}\right)
- \log \left(1+ \sum_{k=1}^\infty \frac{U_k(p(s))}{t^k}\right).
\end{align*}
Note as $\mu \to \infty$
\begin{align*}
\log \frac{t \sqrt{1+s^2}}{t_0 \sqrt{1+s^2_0}}
= \log \frac{1+ \sum_{j=1}^\infty a_j \left(\sqrt{\frac{z^2}{z^2+R'^2}}
\frac{c}{\mu z}\right)^{2j}}{1+ \sum_{j=1}^\infty a_j \left(\sqrt{\frac{z^2}{z^2+R'^2}}
\frac{c_0}{\mu z}\right)^{2j}} \sim \sum_{j=1}^\infty b_j (c,c_0) 
\left(\sqrt{\frac{z^2}{z^2+R'^2}}\right)^{2j} (\mu z)^{-2j},
\end{align*}
for some coefficients $a_j, b_j(c,c_0), j\in \N$. Note also
\begin{align*}
p(s) = \frac{1}{1+s(\mu)^2} = t(\mu) \frac{\sqrt{\frac{z^2}{z^2+R'^2}} (\mu z)^{-1}}
{\sqrt{1+ \left(\frac{c}{\mu z} \sqrt{\frac{z^2}{z^2+R'^2}}\right)^2}} 
= t(\mu) \sum_{j=0}^\infty N_j\left(\sqrt{\frac{z^2}{z^2+R'^2}}\right) (\mu z)^{-2j-1},
\end{align*}
for some polynomials $(N_j)_{j\in \N}$ with each polynomial $N_j$ of polyhomogeneity
order $2j+1$. We infer from \cite[(7.10)]{Olv:AAS} that the polynomials $U_k(p)$
are of the structure $U_k(p) = \sum_{b=0}^k a_{kb} p^{k+2b}$. Hence 
using the expansion of $p(s)$ above we find
\begin{equation}\label{U-expansion}
\begin{split}
\frac{U_k(p(s))}{t^k} &= \sum_{b=0}^k a_{kb}  t^{2b} \left(\frac{p(s)}{t}\right)^{k+2b}
\\ &\sim \sum_{j=0}^\infty W_j\left(\sqrt{\frac{z^2}{z^2+R'^2}}\right) (\mu z)^{-2j-k},
\quad \mu \to \infty,
\end{split}
\end{equation}
for some polynomials $W_j$ with each polynomial $W_j$ of polyhomogeneity
order $2j+k$. Note also the expansion obtained in \eqref{tnu2}.
Consequently there exist polynomials $(Q_j)_{j\in \N}$, 
with each polynomial $Q_j$ of polyhomogeneity order $j$, such that
\begin{equation}\label{e0}\begin{split}
 \log \frac{I_{c_0(\mu z)}}{I_{c(\mu z)}}(\mu R') &\sim 
\sum_{j=1}^\infty Q_j\left(\sqrt{\frac{z^2}{z^2+R'^2}}\right) (\mu z)^{-j}
\\ &+ \sum_{j=0}^\infty (a_j(z,R',c_0) - a_j(z,R',c)) \mu^{-2j+1}, \quad  \mu \to \infty.
\end{split}\end{equation}

Using \eqref{uniform}, \eqref{uniform2} and \eqref{U-expansion} (which holds for $U_*$
replaced by $V_*$) we find as $\mu \to \infty$
\begin{equation}\label{e1}
\begin{split}
\log \frac{I_{c(\mu z)}}{I'_{c(\mu z)}}(\mu R') &\sim 
\log \frac{R'}{z}\sqrt{\frac{z^2}{z^2+R'^2}} 
-\log \sqrt{1+ \frac{z^2}{z^2+R'^2}\left(\frac{c}{\mu z}\right)^2} 
\\ &+ \log \left(1+ \sum_{k=1}^\infty \frac{U_k(p(s))}{t^k}\right) 
- \log \left(1+ \sum_{k=1}^\infty \frac{V_k(p(s))}{t^k}\right)
\\ &\sim \log \frac{R'}{z}\sqrt{\frac{z^2}{z^2+R'^2}} + \sum_{j=1}^\infty 
B_j \left(\sqrt{\frac{z^2}{z^2+R'^2}}\right) (\mu z)^{-j},
\end{split}
\end{equation}
for some polynomials $(B_j)_{j\in \N}$ with each polynomial $B_j$ of polyhomogeneity
order $j$. Similarly, there exist polynomials $(L_j)_{j\in \N}$ with each polynomial $L_j$ of polyhomogeneity
order $j$, such that
\begin{align}\label{e2}
\log I_{c(\mu z)} (\mu R', \A) \sim \sum_{j=1}^\infty 
L_j \left(\sqrt{\frac{z^2}{z^2+R'^2}}\right) (\mu z)^{-j},
\quad \mu \to \infty.
\end{align}
Similar expansion holds for $\log K_{c(\mu z)} (\mu R', \A)$.
In view of the previous Lemma \ref{gnu-increasing}, we may choose
$R' \gg R \gg 0$ sufficiently large such that as $\mu \to \infty$
\begin{equation}\label{e3}
\frac{I_{c(\mu z)}}{K_{c(\mu z)}}(\mu R)
\frac{K_{c(\mu z)}}{I_{c(\mu z)}}(\mu R')= O(\mu^{-\infty}),
\quad \frac{I'_{c(\mu z)}}{K'_{c(\mu z)}}(\mu R)
\frac{K'_{c(\mu z)}}{I'_{c(\mu z)}}(\mu R') 
= O(\mu^{-\infty}).
\end{equation}
Plugging the expansions \eqref{e1}, \eqref{e2}, \eqref{e3}, 
as well as the expansion obtained in \eqref{e0} into
the expression \eqref{F}, we arrive at the statement of the
proposition.
\end{proof}

\begin{prop}\label{F-0}
There exist $(m_j)_{j\in \N}$ such that for any $Q\in \N$
and fixed $R\gg 0$  
$$F_0 \equiv F_0 (\mu, R, R')\sim  \sum_{j=1}^{Q-1} m_j (\mu R')^{-j} 
+ O(\mu R')^{-Q}, \quad (\mu R') \to \infty.$$
\end{prop}
\begin{proof}
We study the individual summands in the last expression 
of \eqref{F} with $u=0$. Using the expansions \eqref{IK1} and \eqref{IK2}
we find as $\mu \to \infty$
\begin{align*}
&\log \frac{I_{c_0}}{I_{c}}(\mu R') \sim \sum_{j=1}^\infty q_j(\mu R')^{-j}, \
\log \frac{I_{c}}{I'_{c}}(\mu R') \sim \sum_{j=1}^\infty p_j(\mu R')^{-j}, \\
&\log I_{c} (\mu R', \A) \sim \sum_{j=1}^\infty r_j(\mu R')^{-j},
\end{align*}
for certain coefficients $(q_j)_{j\in \N}, (p_j)_{j\in \N}, (p_j)_{j\in \N}$.
Similar expansion holds for $\log K_{c} (\mu R', \A)$.
Moreover, as $\mu \to \infty$
\begin{align*}
&\frac{I_{c}}{K_{c}}(\mu R)
\frac{K_{c}}{I_{c}}(\mu R')= O(e^{-\mu (R'-R)}),
\quad \frac{I'_{c}}{K'_{c}}(\mu R)
\frac{K'_{c}}{I'_{c}}(\mu R') 
= O(e^{-\mu (R'-R)}).
\end{align*}
This proves the statement with $m_j=q_j+p_j-r_j, j\in \N$, since $R'>R$.
\end{proof}

We summarize the results of the both preceeding propositions.
\begin{cor}\label{RR} Assume $R' \gg R \gg 0$ are fixed.
Then there exist polynomials $(M_j)_{j\in \N}$ with each polynomial $M_j$ of polyhomogeneity
order $j$, coefficients $(m_j)_{j\in \N}$ and coefficients $a_j(z,R',c), a_j(z,R',c_0)$ that 
are bounded as $|z|\to \infty$, such that
\begin{align*}
t(-z^2,\mu) &\sim 
\log \frac{R'}{z}\sqrt{\frac{z^2}{z^2+R'^2}} + \sum_{j=1}^\infty \left( M_j\left(\sqrt{\frac{z^2}{z^2+R'^2}}\right) (\mu z)^{-j}
+ m_j(\mu R')^{-j}\right) \\ &+ \sum_{j=0}^\infty (a_j(z,R',c_0) - a_j(z,R',c)) 
\mu^{-2j+1}, \quad  \mu \to \infty,
\end{align*}
uniformly in $z$ with $(-z^2)\in \Lambda$.
\end{cor}

The zeta function $\zeta(s, \Delta_{p,\textup{ccl},N})=\sum \mu^{-2s}, \Re(s) >n/2$,
where we sum over $\mu^2 \in \spec \Delta_{p,\textup{ccl},N} \backslash \{0\}$ according to their 
multiplicity, extends meromorphically to $\C$ with pole singularities at integer locations 
$s=n/2-k, k\in \N_0\backslash \{n/2\}$. Since $n=\dim N$ is even, the asymptotic terms 
$\mu^{-2j}, j=1,...,n/2,$ in the expansion of $t(-z^2,\mu)$ as $\mu \to \infty$, lead 
after summation in $\mu^2 \in \spec \Delta_{p,\textup{ccl},N} \backslash \{0\}$ to an
irregular behaviour of $\zeta(s)$ and hence we define
\begin{align}
p(-z^2, \mu) := t(-z^2,\mu) -  \sum_{j=1}^{n/2} \left( M_{2j}\left(\sqrt{\frac{z^2}{z^2+R'^2}}\right) 
(\mu z)^{-2j} + m_{2j}(\mu R')^{-2j}\right).
\end{align}

\subsection{Asymptotic expansion of $p(\lambda,\mu)$ as $\lambda \to \infty$}

\begin{prop}
$$
p(\lambda, \mu) \sim a_\mu \log(-\lambda) + b_\mu + O(\lambda^{-1/2}), \quad z\to \infty,
$$
where
$a_\mu = -1/2$ and $b_\mu = \log(2/\mu) - F_0 - \sum_{j=1}^{n/2}m_{2j}(\mu R')^{-2j}$. 
\end{prop}
\begin{proof}
We study the asymptotic expansion of the individual terms in
\begin{align}\label{p}
p(-z^2, \mu) = F_{\mu z} - F_0 -  \sum_{j=1}^{n/2} \left( M_{2j}\left(\sqrt{\frac{z^2}{z^2+R'^2}}\right) 
(\mu z)^{-2j} + m_{2j}(\mu R')^{-2j}\right).
\end{align}
Applying \eqref{t1} and \eqref{t2} with either $t=c(\mu z)$ or $t=c_0(\mu z)$\footnote{We point
out that $t=c(\mu z) \sim \mu z + O(z^{-1})$ as $z\to \infty$.}, we find
$$
F_{\mu z} \sim -\log(z) + \log(2/\mu) + O(z^{-1}), \quad z\to \infty.
$$
The statement is now straightforward, since
$F_0$ is independent of $z$.
\end{proof}

\subsection{Analytic continuation of $\zeta(s)$ near $s=0$.}

Following Theorem \ref{spreafico-theorem} we put
\begin{equation}\label{AB}
\begin{split}
P(\lambda, s) &:= \sum_{\mu} p(\lambda, \mu) \mu^{-2s}, \quad 
A(s) := \sum_{\mu} a_\mu \mu^{-2s} = -\frac{1}{2} \zeta(s, \Delta_{p,\textup{ccl},N}), \\
B(s) &:= \sum_{\mu} b_\mu \mu^{-2s} = \zeta(s, \Delta_{p,\textup{ccl},N}) \log 2 
+ \frac{1}{2} \zeta'(s, \Delta_{p,\textup{ccl},N}) \\
&- \sum_{j=1}^{n/2} \zeta(s+j, \Delta_{p,\textup{ccl},N}) \, m_{2j} \, (R')^{-2j} - \sum_{\mu} F_0(\mu) \mu^{-2s}.
\end{split}
\end{equation}
we sum over $\mu^2 \in \spec \Delta_{p,\textup{ccl},N} \backslash \{0\}$ according to their 
multiplicity, and indicate the dependence of $F_0$ on $\mu$ by $F_0=F_0(\mu)$.
Then Spreafico's double summation method in Theorem 
\ref{spreafico-theorem} yields the following analytic 
continuation of $\zeta(s)$ to $s=0$
\begin{equation}\label{zeta-cylinder}
\begin{split}
\zeta(s) &= \frac{s}{\Gamma(s+1)} \left(\gamma A(s) - \frac{1}{s} A(s) - B(s) + P(0,s)\right)
\\ &+ \frac{s^2}{\Gamma(s+1)}  \sum_{j=1}^{n/2} \zeta(s+j, \Delta_{p,\textup{ccl},N})
\int_0^\infty \frac{t^{s-1}}{2\pi i}\int_{\Lambda} \frac{e^{-t\lambda}}{(-\lambda)^{1+j}}
M_{2j}\left(\sqrt{\frac{(-\lambda)}{R'^2-\lambda}}\right) d\lambda \, dt \\
 &+ \frac{s^2}{\Gamma(s+1)}  \sum_{j=1}^{n/2} \zeta(s+j, \Delta_{p,\textup{ccl},N})
\int_0^\infty \frac{t^{s-1}}{2\pi i}\int_{\Lambda} \frac{e^{-t\lambda}}{(-\lambda)}
 m_{2j} (R')^{-2j} d\lambda \, dt \\
&+ \frac{s^2}{\Gamma(s+1)}h(s),
\end{split}
\end{equation}
where $h$ is analytic at $s=0$ and $\gamma$ denotes the Euler-Mascheroni constant. 
We may now prove the following central result of this section.

\begin{thm}\label{comparison}
Let $R\gg 0$. Then\footnote{We point out that the proof of this limit does not require the 
expansion of Corollary \ref{RR} to hold uniformly in $R,R'>0$, but only for any fixed 
$R'\gg R\gg 0$ sufficiently large. The only crucial uniformity requirement with respect to $R'$ is the statement
of Proposition \ref{F-0}.} 
\begin{align*}
\lim_{R'\to \infty} \zeta'(0) \equiv \lim_{R'\to \infty} \zeta'(0)[c,c_0, \A,R', \Delta_{p,\textup{ccl},N}] 
= - \zeta(0,\Delta_{p,\textup{ccl},N}) \log 2.
\end{align*}
\end{thm}
\begin{proof}
We find for any $k\in \N_0$
\begin{align*}
\int_0^\infty \frac{t^{s-1}}{2\pi i}\int_{\Lambda} \frac{e^{-t\lambda}}{(-\lambda)^{1+j}}
&\left(\sqrt{\frac{(-\lambda)}{R'^2-\lambda}}\right) ^{2j+k}d\lambda \, dt 
\\ &= (R')^{-2s-2j} \int_0^\infty \frac{t^{s-1}}{2\pi i}\int_{\Lambda} \frac{e^{-t\lambda}}{(-\lambda)^{1+j}}
\left(\sqrt{\frac{(-\lambda)}{1-\lambda}}\right) ^{2j+k}d\lambda \, dt
\end{align*}
In particular, for $j\geq 1$, these terms and their derivatives, both evaluated at $s=0$, vanish in the limit
$R'\to \infty$. Moreover we note
\begin{align*}
\int_0^\infty \frac{t^{s-1}}{2\pi i}\int_{\Lambda} \frac{e^{-t\lambda}}{(-\lambda)}
d\lambda \, dt = 0,
\end{align*}
since the contour $\Lambda$ does not encircle the origin $\lambda = 0$.
Consequently, the expression \eqref{zeta-cylinder} reduces in the limit to 
\begin{align*}
\lim_{R'\to \infty} \zeta'(0)  
= \lim_{R'\to \infty} \left. \frac{d}{ds} \right|_{s=0}
\frac{s}{\Gamma(s+1)} \left(\gamma A(s) - \frac{1}{s} A(s) - B(s) + P(0,s)\right).
\end{align*}
In view of Proposition \ref{F-0} we find for any integer $Q > n$
\begin{align*}
\sum_{\mu} F_0(\mu) \mu^{-2s} = \sum_{j=1}^Q n_j (R')^{-j} \zeta(s+j/2,\Delta_{p,\textup{ccl},N})
 + (R')^{-Q+1} \w (s,R'),
\end{align*}
where we sum over $\mu^2 \in \spec \Delta_{p,\textup{ccl},N} \backslash \{0\}$ according to their 
multiplicity, and $\w(s,R')$ is analytic at $s=0$ with $\w(0,R'), \w'(0,R')$ both bounded as $R'\to \infty$.
Consequently, in view of \eqref{AB}
\begin{align*}
\lim_{R'\to \infty} \left. \frac{d}{ds} \right|_{s=0} \frac{sB(s)}{\Gamma(s+1)} 
&= \lim_{R'\to \infty} \left. \frac{d}{ds} \right|_{s=0} \frac{s}{\Gamma(s+1)} \left(\zeta(s, \Delta_{p,\textup{ccl},N}) \log 2 
+ \frac{1}{2} \zeta'(s, \Delta_{p,\textup{ccl},N})\right) 
\\ &= \zeta(0, \Delta_{p,\textup{ccl},N}) \log 2 
+ \frac{1}{2} \zeta'(0, \Delta_{p,\textup{ccl},N}).
\end{align*}
Denote for each polynomial $M_{2j}$ the coefficient of the lowest degree term by $M_{2j,0}$.
Then, by construction we find
$$
P(0,s) = - \sum_{j=1}^{n/2} (M_{2j,0} + m_{2j}) (R')^{-2j} \zeta (s+j, \Delta_{p,\textup{ccl},N}).
$$
Hence 
$$
\lim_{R'\to \infty} \left. \frac{d}{ds} \right|_{s=0} \frac{sP(0,s)}{\Gamma(s+1)} = 0.
$$
Finally, a straightforward computation yields
$$
\lim_{R'\to \infty} \left. \frac{d}{ds} \right|_{s=0} \frac{sA(s)}{\Gamma(s+1)}
\left(\gamma - \frac{1}{s}\right) = - A'(0) = \frac{1}{2} \zeta'(0, \Delta_{p,\textup{ccl},N}).
$$
These identities lead in view of \eqref{zeta-cylinder} to the statement of the theorem.
\end{proof}

\section{Renormalized Ray-Singer analytic torsion on a model cusp}\label{thm-sec}
In this section we establish Theorem \ref{scalar-torsion-thm}.

\subsection{Comparison of analytic torsions on a truncated and a full model cusp}

Consider the model cusp $\U_R, R\gg 1$, and the truncated 
cusp $\U_R \backslash \U^{\circ}_{R'}, R'\gg R$, which is a finite cylinder $[R,R']\times N$.
The Riemannian metric $g$ and the flat Hermitian vector bundle $(E,\nabla, h)$
over $\U_R$, restrict to metric structures over $\U_R \backslash \U^{\circ}_{R'}$, and the Definitions
\ref{tors} and \ref{tors-norm} carry over to the special case of a compact 
Riemannian manifold $(\U_R \backslash \U^{\circ}_{R'}, g)$ equipped with a flat 
Hermitian vector bundle. \medskip

We write in each degree $p=0, \ldots, n$
\begin{align*}
&\Delta_{H,p} := - (x\partial_x)^2 - (x\partial_x) + \left(\frac{n}{2} -p\right)^2 -\frac{1}{4}:
C^\infty_0(R,\infty) \to C^\infty_0(R,\infty), \\
&\Delta'_{H,p} := - (x\partial_x)^2 - (x\partial_x) + \left(\frac{n}{2} -p\right)^2 -\frac{1}{4}:
C^\infty_0(R,R') \to C^\infty_0(R,R'),
\end{align*}
and their self-adjoint extensions with Dirichlet and Neumann boundary conditions
\begin{align*}
&\dom (\Delta_{H,p}) = \{f\in \dom_{\max} (\Delta_{H,p}) \mid f(R)=0\}, \\
&\dom (\Delta_{H,p, \textup{Neu}}) = \{f\in \dom_{\max} (\Delta_{H,p}) \mid f'(R)=f(R) ((n-1)/2-p)/R \}, \\
&\dom (\Delta'_{H,p}) = \{f\in \dom_{\max} (\Delta'_{H,p}) \mid f(R)=f(R')=0\}.
\end{align*}

Then, in full analogy with \cite[Remark 4.5 and (4.4)]{Ver}, where a similar decomposition
of the de Rham complex into harmonic and non-harmonic subcomplexes has been employed
in the setup of conical singularities, we have the following identity

\begin{equation}\label{torsion-difference}
\begin{split}
&\log T(\U_R\backslash \U^{\circ}_{R'}, E, N^2, g) - \log T(\U_R, E, N, g)
\\ &= \frac{1}{2} \sum_{p=0}^n (-1)^p \zeta(s)\left[\left|\frac{n}{2}-p-1\right|, 
\left|\frac{n}{2}-p\right|, \left(p-\frac{n-3}{2}\right), R', \Delta_{p,\textup{ccl},N})\right]
\\ &+ \frac{1}{2} \sum_{p=0}^n (-1)^{p+1} \dim H^p(N,E) 
\left(\zeta'(0,\Delta'_{H,p}) - \zeta'(0,\Delta_{H,p})\right)
\\ &- \frac{1}{2} \sum_{p=0}^n (-1)^{p+1} (p+1) \dim H^p(N,E)
\left( \zeta'(0,\Delta_{H,p,\textup{Neu}}) - \zeta'(0,\Delta_{H,p}) \right).
\end{split}
\end{equation}

\begin{prop}\label{harmonic-constribution1}
\begin{align*}
&\sum_{p=0}^n (-1)^{p+1} \dim H^p(N,E) \zeta'(0,\Delta'_{H,p})
\\ &\sim \sum_{p\neq n/2} (-1)^{p} \dim H^p(N,E)
\left(\log (R'/R)^{\left|\frac{n}{2}-p\right|} - \log \left|\frac{n}{2}-p\right|\right)
\\ & + (-1)^{\frac{n}{2}} \dim H^{\frac{n}{2}}(N,E) (\log 2 + \log \log R')+ o(1), 
\ \textup{as} \ R'\to \infty. 
\end{align*}
\end{prop}

\begin{proof}
Consider the rescaling $f\mapsto x^{1/2}f$ that extends to a unitary transformation
$L^2([R,R'], dx) \to L^2([R,R'], x^{-1}dx)$. Under that transformation, $\Delta'_{H,p}$
is unitarily equivalent to a self-adjoint extension of $-(x\partial_x)^2+(n/2-p)^2$ in 
$L^2([R,R'], x^{-1}dx)$ with Dirichlet boundary conditions. Under the change of coordinates
$x=e^r$ we obtain that $\Delta'_{H,p}$ is spectrally equivalent to a self-adjoint extension of
$$
D'_{H,p} = -\frac{d^2}{dr^2} +  \left(\frac{n}{2} -p\right)^2 : C^\infty_0(\log R, \log R') 
\to C^\infty_0(\log R, \log R'),
$$
in $L^2([\log R, \log R'], dr)$ with Dirichlet boundary conditions. 
Let $\phi$ and $\psi$ be solutions to $D'_{H,p} f=0$,
where $\phi$ satisfies the Dirichlet boundary conditions at $r=\log R$,
normalized such that $\phi'(\log R)=1$;
and $\psi$ satisfies the Dirichlet boundary conditions at $r=\log R'$,
normalized such that $\psi'(\log R')=-1$. Then
\begin{equation*}
\phi(r) = \left\{ \begin{split}
&\left|n-2p\right|^{-1} \left(
\frac{e^{\left|\frac{n}{2}-p\right|r}}{e^{\left|\frac{n}{2}-p\right|\log R}}
-\frac{e^{-\left|\frac{n}{2}-p\right|r}}{e^{-\left|\frac{n}{2}-p\right|\log R}}
\right), \quad &p\neq \frac{n}{2}, \\
&r-\log R, \quad &p=\frac{n}{2}.
\end{split}
\right.
\end{equation*}
The solution $\psi$ is obtained similarly by replacing $R$ with $R'$ and
multiplying the expression with $(-1)$. Then, by \cite{BFK:DEB}
(recall $W(\phi,  \psi)$ denotes the Wronski determinant of the fundamental system $\phi, \psi$)
\begin{equation*}
\begin{split}
\zeta'(0,D'_{H,p}) &= \zeta'(0,\Delta'_{H,p}) 
= - \log \frac{\pi W(\phi,  \psi)}{2\Gamma(3/2)^2} = - \log 2\psi(\log R)
\\ &= \left\{
\begin{split}
&\log \left|\frac{n}{2}-p\right|-\log \left((R'/R)^{\left|\frac{n}{2}-p\right|}-
(R/R')^{\left|\frac{n}{2}-p\right|}\right), \quad p\neq \frac{n}{2}, \\
&- \log 2 - \log (\log R' - \log R), \quad p= \frac{n}{2}.
\end{split}\right.
\end{split}
\end{equation*}
The statement is now obvious
\end{proof}

\begin{prop}\label{harmonic-constribution2}
\begin{align*}
&\sum_{p=0}^n (-1)^{p+1} \dim H^p(N,E) \zeta'(0,\Delta_{H,p})
\\ &= \sum_{p\neq n/2} (-1)^{p+1} \dim H^p(N,E) 
\left(\left|\frac{n}{2}-p\right| \log R + \frac{1}{2} 
\log \left|\frac{n}{2}-p\right| \right).
\end{align*}
\end{prop}

\begin{proof}
As in the proof of Proposition \ref{harmonic-constribution1}, similar to 
\cite[Lemma 10.6]{Pfaff}, we transform 
$\Delta_{H,p}$ to a self-adjoint extension of
$$
D_{H,p} = -\frac{d^2}{dr^2} +  \left(\frac{n}{2} -p\right)^2 : C^\infty_0(\log R, \infty) 
\to C^\infty_0(\log R, \infty),
$$
in $L^2([\log R, \infty), dr)$ with Dirichlet boundary conditions. Its resolvent kernel is given 
explicitly by
$$
(D_{H,p}+z^2)^{-1}(r,r') = \frac{1}{2k} e^{-k|r-r'|} - \frac{1}{2k}e^{-k(r+r'-2\log R)},
$$
where we have set $k=\sqrt{\mu_p^2 + z^2}, \mu_p = (n/2-p)$ and fixed $z>0$.
We compute
\begin{align*}
\int_{\log R}^{\log R'} (D_{H,p}+z^2)^{-1}(r,r) dr
= \frac{1}{2k} \log (R'/R) + \frac{1}{4k^2} ((R'/R)^{-2k}-1).
\end{align*}
The renormalized trace $\textup{Tr}_{r} (D_{H,p}+z^2)^{-1}$ is defined 
as the constant term in the expansion of the above expression as $R'\to \infty$.
Hence we find 
\begin{align*}
\textup{Tr}_{r} (D_{H,p}+z^2)^{-1} = - \frac{\log R}{2k} - \frac{1}{4k^2}.
\end{align*}
According to the Definition \ref{det-resolvent} we have
\begin{align*}
\zeta'(0, \Delta_{H,p}) &\equiv \zeta'(0, D_{H,p}) 
= 2 \regint_0^\infty z \, \textup{Tr}_{r} (D_{H,p}+z^2)^{-1} dz.
\end{align*}
Straightforward computations lead to the following formulae
\begin{equation*}
\zeta'(0, \Delta_{H,p}) = \left\{
\begin{split}
& \left|\frac{n}{2}-p\right| \log R + \frac{1}{2} \log \left|\frac{n}{2}-p\right|,
\quad & p\neq \frac{n}{2},\\ & \quad 0,  & p= \frac{n}{2}.
\end{split} \right.
\end{equation*}
\end{proof}

\begin{prop}\label{harmonic-contribution3} 
\begin{align*}
&\sum_{p=0}^n (-1)^{p+1} (p+1) \dim H^p(N,E)
\left( \zeta'(0,\Delta_{H,p,\textup{Neu}}) - \zeta'(0,\Delta_{H,p}) \right)
\\ &= \sum_{p\neq n/2} (-1)^{p+1} \dim H^p(N,E) \left|\frac{n}{2}-p\right| \log 
\left(2\left|\frac{n}{2}-p\right| \right).
\end{align*}
\end{prop}
\begin{proof}
We continue under the notation of Proposition \ref{harmonic-constribution2}.
The Laplacian $\Delta_{H,p,\textup{Neu}}$ with generalized Neumann boundary
conditions transforms to a self-adjoint extension of $D_{H,p}$ in $L^2([\log R, \infty), dr)$ with boundary
conditions $f'(\log R) = \mu_p f(\log R)$. Its resolvent kernel is given explicitly
by, cf. \cite[Theorem 5.3]{Der}
\begin{align*}
(D_{H,p,\textup{Neu}}+z^2)^{-1}(r,r') = \frac{1}{2k} e^{-k|r-r'|} 
+ \frac{1}{2k} \frac{k-\mu_p}{k+\mu_p}e^{-k(r+r'-2\log R)}.
\end{align*}
As in the previous proposition we find
\begin{align*}
\textup{Tr}_{r} (D_{H,p,\textup{Neu}}+z^2)^{-1} 
= - \frac{\log R}{2k} + \frac{1}{4k^2} \frac{k-\mu_p}{k+\mu_p}.
\end{align*}
Straightforward computations lead to the following formulae
\begin{equation*}
\zeta'(0,\Delta_{H,p,\textup{Neu}}) - \zeta'(0,\Delta_{H,p}) = \left\{
\begin{split}
-&\log  2 |\mu_p|,
\quad & p < \frac{n}{2},\\ & \log  2 |\mu_p|,
\quad & p > \frac{n}{2},\\ & \quad 0,  & p= \frac{n}{2}.
\end{split} \right.
\end{equation*}
\end{proof}

\begin{cor}\label{scalar-expansion}
\begin{align*}
\log T(\U_R\backslash \U^{\circ}_{R'}, &E, N^2, g)  \sim \log T(\U_R, E, N, g) 
 + \frac{(-1)^{\frac{n}{2}}}{2} \dim H^{\frac{n}{2}}(N,E) (\log 2 + \log \log R')
\\ &+ \sum_{p\neq n/2} \frac{(-1)^{p}}{2} \dim H^p(N,E)
\left(\left|\frac{n}{2}-p\right| \log R' - \frac{1}{2}\log \left|\frac{n}{2}-p\right|\right)
\\ &+ \sum_{p\neq n/2} \frac{(-1)^{p}}{2} \dim H^p(N,E) \left|\frac{n}{2}-p\right| \log 
\left(2\left|\frac{n}{2}-p\right| \right).\\ &+ o(1), \quad R'\to \infty.
\end{align*}
\end{cor}

\begin{proof}
The statement follows from plugging in the results from the 
Propositions \ref{comparison}, \ref{harmonic-constribution1} and \ref{harmonic-constribution2}
into \eqref{torsion-difference}, and the fact that for an even dimensional Riemannian manifold $(N,g^N)$
$$
\sum_{p=0}^n (-1)^p \zeta(s,\Delta_{p,\textup{ccl},N}) \equiv 0.
$$
\end{proof}

\subsection{Metric anomaly and final result for analytic torsion on a model cusp}

The final step in our argument leading up to a formula for the renormalized analytic
torsion on a model cusp is the explicit computation of the Ray Singer analytic 
torsion on $(\U_R\backslash \U^{\circ}_{R'}, E, g)$ in terms of the Br\"uning-Ma 
metric anomaly. Consider a Riemannian metric on $\U_R\backslash \U^{\circ}_{R'} \cong [R,R']\times N$
$$
g_0 = dx^2 + g^N, x\in [R,R'].
$$
Then, by \cite[Theorem 0.1]{BM} we find as in\footnote{On oriented 
compact Riemannian manifolds of odd dimension, the Ray-Singer analytic torsions
for relative and absolute boundary conditions coincide by Poincare duality.} \eqref{BM-thm}
\begin{align*}
\log \left(\frac{\|\cdot
\|^{RS}_{(\U_R\backslash \U^{\circ}_{R'},E, N^2, g)}}
{\|\cdot \|^{RS}_{(\U_R\backslash \U^{\circ}_{R'},E, N^2, g_0)}}
\right) &=\frac{\textup{rank}(E)}{(-2)} \left[\int_{N\times \{R\}}
B\left(\nabla^{T\U_R\backslash \U^{\circ}_{R'}}_{g}\right) + \int_{N\times \{R'\}}
B\left(\nabla^{T\U_R\backslash \U^{\circ}_{R'}}_{g}\right)\right]
\\ &=\frac{\textup{rank}(E)}{(-2)} \left[\int_{N\times \{R\}}
B\left(\nabla^{T\U_R}_{g}\right) + \int_{N\times \{R'\}}
B\left(\nabla^{T\U_R\backslash \U^{\circ}_{R'}}_{g}\right)\right],
\end{align*}
where the subindex indicates dependence on the cusp metric $g$.
The metric anomaly term $B(\nabla^{T\U_R\backslash \U^{\circ}_{R'}}_{g})$
is invariant under scaling of the Riemannian metric, cf. \cite[Proposition 4.1]{MV}.
Hence we may study $B(\nabla^{T\U_R\backslash \U^{\circ}_{R'}}_{g})$ 
at the boundary component at $N\times \{R'\}$ using the rescaled
metric (we write $y=R'-x$)
$$
(R')^2 g = \left(\frac{R'}{R'-y}\right)^2 (dy+g^N) =: f(y, R') (dy^2+g^N), 
\quad y\in [0,\varepsilon),
$$
where $f(0, R')=1$ and $f'(0, R')= 2(R')^{-1}$. \medskip

Our next argument requires some additional notation. 
Let $\mathscr{A}$ and $\mathscr{B}$ be two $\Z_2$-graded algebras with identity and denote by
$\mathscr{A}\widehat\otimes\mathscr{B}$ their $\Z_2$-graded tensor
product. We identify $\mathscr{A}$ with $\mathscr{A}\widehat\otimes I$ and write
$\widehat{\mathscr{B}}:=I\widehat\otimes\mathscr{B}$. Moreover we put $\wedge:=\widehat
\otimes$ so that $\mathscr{A}\widehat\otimes\mathscr{B}=\mathscr{A}\wedge
\mathscr{B}$. \medskip

Let $R^{TN}$ be the curvature tensor of $(N, g^{N})$ and 
denote by $\{e_k\}_{k=1}^{m-1}$ a local orthonormal frame field on $(N,g^{N})$.  
Let $e_m$ denote the inward-pointing unit normal vector at every boundary 
point of $N\times \{R'\}$. Let $\{e^*_k\}_{k=1}^m$ be the dual orthonormal 
frame field of $T^*\U_R\backslash \U^{\circ}_{R'}$ and let 
$\widehat{e^*_k}$ be the canonical identification with the element $e^*_k$
of $\widehat{\Lambda T^*\U_R\backslash \U^{\circ}_{R'}}$. Let $j\colon N=N\times \{R'\}
\hookrightarrow \U_R\backslash \U^{\circ}_{R'}$ be the canonical embedding. Then 
\cite[(1.15)]{BM} defines
\begin{equation}\label{RS3}
\begin{split}
\dot{R}^{TN} &:= \frac{1}{2}\sum_{1\le k,j\le m-1}\langle e_k, R^{TN} 
e_j\rangle \widehat{e^*_k} \wedge \widehat{e^*_j} \in \Lambda T^*N \, \widehat{\otimes} 
\widehat{\Lambda T^*N} \\
\dot{S} &:=\frac{1}{2}j^*\nabla^{T\U_R\backslash \U^{\circ}_{R'}}\widehat{e_m^*}
\in \Lambda T^*N\widehat\otimes \widehat{\Lambda T^*N}.
\end{split}
\end{equation}
$\dot{R}^{TN}$ and $\dot{S}^2$ are both homogeneous of degree two. 
While $\dot{R}^{TN}$ encodes the curvature of $(N,g^N)$, $\dot{S}$ measures the deviation of $g$ from a metric 
product structure near the boundary. By \cite[(4.39)]{BM} $\dot{S}$ is given explicitly by
\begin{align}\label{S-explicit}
\dot{S}=\frac{1}{4}f'(0, R') \sum_{k}e^*_k \wedge \widehat{e^*_k} =
-\frac{1}{4R'} \sum_{k}e^*_k \wedge \widehat{e^*_k}=: -\frac{\dot{S}_0}{4R'}.
\end{align}
where we used the fact that $f'(0, R')=-2(R')^{-1}$.
The final ingredient in the construction is the Berezin integral (see (\cite[Section 1.1]{BM})
\begin{align}
\int^{B_{N}}:\Lambda T^*N \, \widehat{\otimes} \, 
\widehat{\Lambda T^*N} \to \Lambda T^*N,
\end{align}
which is non-trivial only on elements which are homogeneous of degree $(m-1)$. The 
secondary class $B(\nabla^{T\U_R\backslash \U^{\circ}_{R'}})$, 
introduced in  \cite[(1.17)]{BM} is then defined by
\begin{align}\label{RS2}
B(\nabla^{T\U_R\backslash \U^{\circ}_{R'}}) :=\int^{B_{N}} \exp \left(-\frac{1}{2}\dot{R}^{TN}\right) 
\sum_{k=1}^{\infty} \frac{(-\dot{S}^2)^k}{4k\Gamma(k+1)}.
\end{align}
Recall that the Berezin integral is non-trivial only on elements of degree $(m-1)$.
Hence, despite an infinite sum in the formula \eqref{RS2}, only a finite number
of summands yield a non-trivial contribution to $B(\nabla^{T\U_R\backslash \U^{\circ}_{R'}})$.
Consequently, in view of the $(R'^{-1})$ factor in \eqref{S-explicit}, and the fact that the sum in 
\eqref{RS2} starts with $k=1$, the anomaly term 
$B(\nabla^{T\U_R\backslash \U^{\circ}_{R'}})$
vanishes in the limit as $R' \to \infty$ and hence 
\begin{align}\label{BM-limit}
\lim_{R'\to \infty}\log \left(\frac{\|\cdot
\|^{RS}_{(\U_R\backslash \U^{\circ}_{R'},E, N^2, g)}}
{\|\cdot \|^{RS}_{(\U_R\backslash \U^{\circ}_{R'},E, N^2, g_0)}}
\right)=\frac{\textup{rank}(E)}{(-2)} \int_{N\times \{R\}}
B(\nabla^{T\U_R}_{g}).
\end{align}

\begin{thm}
\begin{align*}
\log T(\U_R, E, N, g) &= \frac{\textup{rank}(E)}{(-2)} \int_{N\times \{R\}}
B(\nabla^{T\U_R}_{g}) + \sum_{p\neq n/2}  \frac{(-1)^{p+1}}{4}\dim H^p(N,E) 
\log \left|\frac{n}{2}-p\right| \\ &+ \sum_{p\neq n/2} \frac{(-1)^{p+1}}{2} \dim H^p(N,E)
\left|\frac{n}{2}-p\right| \log R \\
&+ \sum_{p\neq n/2} \frac{(-1)^{p+1}}{2} \dim H^p(N,E) \left|\frac{n}{2}-p\right| \log 
\left(2\left|\frac{n}{2}-p\right| \right).
\end{align*}
\end{thm}

\begin{proof}
For any $\w \in H^p(N,E)\cong H^p(\U_R\backslash \U^{\circ}_{R'},E)$ we compute
\begin{equation}\label{det-norms}
\frac{\|\w \|^2_{g}}{\|\w \|^2_{g_0}} 
= \frac{\int_R^{R'} x^{-(n+1)+2p} \|\w\|^2_{g^N}dx}
{\int_R^{R'} \|\w\|^2_{g^N}dx}
= \left\{
\begin{split}
&\frac{(R')^{2p-n} - R^{2p-n}}{(2p-n)(R'-R)}, \quad & p\neq \frac{n}{2}, \\
&\frac{\log R' - \log R}{(R'-R)}, \quad & p= \frac{n}{2}.
\end{split}
\right.
\end{equation}
Consequently
\begin{equation}\label{det-expansion}
\begin{split}
\log \frac{\| \cdot \|_{\det H^*(\U_R\backslash \U^{\circ}_{R'},E,N^2,g)}}
{\| \cdot \|_{\det H^*(\U_R\backslash \U^{\circ}_{R'},E,N^2,g_0)}}
&= \sum_{p=0}^n \frac{(-1)^{p+1}}{2} \left(\left|\frac{n}{2}-p\right|-1\right)
\dim H^p(N,E) \log R' \\ &+ \sum_{p=0}^n  \frac{(-1)^p}{2}\dim H^p(N,E) 
\left|\frac{n}{2}-p\right| \log R \\ & + \sum_{p\neq n/2} \frac{(-1)^{p}}{2} 
\dim H^p(N,E) \log \left(2\left|\frac{n}{2}-p\right| \right)
\\&+ \frac{(-1)^{\frac{n}{2}+1}}{2} 
\dim H^{\frac{n}{2}}(N,E) \log \log R' + o(1), \quad R'\to \infty.
\end{split}
\end{equation}
The product rule for the scalar analytic torsion implies 
\begin{align}\label{torsion-product}
\log T(\U_R\backslash \U^{\circ}_{R'}, E, N^2, g_0) =
\frac{1}{2} \chi(N,E) \log 2(R'-R).
\end{align}
Note that
\begin{align*}
&\log \frac{\|\cdot \|^{RS}_{(\U_R\backslash \U^{\circ}_{R'},E, N^2, g)}}
{\|\cdot \|^{RS}_{(\U_R\backslash \U^{\circ}_{R'},E, N^2, g_0)}}
 -\log T(\U_R, E, N, g) \\ &= \log T(\U_R\backslash \U^{\circ}_{R'}, E, N^2, g)  - \log T(\U_R, E, N, g) 
\\ &- \log T(\U_R\backslash \U^{\circ}_{R'}, E, N^2, g_0) 
+ \log \frac{\| \cdot \|_{\det H^*(\U_R\backslash \U^{\circ}_{R'},E,N^2,g)}}
{\| \cdot \|_{\det H^*(\U_R\backslash \U^{\circ}_{R'},E,N^2,g_0)}}
\end{align*}
Hence by Corollary \ref{scalar-expansion}, \eqref{BM-limit}, 
\eqref{det-expansion} and \eqref{torsion-product} we find 
\begin{align*}
\log T(\U_R, E, N, g) &\sim  \frac{\textup{rank}(E)}{(-2)} \int_{N\times \{R\}}
B(\nabla^{T\U_R}_{g}) + \sum_{p\neq n/2}  \frac{(-1)^{p+1}}{4}\dim H^p(N,E) 
\log \left|\frac{n}{2}-p\right| \\ &+ \sum_{p\neq n/2} \frac{(-1)^{p+1}}{2} \dim H^p(N,E)
\left|\frac{n}{2}-p\right| \log R \\
&+ \sum_{p\neq n/2} \frac{(-1)^{p+1}}{2} \dim H^p(N,E) \left|\frac{n}{2}-p\right| \log 
\left(2\left|\frac{n}{2}-p\right| \right) \\ &+ o(1), \quad R'\to \infty.
\end{align*}
Since $\log T(\U_R, E, N, g)$ is independent of $R'$, the statement follows by taking 
the limit as $R'\to \infty$.
\end{proof}

\section{Gluing formula for analytic torsion on 
non-compact manifolds}\label{gluing-section}

In this section we consider a non-compact orientied odd-dimensional Riemannian manifold $(M,g)$
with $M=K\cup_N \U$. We do not assume a specific structure of $g$ here, but pose the
Assumptions \ref{assum1}, \ref{assum2} and \ref{assum3} instead. \medskip

The fundamental proof
strategy is due to Lesch \cite{Les} and Pfaff \cite{Pfaff}, where the former proved a gluing formula on compact
possibly singular manifolds, and the latter extended the argument to non-compact hyperbolic
spaces using a sequence of compact manifolds that in some sense approximates the non-compact
space. \medskip

Certain aspects of their argument need to be reproved here due to different assumptions,
which apply to a larger class of non-compact spaces. Moreover, in contrast to \cite{Les} and \cite{Pfaff}
we need formulate the gluing formula in terms of the Ray-Singer norms on the determinant lines.  

\subsection{Stability of analytic torsion under compactly supp. metric variations}
For any $R>1$, we embed $\U_R\subset M$ in an obvious way 
and write $M_R:= M\backslash \U^{\circ}_R$.
Below we omit the lower index $p$ for the Hodge Laplacian and the heat kernel, 
if we refer to their actions on differential forms in all degrees. \medskip

Consider now for $R>1$ the closed double manifold $\widetilde{M}_R:= M_{R+1} \cup_N (-M_{R+1}),$
where $(-M_{R+1})$ denotes a second copy of $M_{R+1}$ with reversed orientation. Choose
a metric on $\widetilde{M}_R$ that coincides with $g$ on both copies of $M_{R}\subset \widetilde{M}_R$
and is product in a tubular neighborhood $N\times (R+1/2, R+3/2) \subset \widetilde{M}_R$ of the
join $N\times \{R+1\}=\partial M_{R+1}$. The vector bundle $(E,\nabla, h)$ yields a flat 
Hermitian vector bundle over $\widetilde{M}_R$ in a canonical way, which we denote by the same letter again.
We denote the heat kernel of the Hodge Laplacian on $\widetilde{M}_R$ by $\mathscr{H}(\widetilde{M}_R)$.
\medskip

The following result, due to Pfaff \cite[Proposition 11.2]{Pfaff} in the context of hyperbolic manifolds, 
does not use any specific geometry of $M$ and holds in the general setting of the present discussion.

\begin{prop}\label{app}
For any differential operator $P$ acting on sections of $\Lambda^*T^*M\otimes E$ and a 
fixed $R>1$ there exist constants $C,c>0$, such that for all $k\in \N_0$, $R'>R$,
$(t,q,q') \in \R^+ \times M_R^2$ and $P$ acting on the first spacial variable $q\in M_R$
$$
\| \partial_t^k P (\mathscr{H} - \mathscr{H}(\widetilde{M}_{R'}))(t,q,q') \| \leq C e^{-c|R'-R|/t}. 
$$
\end{prop}

We can now prove invariance of the renormalized Ray-Singer analytic torsion norm 
under compactly supported metric variations, as stated in Theorem \ref{invariance}.
We employ a special case of Proposition \ref{diff} which we state and prove in the 
next subsection.

\begin{cor}\label{metric-variation}
Let $g_\theta, \theta \in \mathbb{S}^1,$ denote a smooth family of Riemannian metrics on $M$,
with closure of $\supp \partial_\theta g_\theta$ compact in $M$. Then
$$\frac{d}{d\theta} \|\cdot \|^{RS}_{(M,E,g_\theta)} = 0.$$
\end{cor}

\begin{proof}
Denote the Hodge Laplacian corresponding to the 
Riemannian metric $g_\theta$ by $\Delta_\theta$.
The one-parameter family $\Delta_\theta$ fits into 
the framework of Assumption \ref{assum2} and hence
Proposition \ref{diff} applies to the associated heat kernel family, 
which we denote by $\mathscr{H}_\theta$. We find 
\begin{align*}
\frac{d}{d\theta} \textup{Tr}_r \mathscr{H}_\theta 
=  \textup{Tr} \frac{d}{d\theta} \mathscr{H}_\theta 
= -t \textup{Tr} \left( \left(\frac{d}{d\theta} \Delta_\theta \right) \mathscr{H}_\theta \right),
\end{align*}
where the second equality follows by a standard argument, cf. \cite[(14.3)]{Pfaff} and also \cite{MV, MazVer}.
Indicate the action of the Hodge Laplacian and the heat kernel on differential forms in degree $p$ by the lower 
index $p$ again. Then we find by Proposition \ref{app} and the classical argument of 
Ray-Singer \cite{RS} on closed manifolds ($\A_{p,\theta} = *_\theta^{-1} \frac{d}{d\theta} *_\theta \restriction 
\Omega^*(\widetilde{M}_R,E)$)
\begin{align*}
\frac{d}{d\theta} \sum_{p=0}^m (-1)^p \, p \,  \textup{Tr}_r \mathscr{H}_{p,\theta}
&= - t \sum_{p=0}^m (-1)^p \, p \, \textup{Tr} \left( \left(\frac{d}{d\theta} 
\Delta_\theta \right) \mathscr{H}_{p,\theta}\right)
\\ &= - t \sum_{p=0}^m (-1)^p \, p \, \lim_{R\to \infty}\textup{Tr} 
\left( \left(\frac{d}{d\theta} \Delta_\theta \right) \mathscr{H}_{p,\theta} (\widetilde{M}_R) \right)
\\ &= - t \frac{d}{dt} \sum_{p=0}^m (-1)^p \lim_{R\to \infty} \textup{Tr} 
\left(\A_{p,\theta} \mathscr{H}_{p,\theta} (\widetilde{M}_R) \right)
\\ &= - t \frac{d}{dt} \sum_{p=0}^m (-1)^p \textup{Tr} 
\left(\A_{p,\theta} \mathscr{H}_{p,\theta}\right) 
\end{align*} 
Using the 
notation set in Proposition \ref{diff} with additional upper indices $p$
indicating the degree, we write
\begin{align*}
&2 \frac{d}{d\theta} \log T(M,E,g_\theta) \\ &= 
\frac{d}{d\theta} \left. \frac{d}{ds} \right|_{0} 
\int_0^1 \frac{t^{s-1}}{\Gamma(s)} \sum_{p=0}^m (-1)^p \, p \,  \left( \textup{Tr}_r \mathscr{H}_{p,\theta}
- \sum_{j=0}^\ell \sum_{i=0}^{i_j} b^p_{ij}(\theta) \, t^{\A_j} \log^i (t) - b^p_0(\theta) \right) \, dt
\\ &+ \frac{d}{d\theta} \left. \frac{d}{ds} \right|_{0} 
\frac{1}{\Gamma(s)}\regint_0^1 \sum_{p=0}^m (-1)^p \, p \,  
\left(\sum_{j=0}^\ell \sum_{i=0}^{i_j} b^p_{ij}(\theta)t^{s+\A_j-1}
 \, \log^i (t)\right)\, dt
 \\ &+ \frac{d}{d\theta} \left. \frac{d}{ds} \right|_{0} 
\int_1^\infty \frac{t^{s-1}}{\Gamma(s)} \sum_{p=0}^m (-1)^p \, p \,  \left( \textup{Tr}_r \mathscr{H}_{p,\theta}
- \sum_{j=0}^d \sum_{i=0}^{i_j} c^p_{ij}(\theta) \, t^{\A_j} \log^i (t) - c^p_0(\theta) \right)\, dt
\\ &+ \frac{d}{d\theta} \left. \frac{d}{ds} \right|_{0} 
\frac{1}{\Gamma(s)}\regint_1^\infty \sum_{p=0}^m (-1)^p \, p \,  
\left(\sum_{j=0}^d \sum_{i=0}^{i_j} c^p_{ij}(\theta)t^{s+\A_j-1}
 \, \log^i (t)\right)\, dt.
\end{align*}
Using differentiability of the asymptotic expansions in Proposition 
\ref{diff} with respect to the parameter $\theta$, we may pass
differentiation in $\theta$ past the first and third integrals above 
and find as in \cite[Proposition 2.4]{Les}
\begin{align*}
\frac{d}{d\theta} \log T(M,E,g_\theta) &= 
\left. \frac{d}{ds} \right|_{0}  
 \frac{s}{\Gamma(s)}\regint_0^\infty t^{s-1} 
 \frac{d}{d\theta} \sum_{p=0}^m (-1)^p \, p \,  \textup{Tr}_r \mathscr{H}_{p,\theta}
 \, dt \\ &=
\left. \frac{d}{ds} \right|_{0}  
 \frac{s}{\Gamma(s)}\regint_0^\infty t^{s} 
 \frac{d}{dt} \sum_{p=0}^m (-1)^{p+1} \textup{Tr} 
\left(\A_{p,\theta} \mathscr{H}_{p,\theta}\right) dt
\\ &=
 \left. \frac{d}{ds} \right|_{0}  
 \frac{s}{\Gamma(s)}\regint_0^\infty t^{s-1} \sum_{p=0}^m (-1)^p \textup{Tr} 
\left(\A_{p,\theta} \mathscr{H}_{p,\theta}\right) dt
\\ &=
\frac{1}{2} \LIM_{t\to \infty}
\sum_{p=0}^m (-1)^p \textup{Tr} \left(\A_{p,\theta} \mathscr{H}_{p,\theta}\right) 
- \frac{1}{2} \LIM_{t\to 0+}
\sum_{p=0}^m (-1)^p \textup{Tr} \left(\A_{p,\theta} \mathscr{H}_{p,\theta}\right) 
\\ &= \frac{1}{2} 
\sum_{p=0}^m (-1)^p \textup{Tr} \left(\A_{p,\theta}\restriction \ker \Delta \right),
\end{align*}
where in the last equality we used \eqref{large-times} in Assumption \ref{assum2} 
and the fact that by the Duhamel principle
the short time asymptotics of $\textup{Tr} \left(\A_{p,\theta} \mathscr{H}_{p,\theta}\right)$ 
does not admit a constant term, since $\supp \A_{p,\theta}$ is compact in the interior of $M$ 
and $m=\dim M$ is odd. 
\end{proof}

\subsection{Proof of a gluing formula following Lesch and Pfaff}\label{Lesch-Pfaff}

The following result follows the outline of \cite[Proposition 14.1]{Pfaff}, cf. 
also the parametrix construction in \cite{Donn}, where however the 
Gaussian estimate is replaced by Assumption \ref{assum2}.

\begin{prop}\label{diff}
Consider the one-parameter family $\mathscr{H}_\theta$ of heat kernels, 
introduced in Assumption \ref{assum2}. Then the difference $(\mathscr{H}_\theta - \mathscr{H})$ 
is trace class, the renormalized trace of $\cH_\theta$ is differentiable in $\theta \in \mathbb{S}^1$ 
and admits asymptotic expansions
\begin{equation}\begin{split}
\textup{Tr}_{r}\cH_\theta(t) &\sim_{t\to 0+} \sum_{j=0}^\ell
\sum_{i=0}^{i_j} b_{ij}(\theta) \, t^{\A_j} \log^i (t) + b_0(\theta) +
O(t^{\varepsilon}), \\ \textup{Tr}_{r}\cH_\theta(t) &\sim_{t\to \infty}
\sum_{j=0}^d \sum_{i=0}^{k_j} c_{ij}(\theta) \, t^{-\beta_j} \log^i (t) + c_0(\theta)
+ O(t^{-\delta}), \end{split} \end{equation} 
which are differentiable in $\theta$.
Moreover we have\footnote{The trace on the right hand side 
of the equality is defined without the regularization.}
\begin{align*}
\frac{d}{d\theta} \textup{Tr}_r \mathscr{H}_\theta =  
\textup{Tr} \frac{d}{d\theta}\mathscr{H}_\theta.
\end{align*}
\end{prop}

\begin{proof}
Consider cutoff functions $\phi_1,\psi_1 \in C^\infty_0(M)$, where $\phi_1, \psi_1\equiv 1$
over $M_{R+1}$, $\supp \phi_1, \supp \psi_1 \subset M_{R+2}$, $\phi_1\equiv 1$ over $\supp \psi_1$
and $\supp d\phi_1 \cap \supp \psi_1 = \varnothing$. Put $\psi_2 := 1-\psi_1$ and fix some 
cutoff function $\phi_2\in C^\infty(M)$ with $\phi_2 \equiv 0$ on an open neighborhood of $M_{R}$, 
such that  $\phi_2\equiv 1$ over $\supp \psi_2$ and $\supp d\phi_2 \cap \supp \psi_2 = \varnothing$. 
We define for $(t,q,q')\in \R^+ \times M^2$
\begin{align*}
P(t,q,q';\theta) := \phi_1(q) \mathscr{H}_\theta (t,q,q') \psi_1(q')
+ \phi_2(q) \mathscr{H} (t,q,q') \psi_2(q') =: P_1 + P_2
\end{align*}
We assume without loss of generality that the compact subset $\mathscr{K}\subset M$
in the notation of Assumption \ref{assum2} is contained in the open interior of $M_R$.  
Then $\Delta_\theta \circ \phi_2 = 
\Delta \circ \phi_2$ since $\Delta_\theta \equiv \Delta$ over $\U_{R}$.
Moreover, $V_\theta$ commutes with $\phi_1$ by assumption and hence, 
writing $\delta$ for some first order derivatives and $D$ for the Gau\ss \, Bonnet 
operator, we compute
\begin{align*}
(\partial_t + \Delta_\theta) P  
&= \left((\delta^2 \phi_1) \mathscr{H}_\theta + 2 (\delta \phi_1) D \mathscr{H}_\theta \right) \psi_1 \\
&+ \left((\delta^2 \phi_2) \mathscr{H} + 2 (\delta \phi_2) D \mathscr{H}\right) \psi_2
=: Q_1 + Q_2 =: Q.
\end{align*}
We now define inductively for each $k\in \N$
\begin{align*}
Q^{k+1}(t,q,q';\theta) &:= Q * Q^k (t,q,q';\theta) \\ &= \int_0^t \int_M 
Q(t-\widetilde{t}, q,\widetilde{q};\theta) Q^k(\widetilde{t},\widetilde{q},q';\theta) d\widetilde{t} 
\, \textup{dvol}_g(\widetilde{q}),
\end{align*}
where in each step the spacial integration is over a compact region $M_{R+2}$,
since $\supp \delta \phi_{1,2} \subset M_{R+2}$. By Assumption \ref{assum2} we find for the 
pointwise traces and any $S\in \N$
$$
\textup{tr} \, Q^{k+1}(t,q,\cdot\, ;\theta) \leq t^{k+S} \frac{\textup{vol}_g(M_{R+2})^k}{k!}
\, f \, \|\textup{tr} \, Q\|^k_{\infty, (0,t_0]\times M_{R+2}^2},
$$
where $f\in L^2(M,E,g,h)$ and the estimate holds unifomly in 
$(t,q,\theta)\in (0,t_0]\times M\times \mathbb{S}^1$.
Consequently the Volterra series
$$
\widehat{Q}(t,q,q';\theta):= \sum_{k=0}^\infty (-1)^k Q^k(t,q,q';\theta),
$$
converges and admits an $L^2_*(M,E,g,h)$-integrable majorant in $q'\in M$,
times a factor $t^S$ for any $S\in \N$, uniformly in $t\in (0,t_0], \theta \in \mathbb{S}^1$ and $q\in M$. 
Note that in fact $\supp \widehat{Q}(t,\cdot\, , q';\theta) \subset M_{R+2}$.
Similar argument applies to the Volterra series with $Q$ replaced by $\partial_\theta Q$,
and hence $\widehat{Q}$ is differentiable in $\theta$ and 
$\partial_\theta \widehat{Q}$ admits a square-integrable majorant, uniformly in the 
parameters $(t,\theta, q)\in (0,t_0] \times \mathbb{S}^1 \times M$. \medskip

The heat kernel $\mathscr{H}_\theta$ is then recovered by (cf. \cite[pp. 48-50]{Pfaff})
\begin{align*}
\mathscr{H}_\theta &= P + P * \widehat{Q}
\\ &= (P_1 + P_2) + (P_1 * \widehat{Q} + P_2 * \widehat{Q}),
\end{align*}
where by existence of a square integrable majorant for $\widehat{Q}(t,q,\cdot \, ; \theta)$,
uniformly $(t,\theta, q)\in (0,t_0] \times \mathbb{S}^1 \times M$, as well as Assumption 
\ref{assum2} (ii) applied\footnote{Recall, $\supp \widehat{Q}(t,\cdot\, , q';\theta) \subset M_{R+2}$ is bounded.} 
to $P_1, P_2$, the integral kernels $(P_j *  \widehat{Q})$ and 
$(P_j * \partial_\theta \widehat{Q}), j=1,2,$ are trace class, and their traces vanish to infinite order as $t\to 0+$. 
In particular we can interchange differentiation and integration and find
\begin{align*}
&\frac{d}{d\theta} \textup{Tr} (P_2 *  \widehat{Q}) = \textup{Tr} \left(P_2 * 
\frac{d}{d\theta} \widehat{Q}\right) = O(t^\infty), \quad t\to 0+.
\\ 
&\frac{d}{d\theta} \textup{Tr} (P_1 *  \widehat{Q}) = \textup{Tr} \left(P_1 * 
\frac{d}{d\theta} \widehat{Q}\right) = O(t^\infty), \quad t\to 0+.
\end{align*}
In particular $(\mathscr{H}_\theta - \mathscr{H})$
is indeed trace class. Hence the renormalized trace exists by Assumption 
\ref{assum1} and its differentiability in the parameter 
follows from smoothness of the one-parameter family 
of kernels $\cH_{\theta}$. \medskip

The statement on existence and differentiability of the
asymptotic expansion of the renormalized trace as $t\to 0+$ 
now follows from Assumption \ref{assum1} and
the fact that $P_1$ by Assumption \ref{assum2} fits into the interior elliptic parametric
calculus, cf. Shubin \cite{Shubin} and hence its trace admits a classical short time 
asymptotic expansion, differentiable in $\theta$ and $t$. The corresponding 
statement on the large times asymptotics follows from Assumptions \ref{assum1}  
and \ref{assum2}. \medskip

Since $P_2$ does not depend on $\theta$ we find
\begin{align*}
\frac{d}{d\theta} \textup{Tr}_r \mathscr{H}_\theta &= 
\frac{d}{d\theta} \textup{Tr} \left(  P_1 + P_1 * \widehat{Q} + P_2 * \widehat{Q} \right)
\\ &=  \textup{Tr} \frac{d}{d\theta} \left(  P_1 + P_1 * \widehat{Q} + P_2 * \widehat{Q} \right)
=  \textup{Tr} \frac{d}{d\theta} \mathscr{H}_\theta.
\end{align*}
\end{proof}

Assume now that the Riemannian metric $g$ is product 
in an open neighborhood of $N\times \{1\}$, which we may do without
loss of generality by Corollary \ref{metric-variation}. Consider the cut manifold
$M^{\textup{cut}} := K \sqcup \U$ with $\partial M^{\textup{cut}} = N^2$, 
obtained from $(M,g)$ by cutting along the $N\times \{1\}$ separating 
hypersurface. The Riemannian metric $g$ induces a Riemannian metric 
on $M^{\textup{cut}}$, which we denote by the same letter again. Similarly, 
the flat Hermitian vector bundle $(E,\nabla,h)$ over $M$ gives rise
to the corresponding flat Hermitian vector bundle over $M^{\textup{cut}}$,
denoted by the same letter again. By assumption, the metric $g$ 
on $M^{\textup{cut}}$ is product near the boundary. \medskip

The main ingredient in the proof of the gluing formula is 
a family of boundary conditions on $M^{\textup{cut}}$, introduced
by Vishik \cite{Vi}. Denote by $\iota$ the obvious embedding of $N\times \{1\}$
into $K$ and $\U$. We define in each degree $p$ for any 
$\theta \in (0,\pi/2)$ 
\begin{align*}
\dom^p_\theta := \{(\w_1, \w_2) \in \Omega^p(K, E) \oplus 
\Omega^p(\U,E) \mid \cos\theta \, \iota^*\w_1 = \sin \theta \, \iota^*\w_2 \}.
\end{align*}
The corresponding exterior derivative $D_\theta$ with domain 
$\dom (D_\theta) := \dom_\theta$ is then gauge transformed
to a family of operators with constant domain. More precisely, consider a 
cutoff function $\phi \in C^\infty_0(M)$ with $\supp \phi \subset N \times 
(1-2\varepsilon, 1+ 2\varepsilon)$ and $\phi \equiv 1$ over $N \times 
(1-\varepsilon, 1+ \varepsilon)$ for $\varepsilon >  0$ sufficiently small
such that $g$ is product over $N \times (1-2\varepsilon, 1+ 2\varepsilon)
\subset M$. We introduce a reflection map across $N \times \{1\}$
$$
S: N \times (1-2\varepsilon, 1+ 2\varepsilon) \to N \times (1-2\varepsilon, 1+ 2\varepsilon),
\ (q,x) \mapsto (q, 2-x)
$$
Consider an open neighborhood 
$$
W:= N \times (1-2\varepsilon, 1]
\sqcup N \times [1, 1+2\varepsilon) \subset M^{\textup{cut}},
$$
of the boundary of the cut manifold. The cutoff function $\phi$ and 
the action $S$ lift to $W$ and hence we may define in each degree $p$
\begin{align*}
&T: \Omega^p(W,E) \to \Omega^p(W,E), \ T(\w_1,\w_2) := (S^*\w_2, S^*\w_1), 
\\ &\Phi_\theta : \Omega^p(W,E) \to \Omega^p(W,E),\ \Phi_\theta 
:= \cos (\theta \phi) \textup{Id} +  \sin (\theta \phi) T.
\end{align*}
Since $\Phi_\theta \w = \w$ for $\w \in \Omega^p(W,E)$ with 
$\supp \w \subset W \backslash \supp \phi$, $\Phi_\theta$
extends in an obvious way to $\Omega^p(M^{\textup{cut}}, E)$ and defines 
a unitary transformation on the corresponding $L^2$ completion. 
As explained in Lesch \cite[Lemma 5.1 and (5.8)]{Les}, the gauge 
transformed family of exterior derivatives
$$
\widetilde{D}_\theta := \Phi_\theta \circ D_{\theta + \pi /4} \circ \Phi^*_\theta
= D_{\pi/4} + \theta e^{d\phi} T,
$$
is defined on a fixed domain $\dom_{\pi/4}$. We obtain a complex 
$(\dom^*_\theta, \widetilde{D}_\theta)$ and denote by $\Delta_{p,\theta}$
the corresponding family of Hodge Laplace operators acting on $\dom^p_{\pi/4}$
in each degree $p$. As before, we denote in each degree $p$ 
the corresponding family of heat kernels by $\cH_{p,\theta}$.
As explained in \cite[page 26]{Les}, 
$\Delta_{p,\theta} = \Delta_p + V_\theta$, where for each $s\in \R$
$$
V_\theta: H^s_{loc}(M, \Lambda^pT^*M \otimes E) \to 
H^{s-1}_{comp}(M,\Lambda^pT^*M \otimes E),
$$
is a smooth family of symmetric operators that map sections that are locally of Sobolev class
$s$ into the space of compactly supported sections of Sobolev class $(s-1)$. 
The operator family $(V_\theta)$ fits into the framework of Assumption 
\ref{assum2} with $\mathscr{K}=\supp d\phi$. In particular, the corresponding analytic torsion 
$T_\theta(M,E)$ in terms of $\Delta_\theta$ is well-defined by Assumption \ref{assum2}.
\medskip

The next result is obtained as a consequence of Propositions \ref{app}
and \ref{diff} by an ad verbatim repetition of Pfaff's argument
in \cite[Proposition 14.3]{Pfaff}, where variations of renormalized traces on $M$ 
are approximated by the corresponding variations on a sequence of compact 
manifolds $\widetilde{M}_R$ (cf. notation in Proposition \ref{metric-variation})
with $R\to \infty$, and \cite[Theorem 5.3]{Les} is applied for each finite $R>1$.

\begin{thm}\label{deriv}
For $\theta \in (0,\pi/2)$ we have 
\begin{align*}
&\frac{d}{d\theta} \sum_{p=0}^m (-1)^p \, p \, \textup{Tr}_{\textup{reg}} (\cH_{p,\theta}(t))
= -t \frac{4}{\sin 2\theta} \frac{d}{dt} \sum_{p=0}^m (-1)^p \textup{Tr} (\beta_\theta 
\cH_{p,\theta}(t)), \\ &\sum_{p=0}^m (-1)^p \textup{Tr} (\beta_\theta 
\cH_{p,\theta}(t)) = \chi(K,E) - \sin^2 \theta \chi(N,E) + O(t^\infty), \ t\to 0+,
\end{align*}
where $\beta_\theta: \dom^*_\theta \to \Omega^*(K,E), \beta(\w_1,\w_2) = \w_2$
is the obvious restriction.
\end{thm}

The final step in the derivation of a gluing formula is the analysis of certain long exact sequences 
in cohomology. We write 
\begin{align*}
&\Omega^*_r(\U, E) := \{\w \in \Omega^*(\U,E) \mid \iota^* \w = 0\}, \\
&\Omega^*_r(K, E) := \{\w \in \Omega^*(K,E) \mid \iota^* \w = 0\}.
\end{align*}
We consider the following short exact sequences of complexes
\begin{align*}
&0 \to \Omega^*_r(\U, E) \xrightarrow{\alpha_\theta} \dom^*_\theta 
\xrightarrow{\beta_\theta} \Omega^*(K,E) \to 0, \\
&0 \to \Omega^*_r(\U,E) \oplus \Omega^*_r(K,E) \xrightarrow{\gamma_\theta}
\dom^*_\theta \xrightarrow{r_\theta} \Omega^*(N,E) \to 0,
\end{align*}
where $\A_\theta\w = (\w, 0)$ is an extension by zero, $\beta_\theta(\w_1,\w_2)=\w_2$
is the restriction to $K$, $\gamma_\theta(\w_1,\w_2) = (\w_1,\w_2)$ is the inclusion
and $r_\theta(\w_1,\w_2) = \sin\theta \iota^*\w_1 + \cos\theta \iota^*\w_2$. 
We denote the corresponding long exact cohomology sequences by $\cH^\theta(\U,K)$
and $\cH^\theta(\U,K,N)$, respectively. The torsions of these long exact sequences 
$\cH^\theta$ are defined combinatorially, see for instance \cite[\S 2.2]{Les}, 
in terms of the induced $L^2$-Hilbert space structure and are denoted by $\tau(\cH^\theta)$.
The following theorem is due to Lesch \cite[Theorem 4.1]{Les} with minor adaptations
to the present setup.

\begin{thm}\label{deriv2}
For $\theta \in (0,\pi/2)$
$$
\frac{d}{d\theta} \log T_\theta(M,E) = \frac{d}{d\theta} \log 
\tau(\cH^\theta(\U,K,N)) ,
$$
Moreover, $\log T_\theta(M,E) - \log \tau(\cH^\theta)$ is differentiable
at $\theta\in [0,\pi/2)$, where $\cH^\theta$ stands for either $\cH^\theta(\U,K)$
or $\cH^\theta(\U,K,N)$.
\end{thm}

\begin{proof}
Following the proof of \cite[Theorem 4.1]{Les} in \cite[\S 5.2]{Les}
it remains in view of Theorem \ref{deriv} to rule out jumps in the dimensions
of the cohomology groups in $\cH^\theta$ at $\theta = 0$, cf. \cite[\S 5.2.3]{Les}.
Here we employ Assumption \ref{assum3}. \medskip

If the spectrum $\spec \Delta_p \backslash \{0\}$ of the Hodge Laplacian $\Delta_p$ admits 
a spectral gap around zero in all degrees $p$, then same holds for the family $\Delta_{p, \theta}$,
since $\Delta_{p,\theta}$ is a relatively compact perturbation of $\Delta_p$ and essential
spectrum is stable under relatively compact perturbations. In this case, the argument 
of \cite[\S 5.2.3]{Les} carries over ad verbatim. \medskip

In case of no spectral gap around zero, Assumption \ref{assum3} requires 
$H^*(N,E) = \{0\}$. Then exactness of $\cH^\theta(\U,K,N)$ implies  
\begin{align*}
H^*(\dom_\theta^*, \widetilde{D}_\theta) &\cong 
H^*(\Omega^*_r(\U, E)) \oplus H^*(\Omega^*_r(K, E))
\\ &= H^*(\U,E,N) \oplus H^*(K,E,N),
\end{align*}
which forces the dimension of $H^*(\dom_\theta^*, \widetilde{D}_\theta)$
to be independent of $\theta\in [0,\pi/2)$.
\end{proof}

As a consequence of Theorem \ref{deriv2}, the gluing formula 
\cite[Theorem 6.1]{Les} follows ad verbatim and we state it here as follows.

\begin{thm}\label{gluing-thm}
Writing $\cH := \cH^{\pi/4}$ we find
\begin{align*}
&\log T(M,E) = \log T(K,E) + \log T(\U,E,N) + 
\log \tau(\cH(\U,K)) - \chi(N,E) \log \sqrt{2}, \\
&\log T(M,E) = \log T(K,N,E) + \log T(\U,E,N) + 
\log \tau(\cH(\U,K,N)).
\end{align*}
\end{thm}

It will become convenient below to rewrite the gluing formula of Theorem 
\ref{gluing-thm} in terms of renormalized Ray-Singer analytic torsion norms.
Consider the long exact sequences in cohomology
\begin{align*}
\cH(\U,K)&: \ldots H^p(\U,E,N) \xrightarrow{\A^*} H^p(M,E) \xrightarrow{\beta^*}
H^p(K,E) \xrightarrow{\delta^*} H^{p+1}(\U,E,N) \ldots, \\
\cH(\U,K,N)&: \ldots H^p(\U,E,N) \oplus H^p(K,E,N) \xrightarrow{\gamma^*} 
H^p(M,E) \xrightarrow{r^*} H^p(N,E) \\ & \qquad \qquad \qquad \qquad \qquad \quad  
\xrightarrow{\delta^*} H^{p+1}(\U,E,N) \oplus H^{p+1}(K,E,N)\ldots,
\end{align*} 
where $\delta^*$ denotes the respective connecting homomorphisms. 
These sequences induce isomorphisms on determinant lines in a 
canonical way, cf. \cite{Nic}
\begin{align*}
&\Phi: \det H^*(\U,E,N) \otimes \det H^*(K,E) \to \det H^*(M,E), \\
&\Phi': \det H^*(\U,E,N) \otimes \det H^*(K,E,N) \otimes \det H^*(N,E)
\to \det H^*(M,E).
\end{align*}
A careful combinatorial analysis, carried out e.g. in \cite[Theorem 7.12]{Ver2},
implies
\begin{align*}
&\log \|\cdot \|_{\det H^*(\U,E,N)} \otimes \|\cdot \|_{\det H^*(K,E)}
=\log \|\Phi (\cdot \otimes \cdot) \|_{\det H^*(M,E)} +  \log \tau(\cH(\U,K)), \\
&\log \|\cdot \|_{\det H^*(\U,E,N)} \otimes \|\cdot \|_{\det H^*(K,E,N)}
\otimes \|\cdot \|_{\det H^*(N,E)} =  \log \|\Phi' (\cdot \otimes \cdot \otimes \cdot) \|_{\det H^*(M,E)}
\\ & \qquad \qquad \qquad \qquad \qquad  \qquad \qquad \qquad \qquad \qquad \qquad + \log \tau(\cH(\U,K,N)), 
\end{align*}
where the Hilbert structures on the corresponding cohomologies are induced by the 
$L^2$-structure defined by the metrics $g$ and $h$. In combination with 
Theorem \ref{gluing-thm} we arrive at the following result\footnote{This is the statement 
of Theorem \ref{gluing-theorem}.}.

\begin{cor}\label{gluing-corr}
\begin{equation*}\begin{split}
\| \Phi (\cdot \otimes \cdot ) \|^{RS}_{(M,E)} &= 2^{-\frac{\chi(N,E)}{2}}
\| \cdot \|^{RS}_{(\U,E,N)} \otimes \| \cdot \|^{RS}_{(K,E)}, \\
\|\Phi' (\cdot \otimes \cdot \otimes \cdot )\|^{RS}_{(M,E)} &= 
\| \cdot \|^{RS}_{(\U,E,N)} \otimes \| \cdot \|^{RS}_{(K,E,N)} \otimes
\| \cdot \|_{\det H^*(N,E)}.
\end{split}\end{equation*}
\end{cor}

\subsection{Br\"uning-Ma metric anomaly on non-compact manifolds}
 
A particular consequence of the gluing formula in terms of analytic torsion norms, as 
obtained in Corollary \ref{gluing-corr}, is the application of the Br\"uning-Ma metric 
anomaly result \cite{BM} to the non-compact setting. \medskip

Assume that $M$ has non-empy smooth compact 
boundary $\partial M$, which is contained 
in the open interior of $K$. Consider a pair of Riemannian 
metrics $g_1$ and $g_2$ over $M$, which satisfy the conditions
of Assumptions \ref{assum1}, \ref{assum2}, \ref{assum3} and 
coincide over the infinite end $\U$. By Corollary \ref{metric-variation}
we assume without loss of generality that $g_1$ and $g_2$ are
both product over an open neighborhood of the separating 
hyper surface $N\times {1}$. Then by \eqref{BM-thm}
\begin{align*}
log \frac{\| \cdot \|^{RS}_{(M,E,g_1)}}{\| \cdot \|^{RS}_{(M,E,g_2)}}
&= \log \frac{\| \cdot \|^{RS}_{(\U,E,N,g_1)} \otimes \| \cdot \|^{RS}_{(K,E,
g_1)}}{\| \cdot \|^{RS}_{(\U,E,N,g_2)} \otimes \| \cdot \|^{RS}_{(K,E,
g_2)}} = \log \frac{\| \cdot \|^{RS}_{(K,E,
g_1)}}{\| \cdot \|^{RS}_{(K,E,
g_2)}} \\ &= \frac{\textup{rank}(E)}{2} \left[\int_{\partial
M}B(\nabla^{TM}_{g_2})-\int_{\partial M}B(\nabla^{TM}_{g_1})\right].
\end{align*}
We have thus extended the Br\"using-Ma metric anomaly result
to non-compact manifolds subject to Assumptions \ref{assum1}, 
\ref{assum2} and \ref{assum3}. 

\begin{prop}\label{bm}
Assume that $\partial M \neq \varnothing$ is contained 
in the open interior of $K$. Consider a pair of Riemannian 
metrics $g_1$ and $g_2$ over $M$, which satisfy the conditions
of Assumptions \ref{assum1}, \ref{assum2} and \ref{assum3}, and 
moreover coincide over $\U$. Then
\begin{align*}
log \frac{\| \cdot \|^{RS}_{(M,E,g_1)}}{\| \cdot \|^{RS}_{(M,E,g_2)}}
= \frac{\textup{rank}(E)}{2} \left[\int_{\partial
M}B(\nabla^{TM}_{g_2})-\int_{\partial M}B(\nabla^{TM}_{g_1})\right].
\end{align*}
\end{prop}

\subsection{Examples of manifold classes satisfying Assumptions \ref{assum1},
\ref{assum2}, \ref{assum3}}

Our discussion requires the notion of
polyhomogeneous distributions on a manifold with corners, introduced by Melrose cf.
\cite{Mel:COC}. Let $\mathfrak{W}$ be a manifold with corners, 
 modelled over open neighborhoods of $(\R^+)^k \times \R^\ell$, and embedded boundary faces 
$\{(H_i,\rho_i)\}_{i=1}^N$ where $\{\rho_i\}$ denote the corresponding boundary defining functions. 
We adopt the multi-index notation and for any multi-index $b= (b_1,\ldots, b_N)\in \C^N$ 
we write $\rho^b = \rho_1^{b_1} \ldots \rho_N^{b_N}$.  
Consider the space $\mathcal{V}(\mathfrak{W})$ of smooth $b$-vector fields on $\mathfrak{W}$ which 
by definiton are tangent to all boundary faces. 

\begin{defn}\label{phg}
We say that a distribution $w$ on $\mathfrak{W}$ is conormal
if $w\in \rho^b L^\infty(\mathfrak{W})$ for some $b\in \C^N$ and 
its regularity is stable under $b$-vector fields, i.e. $V_1 \ldots V_l w \in \rho^b L^\infty(\mathfrak{W})$
for all $V_j \in \mathcal{V}(\mathfrak{W})$ and for every $l \geq 0$. A collection
$E_i = \{(\gamma,p)\} \subset {\mathbb C} \times {\mathbb N}$ 
is said to be an index set if it satisfies the following hypotheses:
\begin{enumerate}
\item $\textup{Re}(\gamma)$ accumulates only at $+ \infty$,
\item if $(\gamma,p) \in E_i$, then $(\gamma+j,p') \in E_i$ for all $j \in \N_0$ and $0 \leq p' \leq p$,
\item for each $\gamma$ there exists $P_{\gamma}\in \N_0$ such 
that $(\gamma,p)\in E_i$ for every $0\leq p \leq P_\gamma < \infty$.
\end{enumerate}
We define an index family $E = (E_1, \ldots, E_N)$ to be an $N$-tuple of index sets,
a call a conormal distribution $w$ polyhomogeneous on $\mathfrak{W}$ 
with index family $E$, denoted $w\in \mathscr{A}_{\textup{phg}}^E(\mathfrak{W})$, 
if $w$ is conormal and expands near each $H_i$ as
$
w \sim \sum_{(\gamma,p) \in E_i} a_{\gamma,p} \rho_i^{\gamma} (\log \rho_i)^p, \ 
\textup{when} \ \rho_i\to 0,
$
where the coefficients $a_{\gamma,p}$ are required to be conormal distributions on $H_i$ and 
polyhomogeneous with index $E_j$ at any $H_i\cap H_j$. 
\end{defn}

We turn to (parabolic) blowups now. The notion of a blowup has been introduced by Melrose
cf. \cite{Mel:TAP} to capture the nonuniform behavior of Schwartz kernels of certain integral 
operators. Consider $\R^+ \times \R^+$ as a fundamental example of a manifold with corner
at the origin $0$. The blowup $[\R^+\times \R^+, 0]$ is defined as a disjoint union of 
$\R^+\times \R^+ \backslash {0}$ with the interior spherical normal bundle of $0$ in $\R^+\times \R^+$.
The blowup $[\R^+\times \R^+, 0]$ is by definition equipped with the unique minimal differential
structure, with respect to which smooth functions on $\R^+\times \R^+$
and polar coordinates around $0$ are smooth. The blowup is illustrated in Figure 
\ref{blowup}.

\begin{figure}[h]
\begin{center}
\begin{tikzpicture}

\draw (0,0) -- (0,2);
\draw (0,0) -- (2,0);
\draw (3,0.5) -- (3,2);
\draw (3.5,0) -- (5,0);
\draw (3,0.5) .. controls (3.3,0.5) and (3.5,0.3) .. (3.5,0);

\node at (-0.3,1.5) {$x$};
\node at (1.5,-0.3) {$\widetilde{x}$};

\node at (2.7,1.5) {$lf$};
\node at (4.5,-0.3) {$rf$};
\node at (3,0) {$ff$};

\end{tikzpicture}
\end{center}
\caption{$\R^+ \times \R^+$ and its blowup $[\R^+\times \R^+, 0]$}
\label{blowup}
\end{figure}

The front boundary face ff illustrates the interior spherical normal bundle of $0\in \R^+\times \R^+$.
In applications it is often convenient to work with locally defined projective 
coordinates instead of globally defined polar coordinates. 
Projective coordinates near the front face ff, near its lower corner and away from
the left boundary face ff, are given by
$$
\rho_{\textup{rf}} = \frac{x}{\wx}, \quad \rho_{\textup{ff}} = \wx,
$$
where $\rho_{\textup{rf}}$ is a boundary defining function of rf,
and $\rho_{\textup{ff}}$ a boundary defining function of ff. 
Projective coordinates near the front face ff, near its upper corner and away from
the right boundary face rf, are given by
$$
\rho_{\textup{lf}} = \frac{\wx}{x}, \quad \rho_{\textup{ff}} = x,
$$
where $\rho_{\textup{lf}}$ is a defining function of lf,
and $\rho_{\textup{ff}}$ a defining function of ff. \medskip

Similar construction makes sense in case of $X$ and $Y$ being manifolds
with boundary and $[X\times Y, \partial X \times \partial Y]$ is defined
as the blowup of $X\times Y$ at the highest codimension corner  
$\partial X \times \partial Y$. Locally, around each point $q\in \partial X \times \partial Y$
the blowup reduces to the model situation $[R^+\times \R^+, 0]$, where 
$q$ corresponds to the origin and both copies of $\R^+$ correspond
to the boundary defining functions of $X$ and $Y$. We refer for details to 
\cite{Mel:TAP} and \cite{Maz:ETD}.

\subsubsection{First example: Witt manifolds with cusps}\label{ex1}

Assume that $g\restriction \U = x^{-2}(dx^2+g^N)$ and 
the vector bundle $(E,\nabla, h)$ is induced by a unitary 
representation of the fundamental group of $M$. We begin with a proof
of Theorem \ref{reg-trace-thm}. \medskip

\emph{Proof of Theorem \ref{reg-trace-thm}}.
In his dissertation Vaillant derives the microlocal
description for the heat kernel of the square of the Dirac operator, and hence in particular
for the heat kernel $\cH$ of the full Hodge Laplacian on all degrees. From that microlocal 
description one concludes that at the diagonal the pointwise trace of $\cH$ is a polyhomogeneous distribution on 
the parabolic blowup $[M\times \R^+, \partial M\times \{0\}], \partial M=N$.
Here "parabolic" refers to the fact that on the second component
$\R^+$, the parameter space for time $t$, $\sqrt{t}$ is viewed as a smooth coordinate. 
Now, a general observation, e.g. Sher \cite[Theorem 13]{Sher2},
implies that the regularized integral of $\textup{tr}\, \cH$ exists and 
admits an asymptotic expansion as $t\to 0$. \medskip

However, such an asymptotic expansion is needed for the heat kernel $\cH_p$
in each degree $p$ individually, rather than for the whole heat kernel $\cH$.
A priori we cannot conclude an asymptotic expansion for the regularized integral of $\textup{tr}\, \cH_p$,
from the expansion of the regularized integral of $\textup{tr}\, \cH$, since there might be cancellations.
However, the heat kernel $\cH= \oplus_p \cH_p$ is a direct sum of individual heat kernels, where
each heat kernel $\cH_p$ takes values in endomorphisms of $\Lambda^p T^*M$. These vector bundles
are orthogonal to each other for different degrees $p$ and hence the asymptotics
of the individual heat kernel $\cH_p$ in each degree $p$ cannot cancel in its contribution to 
the microlocal description of the full $\cH= \oplus_p \cH_p$ before taking traces.
We conclude that the microlocal description of Vaillant for the full heat kernel $\cH$
holds for the individual heat kernels $\cH_p$ as well. Thus the regularized integral of 
$\textup{tr}\, \cH_p$ exists for any fixed degree $p$ and 
admits an asymptotic expansion as $t\to 0$. This proves Theorem 
\ref{reg-trace-thm} (i) and (ii). \medskip

It remains to prove Theorem \ref{reg-trace-thm} (iii). 
Assume the Witt condition $H^{n/2}(N,E)=0$. By Remark \ref{LP}, the continuous spectrum 
of the Hodge Laplacian $\Delta_{p,\U}$ on $\U$ comes from the
Laplacians associated to the harmonic sub-complexes in the decomposition
of the de Rham complex in \S \ref{decomposition-section}. The continuous
spectrum of these \emph{harmonic} Laplacians in \eqref{harmonic},
resulting from the cohomology $H^p(N,E)$ of the cross section, with 
either Dirichlet or Neumann boundary conditions, is given by  $[(n/2 - p)^2, \infty)$.
This can be read off directly after the transformation in the proof of Proposition
\ref{harmonic-constribution2}. Hence the Witt condition $H^{n/2}(N,E) = \{0\}$
yields a spectral gap at zero, i.e. there exists $\varepsilon >0$ sufficiently small
such that $(0,\varepsilon) \cap \spec \Delta_* = \varnothing$. \medskip

One may now consider the Hodge Laplacian $\Delta_{p,\textup{cut}}$ 
on $M^{\textup{cut}} := K \sqcup \U$
with relative boundary conditions at the boundary $N \sqcup N$. 
By arguments similar to Proposition \ref{diff}, the difference
of heat kernels of $\Delta_{p,\textup{cut}}$ and the Hodge Laplacian
$\Delta_p$ on $M$ is trace class and hence by \cite[Lemma 2.2]{Mue-rel}
$\Delta_p$ admits a spectral gap at zero as well and moreover there 
exist constants $c,C>0$ such that for $t>0$ large enough
\begin{align}\label{Mueller}
\mid \textup{Tr} 
\left(e^{-t\Delta_p} - e^{-t\Delta_{p,\U}}\right) + \dim \ker \Delta_{p,\U}
- \dim \ker \Delta_p \mid \, 
< C e^{-ct}.
\end{align}
The direct sum component of $\Delta_{p,\U}$
corresponding to the non-harmonic sub-complexes, cf.
\S \ref{decomposition-section}, has discrete spectrum, trivial kernel 
and the corresponding heat kernel is trace class with exponentially 
decaying heat trace. \medskip

The direct sum component of $\Delta_{p,\U}$ that corresponds to harmonic sub-complexes 
has been studied in Proposition \ref{harmonic-constribution2} and \ref{harmonic-contribution3},
and is given by $\Delta_{H,p}\oplus \Delta_{H,p,\textup{Neu}}$. While
Propositions \ref{harmonic-constribution2} and \ref{harmonic-contribution3}
compute their regularized resolvent traces explicitly, their 
regularized heat traces may be computed in a similar manner.
The heat kernel of $\Delta_{H,p}$ is given under the transformation 
$L^2([R,\infty), dx) \to L^2([R,\infty), x^{-1}dx), f \mapsto x^{1/2}f$ 
and a change of variables $x=e^r$ by 
$$
\mathscr{H}(t,r,r') = \frac{e^{-t \left(\frac{n}{2}-p\right)^2}}{\sqrt{4\pi t}}
\left(e^{-\frac{(r-r')^2}{4t}} - e^{-\frac{(r+r'-2\log R)^2}{4t}}\right).
$$
We compute
\begin{align*}
\int_{\log R}^{\log R'} \mathscr{H}(t,r,r) dr
&= \frac{e^{-t \left(\frac{n}{2}-p\right)^2}}{\sqrt{4\pi t}} 
\int_0^{\log R'/R} \left(1-e^{-y^2/t}\right) dy
\\ &= e^{-t \left(\frac{n}{2}-p\right)^2} 
\left(\frac{\log R'}{\sqrt{4\pi t}} - \frac{\log R}{\sqrt{4\pi t}} - \frac{1}{4}\right) + o(1), 
\quad R'\to \infty.
\end{align*}
The regularized trace $\textup{Tr}_{r} \mathscr{H}$ is defined 
as the constant term in the expansion above and hence we find for $p\neq n/2$
\begin{align}\label{reg1}
\textup{Tr}_{r} \mathscr{H} = e^{-t \left(\frac{n}{2}-p\right)^2} 
\left(- \frac{\log R}{\sqrt{4\pi t}} - \frac{1}{4}\right) = O(e^{-ct}), t \to \infty,
\end{align}
for some constant $c>0$. Similarly, one finds by explicit computations 
\begin{align}\label{reg2}
\textup{Tr}_{r} (\cH_{\textup{Neu}} - \cH) -\dim \ker \Delta_{p,\U}
= O(e^{-ct}), t \to \infty.
\end{align}
Consequently, in view of \eqref{Mueller}, \eqref{reg1} and \eqref{reg2}, the
regularized trace $\textup{Tr}_r \, e^{-t\Delta_p}$ exists and there 
exist constants $c',C'>0$ such that for $t>0$ sufficiently large
\begin{align*}
\mid \textup{Tr}_r \, e^{-t\Delta_p} - \dim \ker \Delta_p  \mid \, 
< C' e^{-c't}.
\end{align*}
\begin{flushright}
$\Box$
\end{flushright} 

By Theorem \ref{reg-trace-thm} and the spectral gap observation, 
Assumptions \ref{assum1} and \ref{assum3} are satisfied.
By a similar argument, the family $\Delta_{p,\theta}$ in the notation of Assumption \ref{assum2} admits a spectral 
gap at zero. As in \cite{Les}, its kernel is independent of $\theta$ and hence
as before, by \cite[Lemma 2.2]{Mue-rel}, the renormalized heat trace of $\Delta_{p,\theta}$ admits a large 
time asymptotic expansion that is differentiable in the parameter. The full statement 
of Assumption \ref{assum2} (i) now follows from e.g.
\cite[Proposition 3.3.1]{Lesch-habil}. \medskip

Assumption \ref{assum2} (ii) is a direct consequence of the microlocal
asymptotic description of the heat kernel on manifolds with cusps by Vaillant 
\cite{Vai}. More precisely, consider the heat kernel $\cH_{p,\theta}(t,q,q')$ of 
$\Delta_{p,\theta}$. If $q,q'\in \U=[1,\infty) \times N$, we write $q=(x,y), q'=(x',y')$ 
with $x,x'\in [1,\infty)$. By composition formulae of Vaillant $\cH_{p,\theta}$ lies again in his
heat calculus, which in particular asserts that the heat kernel is polyhomogeneous
at the origin of $\R^+_{\sqrt{t}}\times \R^+_{1/x'}$, vanishing to infinite 
order as $t\to 0$ and as $x'\to \infty$, uniformly in other variables as long 
as $1/x \geq \delta>0$. This yields the estimates of Assumption \ref{assum2} (ii).

\subsubsection{Second example: Scattering manifolds}

$g\restriction \U = dx^2+x^2g^N$ and 
the vector bundle $(E,\nabla, h)$ is induced by a unitary 
representation of the fundamental group of $M$, such 
that $H^{*}(N,E) = 0$. Assumption \ref{assum3} is trivially satisfied,
note however that in contrast to the previous example there is no
spectral gap at zero here.
\medskip

The asymptotic properties of the heat kernel have been studied in this
setting by Guillarmou and Sher \cite{Sher}. In particular their construction asserts that 
at the diagonal the pointwise trace of $\Delta_p$ is a polyhomogeneous distribution on 
the parabolic blowup of $[0,\infty]_t \times M$ illustrated in Figure \ref{blowup2}.

\begin{figure}[h]
\begin{center}
\begin{tikzpicture}

\draw (0,0) -- (0,2);
\draw (0,0) -- (4,0);
\draw (4,0) -- (4,2);

\draw (6,0) -- (6,2);
\draw (6,0) -- (9.5,0);
\draw (10,0.5) -- (10,2);
\draw (10,0.5) .. controls (9.6,0.5) and (9.5,0.4) .. (9.5,0);

\node at (-0.4,1) {$1/x$};
\node at (-0.4,0.5) {$\downarrow$};
\node at (-0.4,0) {$0$};

\node at (4.4,1) {$1/x$};
\node at (4.4,0.5) {$\downarrow$};
\node at (4.4,0) {$0$};

\node at (0.5,-0.3) {$0 \leftarrow t$};
\node at (3.4,-0.3) {$t^{-1} \!\!\rightarrow 0$};

\end{tikzpicture}
\end{center}
\caption{$[0,\infty]_{t}\times M$ and its blowup $[[0,\infty]_{t}\times M, \{t=\infty\}\times \partial M]$.}
\label{blowup2}
\end{figure}

Now, the general observation of Sher \cite[Theorem 13]{Sher2}
implies that the regularized integral of $\textup{tr}\, \cH_p$ exists and 
admits an asymptotic expansion as $t\to 0$ and as $t\to \infty$.
Consequently, Assumption \ref{assum1} is satisfied.\medskip

Consider the family $\Delta_{p,\theta}$ in the notation of Assumption \ref{assum2},
and the corresponding heat kernels $\cH_{p,\theta}$. Since $\Delta_{p,\theta}$ differs from 
the Hodge Laplacian $\Delta_{p}$ only on a compact subset, by the composition formulae of 
\cite{Sher}, the kernels $\cH_{p,\theta}$ lie again in the calculus and Assumption \ref{assum2} (i)
follows. Assumption \ref{assum2} (ii) is satisfied by exactly the same argument as in the 
previous example, since the heat calculi of Vaillant and Guillarmou-Sher differ only in their
behaviour at the corners of highest codimension. 

\section{Cheeger-M\"uller theorem for 
Witt-manifolds with cusps}\label{CM-section}

Consider an odd-dimensional 
oriented Riemannian manifold $(M,g)$ with $$M=K\cup_{\partial K} \U_R,$$ 
where $K$ is a compact manifold with smooth boundary $\partial K = N \times \{R\}$, and 
$\U=N \times [R,\infty)$. Consider a flat Hermitian vector bundle $(E,\nabla, h)$
induced by a unitary representation of the fundamental group $\pi_1(M)$. 
Assume that $M$ satisfies the Witt condition, which is a condition  
on the middle degree cohomology $H^{n/2}(N,E)={0}, n=\dim N$. 
Let $g^N$ be a Riemannian metric on the closed manifold $N$. 
We consider Riemannian metrics $g$ such that

\begin{align*}
g\restriction \U_R = \frac{dx^2+g^N}{x^2}, \ x\in [R,\infty).
\end{align*}

Consider also a Riemannian metric $g'$ on $M$ that is product 
in an open neighborhood of $N\times \{R\}$ and coincides with $g$
on $\U_{R+1}=N\times [R+1,\infty)$.\medskip

We continue in the notation of \S \ref{CM}. We write 
$M^*$ for the one-point compactification of $M$ at infinity.
Recall from above that the intersection cohomology of Goresky-MacPherson for the Witt space $M^*$
with values in $E$ is denoted by $IH^*(M^*,E)$, and $H^*(M,E)$ refers to the 
$L^2$-cohomology of the cusp manifold $(M,g)$ with values in $E$. 
Due to the Witt condition $H^{n/2}(N,E)=0$, both cohomologies coincide, 
compare for instance the Hodge cohomology theory 
by Hausel, Hunsicker and Mazzeo \cite{HHM}.
\medskip

The flat Hermitian vector bundle $(E,\nabla, h)$ need not arise from a
unitary representation of the fundamental group $\pi_1(M^*)$ and hence
the definition of intersection R-torsion $\|\cdot \|^R_{(M^*,E)}$ as provided
by Dar \cite{Dar}, does not apply in this setting. We use the definition as 
provided by Albin, Rochon and Sher in \cite{ARS}. Their definition of intersection 
R-torsion in fact applies to a more general class of flat vector bundles $E$ arising from 
unimodular representations $\rho:\pi_1(M) \cong \pi_1(K) \to GL(d, \R)$. It uses a specific cochain complex $R^*(X, \rho)$ 
associated to $\rho$ and a triangulation $X$ of the stratified space $M^*$. If $Y\subset X$ is a sub-complex 
triangulating $K$, then the standard cochain complex $C^*(Y,\rho) = C^*(\widetilde{Y}) \otimes_\rho \R^d$
is a sub-complex of $R^*(X, \rho)$, where $\widetilde{Y}$ denotes the universal cover of $Y$. 
This yields a short exact sequence \cite[(8.15)]{ARS} of complexes 
and the corresponding splitting formula \cite[(8.16)]{ARS} by a formula of Milnor.
We refer to \cite{ARS} for more details on the definition of $\|\cdot \|^R_{(M^*,E)}$
which we use henceforth.
\medskip

The intersection R-torsion $\|\cdot \|^R_{(M^*,E)}$ 
defines a norm on the determinant line of the middle perversity 
intersection cohomology $\det IH^*(M^*,E) \cong \det H^*(M,E)$.
The combinatorial gluing formula by Milnor \cite{Mil}, see also Vishik 
\cite[Proposition 1.3 and 1.4]{Vi}, yields a combinatorial 
analogue of Corollary \ref{gluing-corr}
\begin{equation}\label{gluing-comb}
\| \Phi (\cdot \otimes \cdot ) \|^{R}_{(M^*,E)} = 
\| \cdot \|^{R}_{(\U_R^*,E,N)} \otimes \| \cdot \|^{R}_{(K,E)},
\end{equation}
where $\| \cdot \|^{R}_{(\U_R^*,E,N)}$ denotes the
intersection R-torsion on relative cohomology 
$\det IH^*(\U_R^*,E, N) \cong \det H^*(\U_R,E,N)$, and 
$\| \cdot \|^{R}_{(K,E)}$ coincides with the classical 
Reidemeister torsion as a norm on the determinant line 
$\det H^*(K,E)$. \medskip

The celebrated theorem by Cheeger \cite{Che} and M\"uller \cite{Mue-AT}
has been extended to manifolds with boundary by L\"uck \cite{Lu}
and Vishik \cite{Vi}. By their results we find

\begin{align}\label{cm}
\log \frac{\| \cdot \|^{RS}_{(K,E,g')}}{\| \cdot \|^{R}_{(K,E)}}
= \frac{\chi(N,E)}{4} \log 2.
\end{align}
Combining invariance of the torsion norms under compactly 
supported metric variations with Corollary \ref{gluing-corr}, \eqref{gluing-comb} and 
\eqref{cm}, we compute
\begin{align}\label{cm2}
\log \frac{\| \cdot \|^{RS}_{(M,E,g)}}{\| \cdot \|^{R}_{(M^*,E)}}
=\log \frac{\| \cdot \|^{RS}_{(M,E,g')}}{\| \cdot \|^{R}_{(M^*,E)}}
=\log \frac{\| \cdot \|^{RS}_{(\U_R,E,N,g')}}{\| \cdot \|^{R}_{(\U_R^*,E,N)}}
- \frac{\chi(N,E)}{4} \log 2.
\end{align}
By Proposition \ref{bm}
\begin{align*}
\log \frac{\| \cdot \|^{RS}_{(\U_R,E,N,g')}}{\| \cdot \|^{RS}_{(\U_R,E,N,g)}}
= \frac{\textup{rank}(E)}{2} \int_{\partial \U_R}B(\nabla^{T\U_R}_{g}).
\end{align*}
Consider an orthonormal basis $h_g$ of $IH^*(\U_R^*,E,N)$, with the Hilbert 
structure on cohomology induced by the isomorphism
$IH^*(\U_R^*,E,N) \cong H^*(\U_R,E,N)$, where the right hand side 
is equipped with the Hilbert structure induced by the Riemannian metric $g$. Then
we arrive at the following relation
\begin{equation}\label{fast}\begin{split}
\log \frac{\| \cdot \|^{RS}_{(M,E,g)}}{\| \cdot \|^{R}_{(M^*,E,g)}}
= \log \frac{\| \cdot \|^{RS}_{(\U_R,E,N,g)}}{\| \cdot \|^{R}_{(\U_R^*,E,N,g)}}
&- \frac{\chi(N,E) }{4}\log 2 +  
\frac{\textup{rank}(E)}{2} \int_{\partial \U_R}B(\nabla^{T\U_R}_{g})
\\ = \log \frac{T(\U_R,E,N,g)}{\tau(\U_R^*,E,N,h_g)}
&- \frac{\chi(N,E) }{4}\log 2 +
\frac{\textup{rank}(E)}{2} \int_{\partial \U_R}B(\nabla^{T\U_R}_{g}),
\end{split}\end{equation}
where $\tau(\U_R^*,E,N,h_g)$ is the scalar intersection R-torsion, introduced 
by Dar \cite{Dar} and defined with respect to the preferred basis $h_g$. 
Let an orthonormal basis of the absolute intersection cohomology 
$IH^*(\U_R^*,E) \cong H^*(\U_R,E)$, with the Hilbert structure induced by
the Riemannian metric $g$ as above, be denoted by $h_g$ again. 
Then by Poincare duality and the Witt 
condition (cf. also Hartmann-Spreafico \cite{HS-R-torsion})
$$
\tau(\U_R^*,E,N,h_g)= \tau(\U_R^*,E,h_g).
$$
Consider an orthonormal basis $h_N$ of $H^*(N,E)$, with the Hilbert 
structure on cohomology induced by the Riemannian metric $g^N$.
We obtain a preferred basis on $IH^*(\U_R^*,E)$ by extending the 
basis to $\U_R$, constant in $x-$direction. We denote such a basis 
by $h_N$ again. We may compare intersection R-torsions on $\U_R^*$
defined with respect to the preferred bases $h_g$ and $h_N$. 
Using a computation similar to \eqref{det-norms} we compute
\begin{equation}\label{intersection-R} \begin{split}
\log \tau(\U_R^*,E,h_g) &= \log \tau(\U_R^*,E,h_N) \\&+ \sum_{p<n/2} 
\frac{(-1)^p}{2} \dim H^p(N,E) \left((2p-n) \log R -\log(n-2p)\right)
\\&= \log \tau(\U_R^*,E,h_N) + \sum_{p=0}^n 
\frac{(-1)^{p+1}}{2} \dim H^p(N,E) \left|\frac{n}{2}-p\right| \log R
\\&+ \sum_{p=0}^n 
\frac{(-1)^{p+1}}{4} \dim H^p(N,E) \log \left(2\left|\frac{n}{2}-p\right| \right).
\end{split}\end{equation}
Note that $\tau(\U_R^*,E,h_N)\equiv \tau(\U^*,E,h_N)$ and that
$T(\U_R^*,E,N,g)$ has been computed explicitly in our first main 
Theorem \ref{scalar-torsion-thm}. Hence plugging the formula 
of Theorem \ref{scalar-torsion-thm} as well as \eqref{intersection-R}
into \eqref{fast} we arrive at the following final result.

\begin{thm} Let $(M,g)$ be a complete Witt manifold with 
a cusp end\footnote{Clearly, by the gluing property of analytic torsion a similar
result holds for finitely many cusp ends.} and without boundary. Then
\begin{equation*}\begin{split}
\log \frac{\| \cdot \|^{RS}_{(M,E,g)}}{\| \cdot \|^{R}_{(M^*,E,g)}}
&=  -\log \tau(\U^*,E,h_N) \\
&+\sum_{p\neq n/2}  \frac{(-1)^{p+1}}{2}\dim H^p(N,E) 
\left|\frac{n}{2}-p\right| \log \left( 2\left|\frac{n}{2}-p\right| \right).
\end{split}\end{equation*}
\end{thm}

\def\cprime{$'$} 
\providecommand{\bysame}{\leavevmode\hboxto3em{\hrulefill}\thinspace}
\providecommand{\MR}{\relax\ifhmode\unskip\space\fi MR } 
\providecommand{\MRhref}[2]{%
\href{http://www.ams.org/mathscinet-getitem?mr=#1}{#2} }
\providecommand{\href}[2]{#2}

\end{document}